\documentclass[a4paper]{article}
\usepackage{graphicx} 
\usepackage{geometry}
\usepackage{amsmath, amssymb, mathtools}
\usepackage{mathrsfs, amsfonts}
\usepackage{amsthm}
\usepackage{bm}
\usepackage{titlesec}
\usepackage{hyperref}
\usepackage{enumerate}
\usepackage[dvipsnames]{xcolor} 
\usepackage{tikz-cd}
\usepackage{pgfplots}
\usepackage{tikz}
\usepackage{eucal}
\usepackage{subfigure}
\usepackage{wasysym}
\usepackage{cite}
\usepackage[font=sc,justification=centering]{caption}
\usetikzlibrary{patterns}
\makeatletter
\tikzset{
        hatch distance/.store in=\hatchdistance,
        hatch distance=5pt,
        hatch thickness/.store in=\hatchthickness,
        hatch thickness=5pt,
        every point/.style = {circle, inner sep={.75\pgflinewidth}, opacity=1, draw, solid, fill=white},
        point/.style={insert path={node[every point, #1]{}}}, point/.default={}
        }
\pgfdeclarepatternformonly[\hatchdistance,\hatchthickness]{north east hatch}
    {\pgfqpoint{-1pt}{-1pt}}
    {\pgfqpoint{\hatchdistance}{\hatchdistance}}
    {\pgfpoint{\hatchdistance-1pt}{\hatchdistance-1pt}}%
    {
        \pgfsetcolor{\tikz@pattern@color}
        \pgfsetlinewidth{\hatchthickness}
        \pgfpathmoveto{\pgfqpoint{0pt}{0pt}}
        \pgfpathlineto{\pgfqpoint{\hatchdistance}{\hatchdistance}}
        \pgfusepath{stroke}
    }
\makeatother

\usepackage{titlesec}
\titleformat{\section}[block]{\sffamily\scshape\filcenter\Large}{\thesection.\ }{0pt}{}
\renewcommand{\thesubsection}{(\thesection\alph{subsection})}
\renewcommand{\thesubsubsection}{(\thesection\alph{subsection}-\arabic{subsubsection})}
\titleformat{\subsection}[runin]{\normalfont\bfseries}{\indent\thesubsection\ }{0pt}{}
\titleformat{\subsubsection}[runin]{\normalfont\bfseries}{\indent\thesubsubsection}{0pt}{}

\DeclareTextFontCommand{\emph}{\bfseries}

\DeclareMathOperator{\Nbd}{\mathrm{Nbd}}
\DeclareMathOperator*{\hocolim}{\mathrm{hocolim}}
\DeclareMathOperator{\Skel}{\mathrm{Skel}}
\DeclareMathOperator{\Id}{\mathrm{Id}}
\DeclareMathOperator{\Mod}{\mathrm{Mod}}
\DeclareMathOperator{\rest}{\mathrm{rest}}
\DeclareMathOperator{\Hom}{\mathrm{Hom}}
\DeclareMathOperator{\im}{\mathrm{im}}
\DeclareMathOperator{\Res}{\mathrm{Res}}

\newtheoremstyle{them}
{4pt}
{4pt}
{\normalfont }
{}
{\scshape }
{.}
{.5em}
{}

{\swapnumbers
\theoremstyle{them} 
\newtheorem{definition}{Definition}[section]
\newtheorem{theorem}[definition]{Theorem}
\newtheorem{lemma}[definition]{Lemma}
\newtheorem{remark}[definition]{Remark}
\newtheorem{example}[definition]{Example}
\newtheorem{proposition}[definition]{Proposition}
\newtheorem{corollary}[definition]{Corollary}

}

\newcommand{\C}{\mathbb{C}}
\newcommand{\R}{\mathbb{R}}
\DeclareMathOperator{\rRe}{\mathrm{Re}}
\DeclareMathOperator{\rIm}{\mathrm{Im}}

\newcommand{\Fix}{\operatorname{Fix}}

\title{Fukaya Categories of Hyperplane Arrangements}
\author{Sukjoo Lee, Yin Li, Si-Yang Liu and Cheuk Yu Mak}
\date{\today}

\begin{document}
    \maketitle
    \begin{abstract}
    To a simple polarized hyperplane arrangement (not necessarily cyclic) $\mathbb{V}$, one can associate a stopped Liouville manifold (equivalently, a Liouville sector) $\left(M(\mathbb{V}),\xi\right)$, where $M(\mathbb{V})$ is the complement of finitely many hyperplanes in $\mathbb{C}^d$, obtained as the complexifications of the real hyperplanes in $\mathbb{V}$. The Liouville structure on $M(\mathbb{V})$ comes from a very affine embedding, and the stop $\xi$ is determined by the polarization. In this article, we study the symplectic topology of $\left(M(\mathbb{V}),\xi\right)$. In particular, we prove that their partially wrapped Fukaya categories are generated by Lagrangian submanifolds associated to the bounded and feasible chambers of $\mathbb{V}$. A computation of the Fukaya $A_\infty$-algebra of these Lagrangians then enables us to identify the partially wrapped Fukaya categories $\mathcal{W}(M(\mathbb{V}),\xi)$ with the $\mathbb{G}_m^d$-equivariant hypertoric convolution algebras $\widetilde{B}(\mathbb{V})$ associated to $\mathbb{V}$. This confirms a conjecture of Lauda-Licata-Manion \cite{LLM2020} and provides evidence for the general conjecture of Lekili-Segal \cite{LS} on the equivariant Fukaya categories of symplectic manifolds with Hamiltonian torus actions.
    \end{abstract}

	
    \section{Introduction}

\subsection{Main Results.}\label{section:main}
Very affine varieties are local models of divisor complements in proper Calabi-Yau varieties, so they provide important examples for the study of mirror symmetry and symplectic topology, and are studied extensively in the literature \cite{Ab-Au21,abouzaid2013homological,auroux2017speculations,gammage2022functorial,Gammage-Shende,lekili2020homological}. In this paper, we consider a particular class of very affine varieties, defined via the combinatorial data of a polarized hyperplane arrangement.

A polarized hyperplane arrangement in $\mathbb{R}^d$ consists of a triple $\mathbb{V}=(V,\eta,\xi)$, where the $d$-dimensional subspace $V\subset\mathbb{R}^n$ and the vector $\eta\in\mathbb{R}^n/V$ determine $n$ hyperplanes $H_{\mathbb{R},i}\subset\mathbb{R}^d$, where $1\leq i\leq n$, and $\xi\in(\mathbb{R}^n)^\ast/V^\perp$ is the polarization (see Definition \ref{def:pol.hyp.arr}). Assume that $\mathbb{V}$ is simple, meaning that the non-empty intersection of any $k$ hyperplanes is codimension $k$, we can define a Weinstein sector $(M(\mathbb{V}), \xi)$ as follows. Let $M(\mathbb{V})$ be the complement
\begin{equation*}
M(\mathbb{V}):=\mathbb{C}^d\setminus\bigcup_{i=1}^nH_i,
\end{equation*}
where $H_i\subset\mathbb{C}^d$ is the complexification of the hyperplane $H_{\mathbb{R},i}\subset\mathbb{R}^d$ in the arrangement $\mathbb{V}$. It follows from the definition that we have an affine algebraic embedding $M(\mathbb{V})\hookrightarrow(\mathbb{C}^\ast)^n$, equipping $M(\mathbb{V})$ with the structure of a very affine variety. Symplectically, $M(\mathbb{V})$ is a Weinstein manifold whose Weinstein structure is inherited from the standard Stein structure on $(\mathbb{C}^\ast)^n$. Since the polarization $\xi$ can be interpreted as a superpotential $\xi:M(\mathbb{V})\rightarrow\mathbb{C}$, it defines a stop (i.e. a Weinstein hypersurface) in the ideal boundary $\partial_\infty M(\mathbb{V})$, which we will still denote by $\xi$ by abuse of notations. Our goal in this paper is to identify the partially wrapped Fukaya category $\mathcal{W}(M(\mathbb{V}), \xi)$ with a convolution algebra associated to the combinatorial data $\mathbb{V}$. In particular, we will find a generating set of objects in $\mathcal{W}(M(\mathbb{V}), \xi)$ and show that their endomorphism $A_\infty$-algebra in $\mathcal{W}(M(\mathbb{V}), \xi)$ is formal. 

To describe the generating Lagrangians in $\mathcal{W}(M(\mathbb{V}), \xi)$, note that the hyperplanes in the arrangement $\mathbb{V}$ divide $\mathbb{R}^d$ into a finite number of chambers. Since these hyperplanes come with their co-orientations, each of these chamber in $\mathbb{R}^d$ corresponds to a sign sequence $\alpha \in \{+,-\}^n$, which we call \textbf{feasible}. Among the set of feasible sign sequences $\mathscr{F}(\mathbb{V})$, those for which the linear functional $\xi:\Delta_\alpha\rightarrow\mathbb{R}$ is bounded above are called \textbf{bounded}, where $\Delta_\alpha\subset\mathbb{R}^d$ is the (closed) chamber corresponding to the sign sequence $\alpha\in\mathscr{F}(\mathbb{V})$. Denote by $\mathscr{P}(\mathbb{V})\subset\mathscr{F}(\mathbb{V})$ the set bounded-feasible sign sequences (or equivalently, chambers). See Figure \ref{fig:parrr} for a polarized hyperplane arrangement consisting of three lines in $\mathbb{R}^2$. To each $\alpha \in \mathscr{P}(\mathbb{V})$, we can associate a Lagrangian submanifold $L_\alpha\subset M(\mathbb{V})$, which is the connected component (open chamber) $\Delta_\alpha^\circ\subset\Delta_\alpha$ in the real locus
\begin{equation*}
M(\mathbb{V})\cap\mathbb{R}^d\subset\mathbb{C}^d.
\end{equation*}
This is a non-compact, exact Lagrangian submanifold which gives an object of the Fukaya category $\mathcal{W}(M(\mathbb{V}),\xi)$.

\begin{paragraph}{Conventions.}
In this paper, we will be working over the ring of integers $\mathbb{Z}$, which is where the $\mathbb{G}_m$-equivariant hypertoric convolution algebras are defined \cite{llm1}, although these algebras are originally introduced over $\mathbb{C}$ in \cite{BLPPW}. On the other hand, the work of Ganatra-Pardon-Shende \cite{ganatra2018sectorial}, which our work heavily relies on, is also done over $\mathbb{Z}$.
\end{paragraph}

We will now state the main results of this paper.

\begin{theorem}[Generation, Theorem \ref{thm:generation}]\label{theorem:generation1}
The partially wrapped Fukaya category $\mathcal{W}(M(\mathbb{V}),\xi)$ is generated by the Lagrangian submanifolds $\{L_\alpha\}_{\alpha\in\mathscr{P}(\mathbb{V})}$.
\end{theorem}

The key ingredient of the proof is the Sectoral Descent established in \cite{ganatra2018sectorial} (cf. Theorem \ref{Prelim:sectorial-descent}). we will present the main idea of the proof using the explicit example given in Figure \ref{fig:parrr}.

Theorem \ref{theorem:generation1} enables us to identify the partially wrapped Fukaya category $\mathcal{W}(M(\mathbb{V}), \xi)$ with the endomorphism $A_\infty$-algebra of its objects $\{L_\alpha\}_{\alpha\in\mathscr{P}(\mathbb{V})}$. Consider the Fukaya $A_\infty$-algebra
\begin{equation*}
\widetilde{\mathcal{B}}_\mathbb{V}:=\bigoplus_{\alpha,\beta\in\mathscr{P}(\mathbb{V})}\mathit{CW}^\ast(L_\alpha,L_\beta),
\end{equation*} 
where $\mathit{CW}^\ast(-,-)$ denotes the partially wrapped Floer cochain complex with respect to the stop $\xi$. We will show that $\widetilde{\mathcal{B}}_{\mathbb{V}}$ is quasi-isomorphic to a convolution algebra $\widetilde{B}(\mathbb{V})$ introduced in \cite{BLPPW}, which is the universal flat graded deformation of the hypertoric convolution algebra $B(\mathbb{V})$ \cite{BLPW2010}. More backgrounds and motivations about why this should be the case will be given in Section \ref{sec:Intro-motiv}.


\begin{theorem}[Formality, Theorem \ref{thm:formality}]\label{theorem:conv1}
There is a quasi-isomorphism between $A_\infty$-algebras over $\mathbb{Z}$:
\begin{equation*}
\widetilde{\mathcal{B}}_{\mathbb{V}}\cong\widetilde{B}(\mathbb{V}).
\end{equation*}
 In particular, the Fukaya $A_\infty$-algebra $\widetilde{\mathcal{B}}_{\mathbb{V}}$ is formal.
\end{theorem}


This proves \cite[Conjecture 1.2]{LLM2020} in the general case when the hyperplane arrangement $\mathbb{V}$ is not necessarily \textit{cyclic}. Refer to \cite{LLM2020} for the precise definition of a cyclic hyperplane arrangement. In particular, such a hyperplane arrangement cannot contain two parallel hyperplanes. When $\mathbb{V}$ is cyclic, the Weinstein manifold $M(\mathbb{V})$ has a convenient interpretation as a symmetric product of punctured spheres $\mathbb{C}\setminus\{z_1,\cdots,z_n\}$, and Theorem \ref{theorem:conv1} follows essentially from the computation of the (fully) wrapped Fukaya category $\mathcal{W}(M(\mathbb{V}))$ by Lekili-Polishchuk \cite{lekili2020homological}.

In \cite{BLPW2010, LLM2020}, some functorial properties of the category of perfect $\widetilde{B}(\mathbb{V})$-modules with respect to the deletion and restriction of hyperplane arrangements are discussed. 
In Section \ref{functoriality-section}, we discuss these functors and describe some other natural geometric functors between Fukaya categories.

\begin{figure}
	\centering
	\begin{tikzpicture}
	\filldraw [draw=black,color={black!20},opacity=0.5] (4,2)--(0,2)--(4,-2);
	\filldraw [draw=black,color={black!20},opacity=0.5] (4,-2)--(2,0)--(2,-2);

	\draw [thick](-2,2) to (4,2);
	\draw [thick] (2,-2) to (2,4);
	\draw [thick] (-2,4) to (4,-2);
	\draw [red, thick] (-2,3) to (2,5);
	\draw [red, thick ][->] (0.4,4.2) to (0.1,4.8);
	\draw [red] (2.2,5) node {$\xi$};
	\draw (3,2.5) node {$---$};
	\draw (0.8,2.5) node {$-+-$};
	\draw (-1.7,2.5) node {$-++$};
	\draw (1.2,1.5) node {$++-$};
	\draw (-.5,0) node {$+++$};
	\draw (3,1) node {$+--$};
	\draw (2.7,-1.5) node {$+-+$};
	\end{tikzpicture}
	\caption{A polarized hyperplane arrangement in $\mathbb{R}^2$,\\ with the bounded-feasible chambers shaded \label{fig:parrr}}
\end{figure}
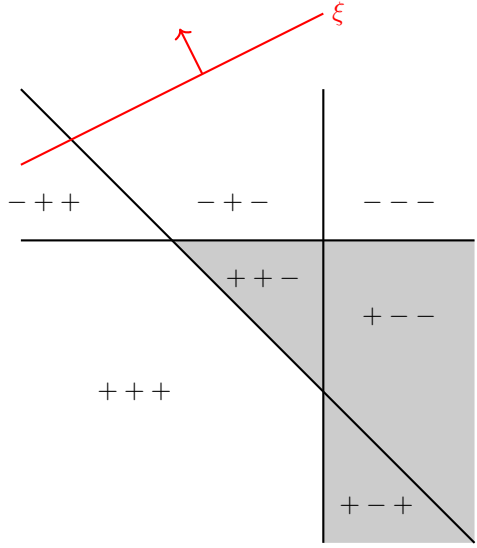

\subsection{Motivations.}\label{sec:Intro-motiv}
To a polarized hyperplane arrangement $\mathbb{V}$, one can associate another geometric object, namely a non-compact hyperk\"ahler variety $\mathfrak{M}_\mathbb{V}$, called a hypertoric variety \cite{BD2000}. Under certain assumptions on $\mathbb{V}$, the variety $\mathfrak{M}_\mathbb{V}$ is smooth. After a suitable hyperk\"ahler twist, one can equip $\mathfrak{M}_\mathbb{V}$ with the structure of a smooth affine variety. The induced Liouville structure on $\mathfrak{M}_\mathbb{V}$, together with a Weinstein hypersurface $\mathfrak{f}_\mathbb{V}$ in the ideal boundary $\partial_\infty\mathfrak{M}_\mathbb{V}$ determined by the polarization $\xi$, provides a stopped Liouville manifold $(\mathfrak{M}(\mathbb{V}), \mathfrak{f}_\mathbb{V})$. Moreover, there is the complex moment map \begin{equation*}
\bar{\mu}_\mathbb{C}:\mathfrak{M}_\mathbb{V}\rightarrow\mathbb{C}^d,
\end{equation*}
which defines an algebraic $(\mathbb{C}^\ast)^d$-fibration whose descriminant loci are precisely the union of the complex hyperplanes $\bigcup_{i=1}^nH_{i}$. In this situation, each feasible chamber $\Delta_\alpha$ gives rise to a Lagrangian $V_\alpha$ that is diffeomorphic to the toric variety associated to the moment polytope $\Delta_\alpha$. 

In \cite{BLPW2010}, the authors introduced the hypertoric convolution algebra $B(\mathbb{V})$, which can be interpreted as the topological convolution algebra formed by the cohomologies of the complex Lagrangians in $\mathfrak{M}_\mathbb{V}$ associated to the bounded feasible chambers in $\mathbb{V}$. They also predicted that $B(\mathbb{V})$ can be realized as an appropriate version of the Fukaya category of $\mathfrak{M}_\mathbb{V}$. In the ongoing project \cite{llm1}, we investigate this question by examining the infinitesimal Fukaya category $\mathcal{F}(\mathfrak{M}_\mathbb{V},\mathfrak{f}_\mathbb{V})$ (i.e. the forward stopped subcategory of the partially wrapped Fukaya category in the sense of \cite[Definition 6.4]{GPSmicrolocal}) and show that  $\mathcal{F}(\mathfrak{M}_\mathbb{V},\mathfrak{f}_\mathbb{V})$ is quasi-equivalent to the category of perfect modules over $B(\mathbb{V})$. 

It is a general philosophy due to Aganagic \cite{mag} and Teleman \cite{ct} that the symplectic topology of $\mathfrak{M}_\mathbb{V}$ is closely related to a half-dimensional space, which is exactly the Weinstein manifold $M(\mathbb{V})$ that we study in this paper. The precise connections between the Fukaya categories of $\mathfrak{M}_\mathbb{V}$ and $M(\mathbb{V})$ are explained in Lekili-Segal \cite[Conjecture A]{LS}. They conjectured that there should be a quasi-equivalence
\begin{equation}\label{eq:conj1}
\mathcal{W}_{\mathbb{G}_m^d}(\mathfrak{M}_\mathbb{V})_{-1}\cong\mathcal{W}(M(\mathbb{V})),
\end{equation}
where $\mathcal{W}_{\mathbb{G}_m^d}(\mathfrak{M}_\mathbb{V})_{-1}\subset\mathcal{W}_{\mathbb{G}_m^d}(\mathfrak{M}_\mathbb{V})$ is the spectral component of the $\mathbb{G}_m^d$-equivariant wrapped Fukaya category at $-1$. We refer the reader to \cite[Section 1.2]{LS} for more details (see also \cite{MSZ} for a related recent work on multiplicative hypertoric varieties). 

Here, what is relevant for us is a partially wrapped version of the conjectural equivalence (\ref{eq:conj1}). The corresponding conjecture relating partially wrapped Fukaya categories then takes the form
\begin{equation}\label{eq:conj2}
\mathcal{F}_{\mathbb{G}_m^d}(\mathfrak{M}_\mathbb{V},\mathfrak{f}_\mathbb{V})\cong\mathcal{W}(M(\mathbb{V}),\xi),
\end{equation}
where $\mathcal{F}_{\mathbb{G}_m^d}(\mathfrak{M}_\mathbb{V},\mathfrak{f}_\mathbb{V})$ consists of objects in $\mathcal{F}(\mathfrak{M}_\mathbb{V},\mathfrak{f}_\mathbb{V})$ which are invariant under the fiberwise $\mathbb{G}_m^d$-action with respect to the complex moment map $\bar{\mu}_\mathbb{C}$. Note that performing localizations (stop removals) on both sides of (\ref{eq:conj2}) would imply the equivalence (\ref{eq:conj1}) conjectured by Lekili-Segal.

\begin{remark}
There are two versions of equivariant Fukaya categories that can be associated to the hypertoric variety $\mathfrak{M}_\mathbb{V}$, defined by using the geometric Hamiltonian $T^d$-action of the real torus $T^d\subset\mathbb{G}_m^d$ on the Weinstein manifold $\mathfrak{M}_\mathbb{V}$, and the algebraic $\mathbb{G}_m^d$-action on the $A_\infty$-category $\mathcal{F}(\mathfrak{M}_\mathbb{V},\mathfrak{f}_\mathbb{V})$, respectively. For the former version, see \cite{gc,df,hls,yx, KLZ, LLL}, and for the latter one, see \cite{lpa,ps3,ss}. In our case, these two versions are expected to be the same, because under hyperk\"{a}hler twist, the fiberwise Hamiltonian $T^d$-action associated to the complex moment map $\bar{\mu}_\mathbb{C}:\mathfrak{M}_\mathbb{V}\rightarrow\mathbb{C}^d$ corresponds to the fiberwise $\mathbb{G}_m^d$-action on the mirror, and the latter is expected to induce a $\mathbb{G}_m^d$-action on the Fukaya category $\mathcal{F}(\mathfrak{M}_\mathbb{V},\mathfrak{f}_\mathbb{V})$. 
\end{remark}

Consider the Fukaya $A_\infty$-algebra
\begin{equation*}
\mathcal{B}_\mathbb{V}:=\bigoplus_{\alpha,\beta\in\mathscr{P}(\mathbb{V})}\mathit{CF}^\ast(V_\alpha,V_\beta),
\end{equation*}
where  $\mathit{CF}^\ast$ denotes the infinitesimally wrapped Floer cochain complexes with respect to the stop $\mathfrak{f}_\mathbb{V}\subset\partial_\infty\mathfrak{M}_\mathbb{V}$ mentioned above. In \cite{llm1}, we will show that there is a quasi-isomorphism between $\mathcal{B}_\mathbb{V}$ and the hypertoric convolution algebra $B(\mathbb{V})$, and that the Lagrangians $\{V_\alpha\}_{\alpha\in\mathscr{P}(\mathbb{V})}$ split-generate the Fukaya category $\mathcal{F}(\mathfrak{M}_\mathbb{V},\mathfrak{f}_\mathbb{V})$. Together with conjectural quasi-equivalences (\ref{eq:conj2}), it follows that there should be a quasi-isomorphism between the partially wrapped Fukaya $A_\infty$-algebra $\widetilde{\mathcal{B}}_\mathbb{V}$ of $(M(\mathbb{V}),\xi)$ and the $\mathbb{G}_m^d$-equivariant convolution algebra $\widetilde{B}(\mathbb{V})$. Theorem \ref{theorem:conv1} confirms that this is indeed the case.

\subsection{Idea of the Proof.}\label{sec:ideaofproof} 
In this subsection, we give an outline of the proof of our generation result (Theorem \ref{theorem:generation1}). Let us start with a  $1$-dimensional hyperplane arrangement, whose complexified complement $M(\mathbb{V})$ is $\mathbb{C}\setminus\{r_1, r_2,\dotsb, r_n\}$, where $-\infty<r_1<\cdots<r_n<\infty$, a complex plane with $n$ points on the real line removed. Note that Lagrangian chambers are the real line segments connecting the consecutive punctures in $\mathbb{R}$ (colored in blue in Figure \ref{fig:1dHA}), together with a ray to positive or negative infinity depending on the choice of polarization $\xi$. 

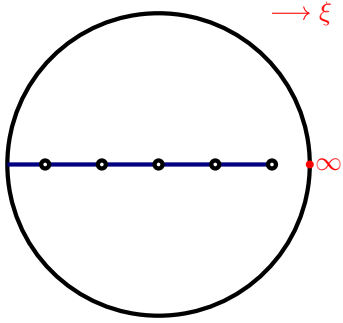
\begin{figure}[ht]
    \centering
    \begin{tikzpicture}[scale=0.5]
        \draw [ultra thick] (0,0) circle (4);
        \draw [blue!50!black,ultra thick] (-4,0) -- (-3,0);
        \foreach \y in {-2,-1,0,1} {
        \draw [blue!50!black,ultra thick] ({1.5*\y},0) -- ({1.5*\y+1.5},0);
        }
        \foreach \x in {-2,-1,0,1,2} {
            \node [circle,draw=black, fill=white, inner sep=0pt,minimum size=3pt,ultra thick] at ({1.5*\x},0) {};
        }
        \node [circle, fill=black, inner sep=0pt, minimum size=3pt, red,ultra thick] at (4,0) {};
        \node [red] at (4.5,0) {$\infty$};
        \draw [->, red] (3,4) -- (4,4) node [right] {$\xi$};
    \end{tikzpicture}
    \caption{One-dimensional hyperplane arrangement\label{fig:1dHA}}
\end{figure}

In this case, the generation is \textit{a priori} clear: we can prove generation via its Weinstein handle structure, by arguing that the real line segments connecting punctures are cocores for a given Weinstein structure on $M(\mathbb{V})$, and applying the generation result in \cite{ganatra2018sectorial}, which says that cocores of a Weinstein manifold generate the fully wrapped Fukaya category. To see that the line segments combined with the ray $(-\infty,r_1)$ generate the partially wrapped Fukaya category $\mathcal{W} (M(\mathbb{V}),\xi)$, we appeal to the computation of wrapped Floer cohomology shown in \cite{abouzaid2013homological}, where an exact triangle relating the blue arcs in Figure \ref{fig:1dHA} to the Lagrangians $L_1^{\pm}$ in Figure \ref{intro:wrapping-exact-triangle} was established via an explicit $A_{\infty}$-algebra computation.\par
To motivate, we prove the generation via a ``local-to-global'' approach: consider vertical real lines cutting $M(\mathbb{V})$ into union of ``sectors'', as shown in Figure \ref{fig:1dHA-descent}. Each piece can be identified with part of $\mathbb{C}^\ast$, together with a stop $\infty$ or two stops $\pm\infty$, so the wrapped Fukaya category of each sector is generated by the blue rays depicted in the figure, where the upper and lower ones are both Lagrangian cocores of $\mathbb{C}^\ast$, which together generate the linking disks to all the stops. Sectorial descent (cf. Theorem \ref{Prelim:sectorial-descent}) then allows us to obtain a collection of generators for $\mathcal{W}(M(\mathbb{V}),\xi)$, as the union of all the blue arcs. Observe that after gluing the sectors together to get the ambient punctured disk $(M(\mathbb{V}),\xi)$, the left most upper and lower blue Lagrangians are isomorphic via counterclockwise rotation, and the lower blue Lagrangians in the other sectors can be obtained from the upper ones by performing iterated surgeries. Therefore the upper collection of the blue arcs would suffice to generate $\mathcal{W} (M(\mathbb{V}),\xi)$.\par
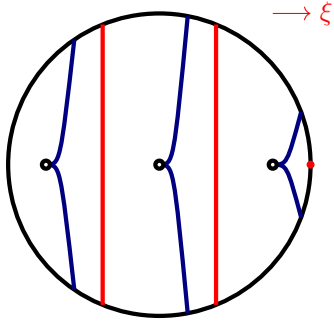
\begin{figure}[ht]
    \centering
    \begin{tikzpicture}
        [scale=0.5]
        \draw [ultra thick] (0,0) circle (4);
        \foreach \y in {-2,0,2} {
            \draw [blue!50!black,ultra thick] ({1.5*\y},0) .. controls ({1.5*(\y+0.25)},0) .. ({1.5*(\y+0.5)},{sqrt(16-(1.5*(\y+0.5))^2)});
            \draw [blue!50!black,ultra thick] ({1.5*\y},0) .. controls ({1.5*(\y+0.25)},0) .. ({1.5*(\y+0.5)},{-sqrt(16-(1.5*(\y+0.5))^2)});
        }
        \foreach \x in {-2,0,2} {
            \node [circle,draw=black, fill=white, inner sep=0pt,minimum size=3pt,ultra thick] at ({1.5*\x},0) {};
        }
        \foreach \z in {-2,0} {
            \draw [red,ultra thick] ({1.5*\z + 1.5},{sqrt(16-(1.5*\z+1.5)^2)}) -- ({1.5*\z + 1.5},{-sqrt(16-(1.5*\z+1.5)^2)});
        }
        \node [circle, fill=black, inner sep=0pt, minimum size=3pt, red, ultra thick] at (4,0) {};
        \draw [->,red] (3,4) -- (4,4) node [right] {$\xi$};
    \end{tikzpicture}
    \caption{Cutting one-dimensional hyperplane arrangements into sectors\label{fig:1dHA-descent}}
\end{figure}
The final step to the proof is to show that the real lines depicted in Figure \ref{fig:1dHA} generate the upper blue arcs in Figure \ref{fig:1dHA-descent}. The left most Lagrangians are canonically identified by rotation, and we observe that the upper Lagrangians in Figure \ref{fig:1dHA-descent} can be obtained from inductively concatenating the real lines in Figure \ref{fig:1dHA} through the stop, or equivalently via ``wrapped Lagrangian surgery'' around the puncture (see Figure \ref{fig:wrapped-surgery-1d}). We therefore conclude that the blue Lagrangians in Figure \ref{fig:1dHA} generate $\mathcal{W} (M(\mathbb{V}),\xi)$.\par
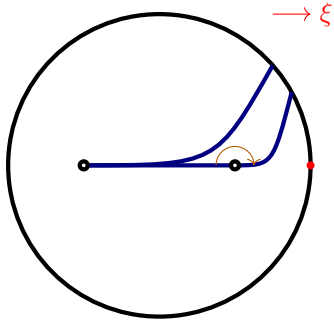
\begin{figure}[ht]
    \centering
    \begin{tikzpicture}
        [scale=0.5]
        \draw [ultra thick] (0,0) circle (4);
        \draw [blue!50!black,ultra thick] (-2,0) .. controls (1.5,0) .. (3,{sqrt(7)});
        \draw [blue!50!black,ultra thick] (-2,0) -- (2,0);
        \draw [blue!50!black,ultra thick] (2,0) .. controls (3,0) .. (3.5, {sqrt(16-(3.5)^2)});
        \draw [orange!70!black,->] (1.5,0) arc [start angle=180, end angle =0, radius=0.5];
        \node [circle, fill=white, draw=black, inner sep=0pt, minimum size=3pt,ultra thick] at (-2,0) {};
        \node [circle, fill=white, draw=black, inner sep=0pt, minimum size=3pt,ultra thick] at (2,0) {};
        \node [circle, fill=black, inner sep=0pt, minimum size=3pt, red, ultra thick] at (4,0) {};
        \draw [->,red] (3,4) -- (4,4) node [right] {$\xi$};
    \end{tikzpicture}
    \caption{Wrapped surgery around a puncture\label{fig:wrapped-surgery-1d}}
\end{figure}

Things get more complicated as the dimension increases, but the key phenomenon can already be seen in complex dimension $2$. To illustrate, we will explain our proof when $M(\mathbb{V})$ is the 2-dimensional pair-of-pants associated with the hyperplane arrangement shown in Figure \ref{fig:parrr}. As a symmetric product of two $1$-dimensional pair-of-pants, the generation of its wrapped Fukaya category can be computed from dimension $1$ cases as explained in \cite{Auroux-H}, but our alternative proof sketched here will also apply to general two-dimensional hyperplane arrangements. Using the polarization $\xi$, we can cut the Weinstein sector $(M(\mathbb{V}),\xi)$ into four subsectors as shown in Figure \ref{sector:decompose-pair-of-pants} (cf. Propositions \ref{sector:CompactWeinsteinDeformation}, \ref{sector:CompactWeinsteinDeformationMultiple}). For $i=1,\dots, 4$, we denote by $\mathfrak{s}_i$ the $i$-th sectorial hypersurface and by $X_{i,i+1}$ the subsectors after cutting. Each of these subsectors looks pretty similar in the figure, containing of a single crossing point and several non-intersecting line segments. After curving the picture, we get four sectors on the right-hand side of Figure \ref{sector:decompose-pair-of-pants}, and the generation problem in this case reduces to that of the wrapped Fukaya category for each subsector.

\begin{figure}[ht]
    \centering
    \begin{tikzpicture}[scale=0.5]
        \draw [ultra thick] (-4,0) -- (4,0);
        \draw [ultra thick] (0,-5) -- (0,4);
        \draw [ultra thick] (-4,2) -- (3,-5);
        \draw [->, red] (4,2) -- (3,4) node[above] {$\xi$};
        \draw [ultra thick, red] (-4,-5) -- (4,-1);
        \draw [ultra thick, red] (-4,-3) -- (4,1);
        \draw [ultra thick, red] (-4,-1.5) -- (4,2.5);
        \draw [ultra thick, red] (-4,1) -- (2,4);
        \node at (5,0) {\Large $\Rightarrow$};
        \draw [red, ultra thick] (6,1) -- (14,1) node [right] {$\mathfrak{s}_2$};
        \draw [red, ultra thick] (6,-1) -- (14,-1) node [right] {$\mathfrak{s}_3$};
        \draw [red, ultra thick] (6,3) -- (14,3) node [right] {$\mathfrak{s}_1$};
        \draw [red, ultra thick] (6,-3) -- (14,-3) node [right] {$\mathfrak{s}_4$};
        \draw [ultra thick] (8,3) -- (12,-1) -- (12,-5);
        \draw [ultra thick] (10,3) -- (8,1) -- (8,-1) -- (10,-3) -- (10,-5);
        \draw [ultra thick] (12,3) -- (12,1) -- (8,-3) -- (8,-5);
        \node at (14,2) {$X_{12}$};
        \node at (14,0) {$X_{23}$};
        \node at (14,-2) {$X_{34}$};
        \node at (14,-4) {$X_{45}$};
    \end{tikzpicture}
\caption{Decomposing $2$-dimensional pair-of-pants into sectors\label{sector:decompose-pair-of-pants}}
\end{figure}
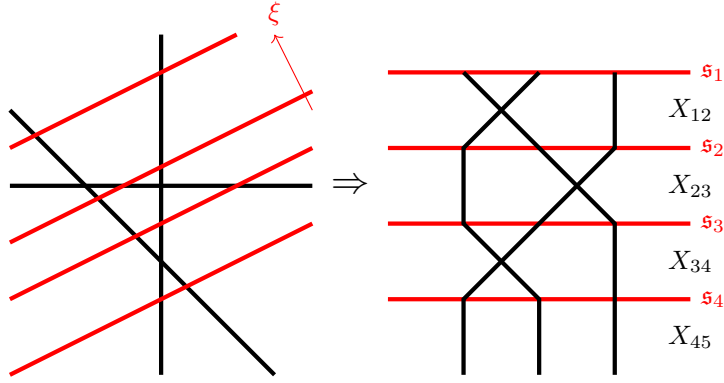

For each of these subsectors, we perform a further cut by hyperplanes which are transverse to the hyperplane defined by the polarization $\xi$, cutting each sector into ``squares'' as shown in Figure \ref{sector:second-cut}. Squares that surround an intersection point of the hyperplanes can be identified with $(\mathbb{C}^{\times})^2$ with four stops given by the two linear functions, and the other squares are products of the red boundary with $T^{\ast} [0,1]$, whose Fukaya category is equivalent to the Fukaya category of one of the two red boundaries \cite{ganatra2018sectorial}. It is well-known that the wrapped Fukaya category of $(\mathbb{C}^{\ast})^2$ is generated by one of the four real quadrants (a cotangent fiber) \cite{abouzaid2011cotangent}, and adding stops adds linking disks associated to these stops as new generators. Moreover, the linking disks correspond to generating Lagrangians of the stops, which are $M(\mathbb{V})$ for some one-dimensional hyperplane arrangement, and from the previous paragraph on dimension $1$ hyperplane arrangements we know that they are generated by real lines connecting punctures.

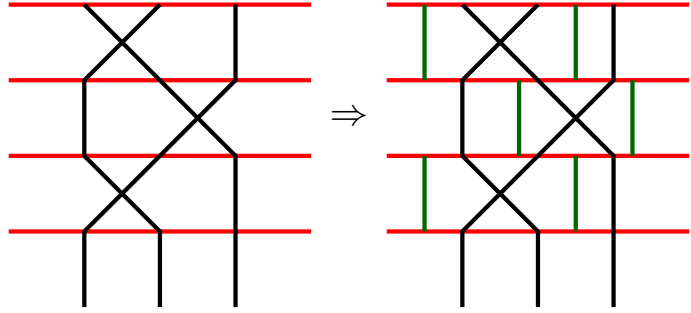
\begin{figure}[ht]
    \centering
    \begin{tikzpicture}[scale=0.5]
        \draw [red, ultra thick] (-4,1) -- (4,1);
        \draw [red, ultra thick] (-4,-1) -- (4,-1);
        \draw [red, ultra thick] (-4,3) -- (4,3);
        \draw [red, ultra thick] (-4,-3) -- (4,-3);
        \draw [ultra thick] (-2,3) -- (2,-1) -- (2,-5);
        \draw [ultra thick] (0,3) -- (-2,1) -- (-2,-1) -- (0,-3) -- (0,-5);
        \draw [ultra thick] (2,3) -- (2,1) -- (-2,-3) -- (-2,-5);
        \node at (5,0) {\Large $\Rightarrow$};
        \draw [red, ultra thick] (6,1) -- (14,1);
        \draw [red, ultra thick] (6,-1) -- (14,-1);
        \draw [red, ultra thick] (6,3) -- (14,3);
        \draw [red, ultra thick] (6,-3) -- (14,-3);
        \draw [ultra thick] (8,3) -- (12,-1) -- (12,-5);
        \draw [ultra thick] (10,3) -- (8,1) -- (8,-1) -- (10,-3) -- (10,-5);
        \draw [ultra thick] (12,3) -- (12,1) -- (8,-3) -- (8,-5);
        \draw [black!60!green, ultra thick] (7,3) -- (7,1);
        \draw [black!60!green, ultra thick] (7,-1) -- (7,-3);
        
        \draw [black!60!green, ultra thick] (11,3) -- (11,1);
        \draw [black!60!green, ultra thick] (11,-1)--(11,-3);
        
        \draw [black!60!green, ultra thick] (9.5,1) -- (9.5,-1);
        
        \draw [black!60!green, ultra thick] (12.5,1) -- (12.5,-1);
    \end{tikzpicture}
    \caption{Second cut of the $P_j$'s into the standard pieces.\label{sector:second-cut}}
\end{figure}

By Sectorial Descent (Theorem \ref{Prelim:sectorial-descent}), we get a collection of generating Lagrangians, which are depicted in Figure \ref{intro:generating-lags-for-2d-pairs-of-pants}. 
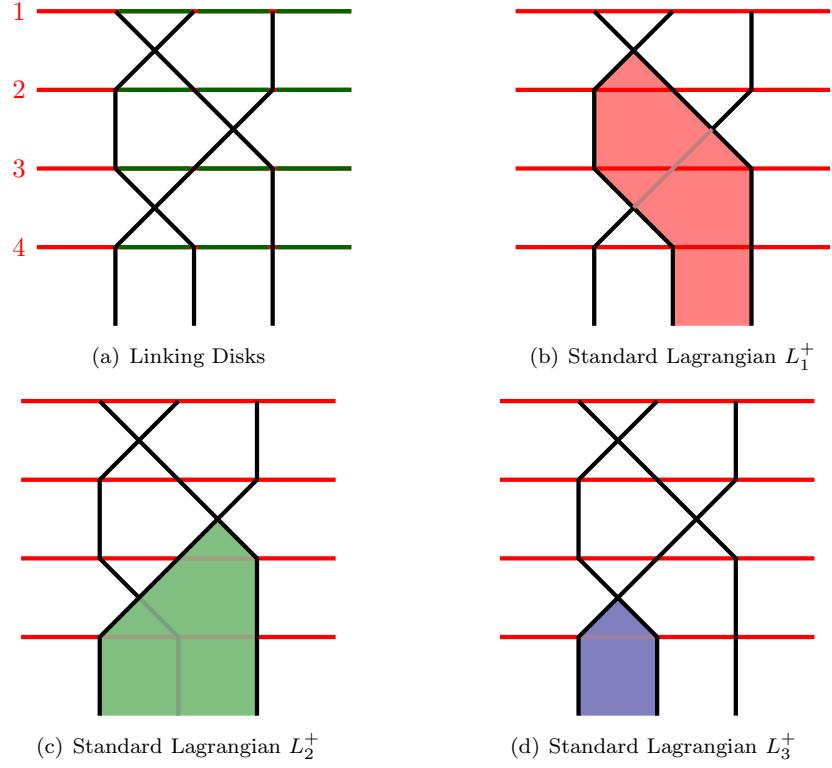
\begin{figure}[ht]
	\centering
	\subfigure[Linking Disks\label{intro:linking-disks}]{\begin{tikzpicture}[scale=.52]
		\draw [red, ultra thick] (-4,1) node [left] {$2$} -- (4,1);
        \draw [red, ultra thick] (-4,-1) node [left] {$3$} -- (4,-1);
        \draw [red, ultra thick] (-4,3) node [left] {$1$} -- (4,3);
        \draw [red, ultra thick] (-4,-3) node [left] {$4$} -- (4,-3);
        \draw [ultra thick] (-2,3) -- (2,-1) -- (2,-5);
        \draw [ultra thick] (0,3) -- (-2,1) -- (-2,-1) -- (0,-3) -- (0,-5);
        \draw [ultra thick] (2,3) -- (2,1) -- (-2,-3) -- (-2,-5);
		
		\draw [black!60!green, ultra thick] (-1.9,3) -- (-0.1,3);
		\draw [black!60!green, ultra thick] (0.1,3) -- (1.9,3);
		\draw [black!60!green, ultra thick] (2.1,3) -- (4,3);

		\draw [black!60!green, ultra thick] (-1.9,1) -- (-0.1,1);
		\draw [black!60!green, ultra thick] (0.1,1) -- (1.9,1);
		\draw [black!60!green, ultra thick] (2.1,1) -- (4,1);

		\draw [black!60!green, ultra thick] (-1.9,-1) -- (-0.1,-1);
		\draw [black!60!green, ultra thick] (0.1,-1) -- (1.9,-1);
		\draw [black!60!green, ultra thick] (2.1,-1) -- (4,-1);

		\draw [black!60!green, ultra thick] (-1.9,-3) -- (-0.1,-3);
		\draw [black!60!green, ultra thick] (0.1,-3) -- (1.9,-3);
		\draw [black!60!green, ultra thick] (2.1,-3) -- (4,-3);
	\end{tikzpicture}}\hspace{2cm}
	\subfigure[Standard Lagrangian $L_1^+$]{\begin{tikzpicture}[scale=0.52]
        \fill [red!50] (-1,2) -- (-2,1) -- (-2,-1) -- (0,-3) -- (0,-5) -- (2,-5) -- (2,-1) -- (-1,2);
		\draw [red, ultra thick] (-4,1) -- (4,1);
        \draw [red, ultra thick] (-4,-1) -- (4,-1);
        \draw [red, ultra thick] (-4,3) -- (4,3);
        \draw [red, ultra thick] (-4,-3) -- (4,-3);
        \draw [ultra thick] (-2,3) -- (2,-1) -- (2,-5);
        \draw [ultra thick] (0,3) -- (-2,1) -- (-2,-1) -- (0,-3) -- (0,-5);
        \draw [ultra thick] (2,3) -- (2,1) -- (1,0);
        \draw [black!50!red!50, ultra thick] (1,0) -- (-1,-2);
        \draw [ultra thick] (-1,-2) -- (-2,-3) -- (-2,-5);
	\end{tikzpicture}}\\
	\subfigure[Standard Lagrangian $L_2^+$]{\begin{tikzpicture}[scale=0.52]
        \fill [green!50!black!50] (1,0) -- (-2,-3) -- (-2,-5) -- (2,-5) -- (2,-1) -- (1,0);
		\draw [red, ultra thick] (-4,1) -- (4,1);

        \draw [red, ultra thick] (-4,-1) -- (0,-1);
        \draw [red!50!green!50!black!50, ultra thick] (0,-1) -- (2,-1);
        \draw [red, ultra thick] (2,-1) -- (4,-1);

        \draw [red, ultra thick] (-4,-3) -- (-2,-3);
        \draw [red!50!green!50!black!50, ultra thick] (-2,-3) -- (2,-3);
        \draw [red, ultra thick] (2,-3) -- (4,-3);

        \draw [red, ultra thick] (-4,3) -- (4,3);
        \draw [ultra thick] (-2,3) -- (2,-1) -- (2,-5);

        \draw [ultra thick] (0,3) -- (-2,1) -- (-2,-1) -- (-1,-2);
        \draw [black!50!green!50!black!50, ultra thick] (-1,-2) -- (0,-3) -- (0,-5);

        \draw [ultra thick] (2,3) -- (2,1) -- (-2,-3) -- (-2,-5);
	\end{tikzpicture}}
	\hspace{2cm}\subfigure[Standard Lagrangian $L_3^+$ ]{\begin{tikzpicture}[scale=0.52]
        \fill [blue!50!black!50] (-1,-2) -- (-2,-3) -- (-2,-5) -- (0,-5) -- (0,-3) -- (-1,-2);
		\draw [red, ultra thick] (-4,1) -- (4,1);
        \draw [red, ultra thick] (-4,-1) -- (4,-1);
        \draw [red, ultra thick] (-4,3) -- (4,3);

        \draw [red, ultra thick] (-4,-3) -- (-2,-3);
        \draw [red!50!blue!50!black!50, ultra thick] (-2,-3) -- (0,-3);
        \draw [red, ultra thick] (0,-3) -- (4,-3);
        
        \draw [ultra thick] (-2,3) -- (2,-1) -- (2,-5);
        \draw [ultra thick] (0,3) -- (-2,1) -- (-2,-1) -- (0,-3) -- (0,-5);
        \draw [ultra thick] (2,3) -- (2,1) -- (-2,-3) -- (-2,-5);
	\end{tikzpicture}}
	\caption{Generating Lagrangians for Two-Dimensional Pairs of Pants.\label{intro:generating-lags-for-2d-pairs-of-pants}}
\end{figure}
The green line segments in Figure \ref{intro:linking-disks} represent the linking disks associated to the corresponding generating Lagrangians of the stop, and the remaining three pictures describe the three generating Lagrangians obtained from each crossing point (after gluing together all the sectors). We label the three standard Lagrangians by $L_1^+, L_2^+, L_3^+$, respectively, where the $+$ indicates that the Lagrangian is the one shifted upwards (under the Hamiltonian flow) in direction of $\mathrm{Im}(\xi)$ in $\mathbb{C}^2$. We denote the linking disks by $D_{ij}$, where $1\leq i\leq 4$ is the index of the stop and $1\leq j\leq 3$ is the order of cocores, from left to right. It remains to show that the Lagrangians associated to the bounded-feasible chambers $L_{++-}$, $L_{+--}$ and $L_{+-+}$, as depicted in Figure \ref{intro:three-lag-chambers}, generate all the standard Lagrangians $L_1^+, L_2^+, L_3^+$ and the linking disks $D_{ij}$ in the partially wraped Fukaya category $\mathcal{W}(M(\mathbb{V}), \xi)$. Applying the surgery exact triangle (cf. Proposition \ref{prop:flow-surgery-cobordism}) we have: 

\begin{lemma}[cf. Section \ref{sec:gluingchambers}]\label{lem:intro-Lag gluing}
There are quasi-isomorphisms 
\[
L_1^+\simeq \mathit{Cone}(L_{++-} \to  L_{+--}),\quad L_2^+\simeq \mathit{Con}e(L_{+--}\to L_{+-+}),\quad L_3^+\simeq L_{+-+}
\]
between objects in the triangulated $A_\infty$-category $\mathcal{W}(M(\mathbb{V}),\xi)^\mathit{perf}$ (cf. Section \ref{section:pr-wrap}), where $\mathit{Cone}(\bullet\rightarrow\bullet)$ stands for the mapping cone.
\end{lemma}

Thus it remains to show that the linking disks $D_{ij}$ are generated by the ``chamber Lagrangians" $L_{++-}$, $L_{+--}$ and $L_{+-+}$. To prove this, let us first focus on the top sector $X_{12}$ in Figure \ref{intro: X_12}.

\begin{figure}[ht]
	\centering
	\begin{tikzpicture}[scale=0.5]
	\draw [red, ultra thick] (-5,1) node [left] {$\mathfrak{s}_1$} -- (5,1);
	\draw [red, ultra thick] (-5,-3) node [left] {$\mathfrak{s}_2$} -- (5,-3);
	\draw [black, ultra thick] (-3,1) -- (1,-3);
	\draw [black, ultra thick] (1,1) -- (-3,-3);
	\draw [black, ultra thick] (3,1) -- (3,-3);
	
	\draw [black!60!green, ultra thick] (-2.9,1) -- (0.9,1);
	\draw [black!60!green, ultra thick] (1.1,1) -- (2.9,1);
	\draw [black!60!green, ultra thick] (3.1,1) -- (4.9,1);
	
	\node at (-1, 1.5) [black!60!green] {$D_{11}$};
	\node at (2, 1.5) [black!60!green] {$D_{12}$};
	\node at (4, 1.5) [black!60!green] {$D_{13}$};
	
	\draw [black!60!green, ultra thick] (-2.9,-3) -- (0.9,-3);
	\draw [black!60!green, ultra thick] (1.1,-3) -- (2.9,-3);
	\draw [black!60!green, ultra thick] (3.1,-3) -- (4.9,-3);
	
	\node at (-1, -3.5) [black!60!green] {$D_{21}$};
	\node at (2, -3.5) [black!60!green] {$D_{22}$};
	\node at (4, -3.5) [black!60!green] {$D_{23}$};
	
	\node at (-1, 0.2) {$L_1^+$};
	\node at (-1, -2.2) {$\widetilde{L}_1^+$};
	\end{tikzpicture}
	\caption{Description of $X_{12}$. $\tilde{L}_1^+$ is a flipping of $L_1^+$. \label{intro: X_12}}
\end{figure}
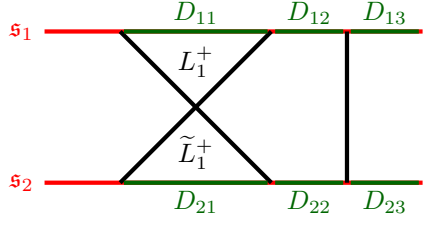
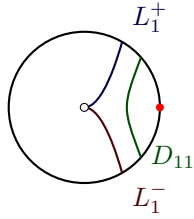
\begin{figure}[ht]
	\centering
	\begin{tikzpicture}[scale=0.5]
	\draw [thick] (0,0) circle [radius=2];
	\draw [thick, red!30!black] (0,0) parabola bend (0,0) (1,{-sqrt(3)}) node [below right] {$L_1^-$};
	\draw [thick, blue!30!black] (0,0) parabola bend (0,0) (1,{sqrt(3)}) node [above right] {$L_1^+$};
	\draw [thick, green!30!black] (1.5,{sqrt(7)/2}) .. controls (1,0) .. (1.5, {-sqrt(7)/2}) node [right] {$D_{11}$};
	\node at (2,0) [circle,fill=black,inner sep=0pt,minimum size=3pt,red] {};
	\node at (0,0) [circle,draw=black, fill=white, inner sep=0pt,minimum size=3pt] {};
	\end{tikzpicture}
	\caption{\label{intro:wrapping-exact-triangle}The Wrapping Exact Triangle}
\end{figure}

Consider the Lagrangian $L_1^+$ in this sector. We have the wrapping exact triangle for the linking disk $D_{11}$
\[
L_1^-\to D_{11}\to L_1^+\to L_1^- [1],
\]
\noindent where $L_1^-$ is the triangular Lagrangian obtained by shifting $L_1$ down in the direction of $\mathrm{Im}(\xi)$, see Figure \ref{intro:wrapping-exact-triangle}. This implies that $D_{11}$ is generated by $L_1^-$ and $L_1^+$. Consider the second sectorial hypersurface $\mathfrak{s}_2$, regarded as a stop in $X_{12}$, there is a similar wrapping exact triangle  
\[
\widetilde{L}_1^-\to D_{21}\to \widetilde{L}_1^+\to \widetilde{L}_1^-[1].
\]
Here, $\widetilde{L}_1^{\pm}$ is a flipping of $L_1^{\pm}$ (cf. Section \ref{sec:flippingLagrangians}). Since there is a quasi-isomorphism $\tilde{L}_1^{\pm} \cong L_1^{\pm}$ in $\mathcal{W}(X_{12})$, we can abuse notations and write both as $L_1^{\pm}$. This shows that we could get $L_1^-$ from $D_{21}$, therefore $D_{11}$ is generated by $L_1^+$ and $D_{21}$, hence it can be removed from the collection of generating objects for $\mathcal{W}\left(X_{12}\right)$. For the remaining linking disks in the first stop $\mathfrak{s}_1$, we can show via Lagrangian surgery that 

\begin{lemma}[cf. Section \ref{s:chamber_moves}]\label{lemma:qis}
We have quasi-isomorphisms
\[
D_{22}\simeq\mathit{Cone}(D_{11} \to  D_{12}) \quad\text{and}\quad D_{13}\simeq D_{23}
\]
in the Fukaya category $\mathcal{W}(X_{12})^\mathit{perf}$. 
\end{lemma}

\begin{corollary}\label{cor:intro-X12}
The partially wrapped Fukaya category $\mathcal{W}(X_{12})$ is generated by the Lagrangian $L_1^+$ and the linking disks $D_{21}, D_{22}$ and $D_{23}$. 
\end{corollary}

Note that the quasi-isomorphisms in Lemma \ref{lemma:qis} also hold in the wrapped Fukaya category of the ambient Liouville sector $(M(\mathbb{V}),\xi)$, by the existence of an inclusion functor associated to the embedding $X_{12}\hookrightarrow (M(\mathbb{V}),\xi)$. Under this inclusion, $L_1^+$ in $\mathcal{W}(X_{12})$ is sent to the standard Lagrangian in Figure \ref{intro:generating-lags-for-2d-pairs-of-pants} (b). 

Similar arguments can be carried out for the Liouville sectors $X_{23}$ and $X_{34}$. We can identify linking disks at the top and the bottom via surgery exact sequences, therefore inductively replacing linking disks from the top with those at the bottom in the collection of generating objects. Finally, it remains to deal with the linking disks $D_{41}, D_{42}$ and $D_{43}$. To do this, note that the subsector $X_{45}$ at the bottom has a geometric decomposition:

\begin{lemma}
We have an isomorphism of Liouville sectors $X_{45}\cong\mathfrak{s}_4\times\mathbb{C}_{\rRe\geq 0}$.
\end{lemma}

This implies that all the linking disks $D_{41}, D_{42}$ and $D_{43}$ become zero objects after pushing forward to the ambient sector $(M(\mathbb{V}),\xi)$. They can thus be removed from the collection of generating objects. We conclude from Lemma \ref{lem:intro-Lag gluing} and corollary \ref{cor:intro-X12} that 

\begin{proposition}[cf. Theorem \ref{thm:generation}]
The partially wrapped Fukaya category $\mathcal{W} (M(\mathbb{V}),\xi)$ is generated by the Lagrangians $L_{++-}, L_{+--}$ and $L_{+-+}$ corresponding to the bounded and feasible chambers in the polarized hyperplane arrangement $\mathbb{V}$.
\end{proposition}

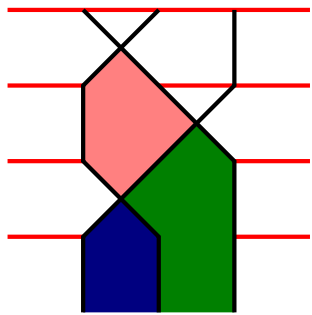
\begin{figure}[ht]
	\centering
	\begin{tikzpicture}[scale=0.5]
		\draw [red, ultra thick] (-4,1) -- (4,1);
        \draw [red, ultra thick] (-4,-1) -- (4,-1);
        \draw [red, ultra thick] (-4,3) -- (4,3);
        \draw [red, ultra thick] (-4,-3) -- (4,-3);
		\fill [red!50] (-1,2) -- (-2,1) -- (-2,-1) -- (-1,-2) -- (1,0) -- (-1,2);
		\fill [green!50!black] (1,0) -- (-1,-2) -- (0,-3) -- (0,-5) -- (2,-5) -- (2,-1) -- (1,0);
		\fill [blue!50!black] (-1,-2) -- (-2,-3) -- (-2,-5) -- (0,-5) -- (0,-3) -- (-1,-2);
        \draw [ultra thick] (-2,3) -- (2,-1) -- (2,-5);
        \draw [ultra thick] (0,3) -- (-2,1) -- (-2,-1) -- (0,-3) -- (0,-5);
        \draw [ultra thick] (2,3) -- (2,1) -- (-2,-3) -- (-2,-5);
	\end{tikzpicture}
	\caption{\label{intro:three-lag-chambers}The Bounded Feasible Chambers.\\ The red one is $L_{++-}$, the green one is $L_{+--}$ and the blue one is $L_{+-+}$}
\end{figure}

\subsection{Outline.}
The paper is organized as follows. In Section \ref{Topology-of-Hyperplane-Arrangements}, we recall the basics of Liouville sectors and Fukaya categories based on the works \cite{ganatra2018sectorial,ganatra2020covariantly}. Section \ref{sec:symplecticgeometry of hypargg} contains the main technical input of this paper. We first recall the notion of a polarized hyperplane arrangement in Section \ref{sec:hyparr}, and then establish some basic facts about the symplectic geometry of the Weinstein manifold $M(\mathbb{V})$ in Section \ref{ss:essential}. The sectorial cut of $\left(M(\mathbb{V}),\xi\right)$ is carried out rigorously in Sections \ref{section:homotopy} and \ref{sec:sectorial decomposition}. This is followed by a discussion about the generation for each subsector in the sectorial decomposition in Section \ref{section:piece}. In Section \ref{Exact triangle, surgery and isomorphic objects}, we prove Theorem \ref{theorem:generation1} using the surgery exact triangle for Lagrangian submanifolds with clean intersections. Using a computational technique similar to that of \cite{auroux2017speculations}, Theorem \ref{theorem:conv1} is proved in Section \ref{sec:computation}. As an application, we study the functoriality of Fukaya categories with respect to the deletion and restriction of hyperplane arrangements in Section \ref{functoriality-section}.

\subsection*{Acknowledgements}
We thank Sheel Ganatra for numerous discussions on his works \cite{ganatra2020covariantly}, \cite{ganatra2018sectorial}. We also thank Aaron Lauda and Nick Sheridan for helpful communications. We thank the anonymous referees for carefully reading our paper and providing useful suggestions.

We thank Heilbronn focused research grants for the support on this project. S.L. was supported by the Leverhulme Prize award from the Leverhulme Trust and the Enhancement Award from the Royal Society University Research Fellowship, both awarded to Nick Sheridan. Y.L. is partially supported by Simons grant \#385571. S.-Y. L. was partly supported by NSF Career Award \# DMS-2048055, awarded to Sheel Ganatra. C.M. was partly supported by the Royal Society University Research Fellowship while working on this project.

\section{Recollement of Sector Theory}\label{Topology-of-Hyperplane-Arrangements}

In this section, we recall the definitions and some important structural results for Fukaya categories of Weinstein manifolds.

\subsection{Liouville Sectors.} Recall that a \emph{symplectic manifold} is a pair $(X,\omega)$, where $X$ is an even dimensional real manifold and $\omega\in\Omega^2 (X)$ is a closed nondegenerate $2$-form on $X$. In this paper, we will study a particular class of non-compact symplectic manifolds.

\begin{definition}\label{Prelim:Liouville-manifolds}
 A symplectic manifold $(X,\omega)$ is a \emph{Liouville manifold} if $\omega$ is exact and there exists a vector field $Z$ on $X$ such that 
    \begin{enumerate}[\indent (i)]
       \item $Z$ is complete;
       \item $\mathcal{L}_Z\omega =\omega$, or equivalently, $Z$ is dual to a primitive $\lambda$ of $\omega$;
       \item there exists a (codimension $0$) compact symplectic submanifold with boundary $(X_0 ,\omega)\subset(X,\omega)$ such that $Z$ is outward pointing along $\partial X_0$ and $Z$ has no critical points outside $X_0$.
    \end{enumerate}
We call this vector field $Z$ a \emph{Liouville vector field} and the pair $(\omega ,Z)$ a \emph{Liouville structure} on $X$. The primitive $\lambda$ of $\omega$ is called a \emph{Liouville $1$-form}.
\end{definition}

In this paper, all Liouville manifolds are assumed to have vanishing first Chern class, i.e. $c_1(X)=0$.

The compact symplectic submanifold $(X_0,\omega)$ in the item (iii) is called a \emph{Liouville domain}. We can also construct a Liouville manifold $X$ out of the given Liouville domain $(X_0,\omega)$ by taking its \emph{completion}:
\[
    X =X_0\cup_{\partial X_0}\partial X_0\times [1,+\infty),
\]
where the Liouville structure on $\partial X_0\times [1,+\infty)_r$ is $\left(\omega =d(r\lambda\vert_{\partial X_0}),Z=r\partial_r\right)$.

\begin{definition}
Let $(X,\omega, Z)$ be a Liouville manifold. We call $X$ \emph{Weinstein} if there exists an exhausting Morse function $\phi\colon X\to\mathbb{R}$ that is Lyapunov for $Z$. The triple  $(\omega ,Z,\phi)$ is called a \emph{Weinstein structure} on $X$.
\end{definition} 

\begin{definition}
 A \emph{Stein manifold} is a complex manifold $(X,J)$ which admits a proper holomorphic embedding $i\colon X\to\mathbb{C}^N$ for some $N$.
\end{definition}
 Given a Stein manifold $(X,J)$ equipped with an embedding $i\colon X\hookrightarrow\mathbb{C}^N$, one can consider the squared distance function  $\phi_{\mathbb{C}^N} (z)=\displaystyle\sum_{i=1}^N\vert z_i\vert^2$. The pull-back function $\phi\coloneqq\phi_{\mathbb{C}^{N}}\circ i$ defines a pluri-subharmonic function on $X$, hence also a symplectic form $\omega_{\phi}=-dd^\mathbb{C}\phi$ and a Liouville vector field $\nabla\phi$ dual to $-d^\mathbb{C}\phi$. This assignment defines a map from the category of Stein manifolds with a fixed potential $\phi$ to the category of Weinstein manifolds with the same potential $\phi$. It is shown in \cite[Theorem 1.1]{cieliebak2012stein} that this map is a weak homotopy equivalence.

A submanifold $S\subseteq X$ is \emph{coisotropic} if $\dim S\geq\frac{1}{2}\dim X$ and for every $s\in S$, $\ker\omega\vert_{T_sS}$ has constant positive rank. It follows directly that every real hypersurface of $X$ is coisotropic. If $S\subseteq X$ is coisotropic, write $\mathcal{D}_S\coloneqq\ker\omega\vert_S\subseteq TS$ for the \emph{characteristic foliation} of $S$.

\begin{definition}\label{Prelim:Liouville-sectors}
Let $(X,\omega,Z)$ be a Liouville manifold. A real oriented hypersurface $S\subseteq X$ is called \emph{sectorial} if the following conditions are satisfied:
    \begin{enumerate}[\indent (i)]
        \item $Z$ is tangent to $S$ near infinity;
        \item There exists a $Z$-invariant neighbourhood $\Nbd^Z (S)$ of $S\subset X$ and functions
        \begin{equation} \nonumber
        R,I\colon\Nbd^Z (S)\to\mathbb{R}
        \end{equation}
        such that $R^{-1}(0)=S$ and 
        \begin{equation}\nonumber
        (R,I)\colon\Nbd^Z (S)\to\mathbb{C}_{-\varepsilon <\rRe<\varepsilon}
        \end{equation}
        is a trivial symplectic fibration onto the open strip in $\mathbb{C}$ consisting of complex numbers with real part constrained in $(-\varepsilon,\varepsilon)$, where $\mathbb{C}_{-\varepsilon <\rRe<\varepsilon}$ is equipped with the standard symplectic form restricted from $\mathbb{C}$;
        \item Under the above decomposition $\Nbd^Z (S)\xrightarrow{\simeq}(R,I)^{-1} (0)\times\mathbb{C}_{-\varepsilon <\rRe<\varepsilon}$, the Liouville $1$-form $\lambda_X$ on $X$ is sent to $\lambda_X\vert_{(R,I)^{-1} (0)} +\lambda_{\mathbb{C}}^{\alpha}+df$ for some $0<\alpha<1$, where 
\begin{align}
    \lambda_{\mathbb{C}}^{\alpha} =-\alpha ydx +(1-\alpha)xdy, \label{eq:lambdaa}
\end{align}
 and $f$ is some real-valued function with compact support.
    \end{enumerate}
    A \emph{Liouville sector} $Y$ is the closure of one connected component of $X\backslash\displaystyle\bigsqcup_{i=1}^k S_i$, where the $S_i$'s are all disjoint sectorial hypersurfaces.
\end{definition}

Alternatively, we can describe Liouville sectors in terms of pairs of Liouville manifolds \cite{eliashberg2017weinstein}. A \emph{Liouville pair} is a pair $(X,F)$ consisting of a Liouville domain $(X,\omega,Z)$ and a Liouville hypersurface $(F,\omega\vert_F, Z_F)$ in $\partial X$. We say the pair is \emph{Weinstein}, if in addition $X$ is a Weinstein domain and $F\subseteq \partial X$ is Weinstein with respect to the induced Weinstein structure of $X$. Ganatra-Pardon-Shende \cite{ganatra2020covariantly} showed that we have a canonical way to construct a Liouville sector out of a Liouville pair, so that $F$ arises as one of the fibers $(R,I)^{-1} (0)$ for some chosen $R$- and $I$-functions. Such a hypersurface $F\subseteq \partial X$, or more generally a boundary component of a Liouville sector, is usually referred to as a \emph{stop} of the sector.

\begin{remark}
This definition is different from the one in \cite{ganatra2020covariantly}, but by Proposition 2.25 in \textit{loc.cit.}, they are equivalent as definitions for Liouville sectors. In other words, it is not hard to show that our definition gives a Liouville sector, and conversely, for any given Liouville sector $X$ as in \cite{ganatra2020covariantly}, there is a convex completion of $X$ into a Liouville manifold $\widetilde{X}$, so that $\partial X\hookrightarrow\widetilde{X}$ is a sectorial hypersurface in $\widetilde{X}$.
\end{remark}

\begin{definition}[cf. \cite{ganatra2018sectorial} Definition 1.32]\label{prelim:cornered-Liouville-sectors}
    Let $S_1,\dotsb ,S_k$ be a collection of transversely intersecting sectorial hypersurfaces in $X$. The collection $\{S_1,\dotsb ,S_k\}$ is \emph{sectorial} if there is a collection of pairs of real-valued functions $\{(R_i ,I_i)\}_{1\leq i\leq k}$ together with positive real numbers $(\varepsilon_i)_{1\leq i\leq k}$ such that $R_i^{-1}(0)=S_i$ and:
    \begin{enumerate}[\indent (i)]
        \item The intersections $\displaystyle\bigcap_{i\in J} S_i\eqqcolon S_J$ for all $J\subseteq\{1,\dotsb ,k\}$ are either empty or coisotropic;
        \item When $S_J$ is non-empty, there exists a $Z$-invariant neighbourhood $\Nbd^Z (S_J)$ of $S_J$ so that the functions $R_i$ and $I_i$ for $i\in J$ define a trivial symplectic fibration
        \begin{equation}\label{eq:decomp}
        \left((R_i ,I_i)_{i\in J}\right):\Nbd^Z (S_J)\rightarrow\mathbb{C}_{-\varepsilon_i <\rRe<\varepsilon_i}^{|J|};
        \end{equation}
        \item Under the decomposition (\ref{eq:decomp}), the Liouville $1$-form $\lambda_X$ of $X$ projects to
        \begin{equation}\nonumber
        \lambda_X\vert_{\left((R_i ,I_i)_{i\in J}\right)^{-1} (0)} +\sum_{i\in J}\lambda_{\mathbb{C}}^{\alpha_i} +df
        \end{equation} 
        for some finite collection of positive numbers $0<\alpha_i<1$ and a real-valued function $f$ with compact support.
    \end{enumerate}
    Similar to Definition \ref{Prelim:Liouville-sectors}, we define a \emph{cornered Liouville sector} to be the closure of one connected component of the complement $\displaystyle X\backslash\bigcup_{i=1}^k S_i$, where $\{S_i\}_{1\leq i\leq k}$ is a transversely intersecting sectorial collection of real hypersurfaces in $X$.
\end{definition}

Given a Liouville sector $(X,\omega ,Z)$ and a submanifold $S\subseteq X$ which is cylindrical near infinity, we define the \emph{boundary at infinity} of $S$ to be the manifold 
\[
    \partial_{\infty} S\coloneqq\left\{[x]\middle\vert x\in S\setminus X_0,\ x\sim y\Leftrightarrow\exists t\in\mathbb{R},\ \varphi^t_Z (x)=y\right\},
\]
where $X_0\hookrightarrow X$ is some subdomain such that $Z$ has no zeroes outside $X_0$ (c.f. Definition \ref{Prelim:Liouville-manifolds}) and $\varphi_Z^t$ is the time-$t$ flow of the Liouville vector field $Z$. In particular, $\partial_{\infty} X$ is a contact manifold with boundary.

\subsection{Wrapped Fukaya Categories.}\label{section:pr-wrap}
In this subsection, we briefly recall the definition of wrapped Fukaya categories associated to a cornered Liouville sector, following \cite{ganatra2020covariantly} and \cite{ganatra2018sectorial}. Given a Liouville sector $(X,\omega ,Z)$, recall that a \emph{Lagrangian submanifold} of $X$ is a submanifold $L\subseteq X$ of half dimension such that $\omega\vert_L=0$.

\begin{definition}
    A Lagrangian submanifold $L\subseteq X$ is \emph{exact} if the restrction of the primitive $\lambda\vert_L$ is exact on $L$. It is \emph{cylindrical near infinity} if $Z$ is tangent to $L$ near infinity.
\end{definition}

To any two transversely intersecting exact cylindrical Lagrangian submanifolds $L, K\subseteq X$ equipped with $\mathit{Spin}$ structures, we can associate a $\mathbb{Z}$-graded free $\mathbb{Z}$-module called the Floer cochain complex
\[
    \mathit{CF}^{\ast} (L,K)\coloneqq\bigoplus_{p\in L\cap K}\mathfrak{o}_{p},
\]
where $\mathfrak{o}_{p}$ is the orientation line associated to $p\in L\cap K$.  $\mathit{CF}^{\ast} (L,K)$ admits a differential defined by an oriented count of holomorphic strips with boundary on $L\cap K$, whose cohomology gives the Lagrangian Floer cohomology group $\mathit{HF}^{\ast} (L,K)$.

\begin{definition}
An isotopy $\{L_t\}$ of exact cylindrical Lagrangian submanifolds $L_t\subseteq X$ is called \emph{positive} if for some, equivalently any, contact form $\alpha$ on $\partial_{\infty} X$, we have $\displaystyle\alpha\left(\partial_t\partial_{\infty} L_t\right)>0$.
\end{definition}

A positive isotopy $(L_0\leadsto L_1)\coloneqq\{L_t\}_{0\leq t\leq 1}$ of Lagrangian submanifolds $L_t$ such that $L_0, L_1$ are transversal to $K$ induces a chain map
\[
    \mathit{CF}^{\ast} (L_0 ,K)\to\mathit{CF}^{\ast} (L_1,K)
\]
between Floer cochain complexes. We define the \emph{wrapped Floer cohomology} associated to $L,K$ by the homotopy colimit over positive isotopies
\[
    \mathit{HW}^{\ast} (L,K)\coloneqq\varinjlim_{L\leadsto L'} \mathit{HF}^{\ast} (L',K).
\]

Suppose we have a sequence of positive isotopies $L_0 \leadsto L_1 \leadsto \dots$ and let $(L_t)_{t \in \mathbb{R}_{\ge 0}}$ be the concatenation of them.
Ganatra-Pardon-Shende \cite[Lemma 3.29]{ganatra2020covariantly} show that the sequence is cofinal if 
\begin{align}\label{eq:cofinal}
\int_0^{\infty} \min_{\partial_{\infty} L_t} \alpha\left(\partial_t\partial_{\infty} L_t\right)=\infty.
\end{align}
In other words, when \eqref{eq:cofinal} is satisfied, we can compute the wrapped Floer cohomology as
\[
    \mathit{HW}^{\ast} (L_0,K) = \varinjlim_{n} \mathit{HF}^{\ast} (L_n,K).
\]

For a smooth function $H:[0,1] \times X\to\mathbb{R}$, we denote $H(t, \cdot)$ by $H_t$.
The \emph{Hamiltonian vector field} $X_H=(X_{H_t})_{t \in [0,1]}$ is the unique $t$-dependent vector field on $X$ so that $\omega (-,X_{H_t}) =dH_t(-)$. We denote the time $1$ map of the flow by $\phi_H$. 
As positive isotopies are in particular exact, for each $n \in \mathbb{N}$, we can write $L_n$ as $\phi_{H_n}(L_0)$ for some Hamiltonian functions $H_n \in C^{\infty}([0,1] \times X)$.
The generators of $\mathit{CF}^{\ast} (L_n,K)$ can then be identified with time-$1$ Hamiltonian chords of $X_{H_n}$ from $L_0$ to $K$.
Moreover, it is well-known that for suitably defined $\omega$-tamed almost complex structures, there is a bijective correspondence between finite energy pseudo-holomorphic strips with boundary on $L_n, K$ and finite energy solutions of the $X_{H_n}$-perturbed Cauchy-Riemann equation, so one can define a chain complex  $\mathit{CF}^{\ast} (L_0,K;H_n)$ using Hamiltonian chords as generators and there is a canonical isomorphism
$\mathit{CF}^{\ast} (L_n,K) \cong \mathit{CF}^{\ast} (L_0,K;H_n)$.
We will use these identifications interchangably when we compute the wrapped Fukaya cohomology (see Section \ref{sec:computation}).

The \emph{wrapped Fukaya category} $\mathcal{W}(X)$ of a Liouville sector $X$ is an $A_{\infty}$-category first introduced by Auroux \cite{Auroux-H}, and later rigorously defined in \cite[Definition 2.20, 2.32]{ganatra2018sectorial} and \cite{ZS}. The objects of $\mathcal{W}(X)$ are exact cylindrical Lagrangian submanifolds $L\subseteq X$ equipped with $\mathit{Spin}$ structures and the morphism space between $L$ and $K$ is the localization of 
$\mathit{CF}^{\ast} (L,K)$ along continuation elements. In particular, the cohomology of the morphism space is $\mathit{HW}^{\ast} (L,K)$.

The $A_\infty$-category $\mathcal{W}(X)$ becomes easier to study from an algebraic perspective if we pass to its enlargement $\mathcal{W}(X)^\mathit{perf}$, the $A_\infty$-category of perfect modules over $\mathcal{W}(X)$. $\mathcal{W}(X)^\mathit{perf}$ is a triangulated $A_\infty$-category and there is a fully faithful embedding $\mathcal{W}(X)\hookrightarrow\mathcal{W}(X)^\mathit{perf}$. We write $D^\mathit{perf}\mathcal{W}(X)=H^0\left(\mathcal{W}(X)^\mathit{perf}\right)$ for the derived category of $\mathcal{W}(X)$. This is a genuine triangulated category.

\subsection{Generation of Wrapped Fukaya Categories.} 
In this section we recall the generation results for wrapped Fukaya categories of Weinstein sectors proved in \cite{chantraine2017geometric,ganatra2018sectorial}. A Liouville sector $(X,\omega ,Z)$ is said to be \emph{Weinstein} if the convex completion of $X$ is Weinstein and there exists a symplectic decomposition $\Nbd^Z_X (\partial X)\simeq F\times T^{\ast} (-\varepsilon,0]$ such that $F$ is a Weinstein hypersurface in $\partial X$. Similarly, a cornered Liouville sector $(X,\omega ,Z)$ is Weinstein if each sectorial corner $\partial_I X$, which is the intersection of $|I|=k$ sectorial boundaries, decomposes into a product of $T^{\ast} (-\varepsilon ,0]^k$ and a Weinstein submanifold $F_I\subset\partial_IX$. 

The wrapped Fukaya categories of Weinstein manifolds behave nicely since the Lagrangian cocores provide a natural collection of generating objects.

\begin{proposition} [\cite{cieliebak2012stein}, Proposition 11.9]
If $(X,\omega ,Z,\phi)$ is a Weinstein manifold, then all the unstable manifolds of the pair $(\phi ,Z)$ are coisotropic. In particular, if the index of a critical point of $\phi$ is $n=\frac{1}{2}\dim X$, then a corresponding unstable manifold is Lagrangian.
\end{proposition}

\begin{definition}
We call an unstable manifold of a Weinstein manifold $(X^{2n},\omega,Z,\phi)$ with index $n$ a \emph{cocore} of $X$.
\end{definition}

\begin{theorem}[\cite{chantraine2017geometric,ganatra2018sectorial}]
The wrapped Fukaya category of a Weinstein manifold $(X,\omega ,Z,\phi)$ is generated by its cocores.
\end{theorem}

In particular, wrapped Fukaya categories of Weinstein manifolds are finitely generated.

We then consider the more general case of a cornered Weinstein sector. In this case, each corner stratum is a product $\Nbd^Z_X (\partial_I X)\simeq F_I\times T^{\ast} [-\varepsilon ,0]^{\vert I\vert}$ for some $I\subseteq\{1,\dotsb ,n\}$, hence for $L\subseteq F_I$ a cocore of $F_I$ and $\gamma\subseteq T^{\ast} [-\varepsilon ,0]^{\vert I\vert}$ a product of cotangent fibers, we get an object $L\times\gamma$ in the wrapped Fukaya category of $X$, which we call a \emph{linking disk} associated to the corner $\partial_I X$.

\begin{theorem}[\cite{chantraine2017geometric,ganatra2018sectorial}]\label{Prelim:Generation-for-Weinstein-sectors}
 The wrapped Fukaya category of a cornered Weinstein sector $(X,\omega ,Z,\phi)$ is generated by its cocores and linking disks associated to all its corner strata.
\end{theorem}

We review some examples that will become the building blocks in the proof of our generation result (cf. Proposition \ref{sector:GenerationResult}).

\begin{example}\label{Prelim:generation-cotangent-fiber}
In \cite{abouzaid2011cotangent}, the author proves that for the cotangent bundle $T^{\ast}Q$ over a compact manifold $Q$ which is $\mathit{Spin}$, the wrapped Fukaya category $\mathcal{W}(T^\ast Q)$ is generated by a single cotangent fiber. This can be generalized to the case when the manifold $Q$ is compact with corners, in which case the cotangent bundle of $Q$ is a cornered Weinstein sector. In particular, consider the cotangent bundle of an interval, as drawn in Figure \ref{Prelim:cotangent-bundle-of-interval}.

    \begin{figure}[ht]
        \centering
        \begin{tikzpicture}[scale=0.5]
            \draw [red, thick] (-3,1) -- (3,1);
            \draw [red, thick] (-3,-1) -- (3,-1);
            \draw [thick] (0,1) node [above] {$0$} -- (0,-1) node [below] {$1$};
            \draw [green, thick] (-3,0) -- (3,0);
        \end{tikzpicture}
        \caption{The cotangent bundle of $[0,1]$.\label{Prelim:cotangent-bundle-of-interval}}
    \end{figure}
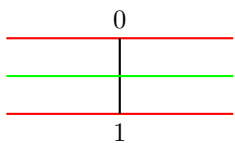

The red boundary components are sectorial boundaries, and the wrapped Fukaya category $\mathcal{W} (T^{\ast} [0,1])$ is generated by the cotangent fiber $L\coloneqq T^{\ast}_0 [0,1]$, which is colored in green in the picture. From this figure, it is not hard to see that $\mathit{CW}^{\ast} (L,L)\cong\mathbb{Z}(0)$, therefore $\mathcal{W} (T^{\ast} [0,1])\simeq\Mod_{\mathbb{Z}}$ is quasi-equivalent to the category of modules over $\mathbb{Z}$. By K\"unneth theorem, it follows that for any cornered Weinstein sector $X$, we have $\mathcal{W} (X\times T^{\ast} [0,1])\simeq\mathcal{W} (X)$.
\end{example}

\begin{example}\label{Prelim:generation-algebraic-torus}
Let $X=\left(\mathbb{C}^{\ast}\right)^d$ be an algebraic torus, which we can regard symplectically as the cotangent bundle $T^{\ast} T^d$ of the real torus $T^d$. By Example \ref{Prelim:generation-cotangent-fiber}, we know that $\mathcal{W} (T^{\ast} T^d)$ is generated by one cotangent fiber, which corresponds to one of the $2^d$ quadrants of the real locus of $(\mathbb{C}^{\ast})^d$. Cutting $(\mathbb{C}^{\ast})^d$ by sectorial hypersurfaces simply adds linking disks associated to these sectorial hypersurfaces (and their intersections) to the generating set.
\end{example}

It is a convenient viewpoint to regard linking disks as generalizations of cocores. To make it precise, we introduce the notion of a relative core.

\begin{definition}
Let $(X,\omega ,Z,\phi )$ be a Weinstein manifold, then the union of the stable manifolds of all critical points of $\phi$ is called the \emph{core} of $X$. For a Weinstein pair $(X,F)$, the \emph{relative core} of $(X,F)$ is the union $\mathfrak{c}_X\cup(\mathfrak{c}_F\times\mathbb{R})$, where $\mathfrak{c}_F$ is the core of $F$ and $\mathfrak{c}_F\times\mathbb{R}\subset X$ consists of points whose positive Liouville flow converge to $\mathfrak{c}_F$.  
\end{definition}

\begin{definition}[\cite{ganatra2018sectorial}]
Let $\mathfrak{c}_{X,F}$ be the relative core of a Weinstein sector $(X,F)$, a \emph{generalized cocore} of $(X,F)$ associated to a component of $\mathfrak{c}_{X,F}$ is a cylindrical exact Lagrangian submanifold intersecting $\mathfrak{c}_{X,F}$ transversely at a single point in this component.
\end{definition}

Theorem \ref{Prelim:Generation-for-Weinstein-sectors} remains valid for generalized cocores: the wrapped Fukaya category of a cornered Weinstein sector $(X,\omega ,Z,\phi)$ is generated by a collection of generalized cocores, one for each component of the relative core.

In particular, instead of picking the standard generator of $T^{\ast} [0,1]$, we can pick any curve $\gamma$ in $T^{\ast} [0,1]$ with a single transverse intersection with the zero section, and the product $\gamma\times L$ for any generating Lagrangian $L$ in $\mathcal{W} (F)$ will give us a linking disk in $\mathcal{W} (X)$ to associated to $F$. Figure \ref{Prelim:generalized-cocore} shows one possible choice of a generalized cocore in $T^\ast[0,1]$.

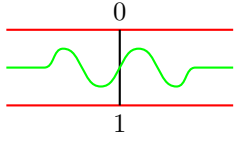
\begin{figure}[ht]
    \centering
    \begin{tikzpicture}[scale=0.5,domain=-3:3,samples=200]
        \draw [red, thick] (-3,1) -- (3,1);
        \draw [red, thick] (-3,-1) -- (3,-1);
        \draw [thick] (0,1) node [above] {$0$} -- (0,-1) node [below] {$1$};
        \draw [green, thick] (-3,0) -- (-2,0) to [in=180,out=0] (-1.5, 0.5) to [in=120,out=0] (-1,0) to [in=180,out=-60] (-0.5,-0.5) to [in=-120,out=0] (0,0) to [in=180,out=60] (0.5,0.5) to [in=120,out=0] (1,0) to [in=180,out=-60] (1.5,-0.5) to [in=180,out=0] (2,0) to (3,0);
    \end{tikzpicture}
    \caption{An example of generalized cocore.}
\label{Prelim:generalized-cocore}
\end{figure}

\subsection{Stopped Inclusion and the Viterbo Restriction.} Let $(X,\lambda)$ and $(X',\lambda')$ be two Liouville sectors. A  \emph{(stopped) inclusion} $i\colon X'\hookrightarrow X$ is a proper map that is a diffeomorphism onto its image, such that $i^{\ast}\lambda =\lambda' +df$ for some compactly supported function $f:X'\rightarrow\mathbb{R}$. Ganatra-Pardon-Shende \cite{ganatra2020covariantly} showed that an inclusion of Liouville sectors induces an $A_\infty$-functor between the wrapped Fukaya categories
\begin{equation}\nonumber
\mathcal{W} (X')\to\mathcal{W} (X),
\end{equation}
which we call a \emph{stopped inclusion functor}. On the other hand, given an embedding of Liouville domains $X_0\hookrightarrow X_1$, one would expect that there is a \emph{Viterbo restriction functor}
\[
\mathcal{W} (X_1)\to\mathcal{W} (X_0)
\]
between wrapped Fukaya categories, which is contravariant with respect to Liouville embedding. The geometric construction, under additional assumptions on cylindrical exact Lagrangians, has been carried out by Abouzaid-Seidel \cite{abouzaid2010open}. Algebraically, the existence of such a restriction functor for $X_0$ a Weinstein domain is proved by Ganatra-Pardon-Shende \cite{ganatra2018sectorial}. The proof in \cite[Section 11.1]{ganatra2018sectorial} suggests that such an $A_\infty$-functor also exists for embeddings of Weinstein sectors $W_0\hookrightarrow W_1$, provided that the Weinstein embedding is a diffeomorphism on connected components of sectorial boundaries. We record this fact in the following lemma.

\begin{lemma}
Let $X_0\hookrightarrow X_1$ be a Weinstein embedding of Weinstein sectors that is a diffeomorphism on the boundary, then there exists a Viterbo restriction functor $\mathcal{W} (X_1)\to\mathcal{W} (X_0)$.
\end{lemma}

\subsection{The Orlov Functor.\label{prelim:Orlov-functor}} Recall from Definition \ref{Prelim:Liouville-sectors} that given a Liouville sector $(X,\omega ,Z)$, the boundary $\partial X$ admits a $Z$-invariant neighbourhood $\Nbd^Z (\partial X)\cong F\times T^{\ast} [-\varepsilon ,0]$ for some $\varepsilon >0$, where $F\subset\partial X$ is a Liouville hypersurface. When $F$ happens to be Weinstein, Theorem \ref{Prelim:Generation-for-Weinstein-sectors} tells us that the wrapped Fukaya category $\mathcal{W} (F)$ is generated by its cocores. Example \ref{Prelim:generation-cotangent-fiber} implies that the wrapped Fukaya category of $\Nbd^Z (\partial X)$, regarded as a Weinstein sector, is quasi-equivalent to $\mathcal{W} (F)$. Together with the stopped inclusion $\Nbd^Z (F)\hookrightarrow X$, we obtain an $A_\infty$-functor 
\begin{equation*}
\mathsf{Or}:\mathcal{W} (F)\rightarrow\mathcal{W} (X), 
\end{equation*}
which is usually known as the \emph{Orlov functor}. Note that the images of the Orlov functor is exactly the linking disks associated to the stop $F$.

\subsection{Sectorial Descent.}

Suppose a Liouville sector $X$ is covered by the other two Liouville sectors $X_1$ and $X_2$ such that the intersection $X_{12}:=X_1\cap X_2$ becomes a sectorial hypersurface of both $X_1$ and $X_2$ (see Definition \ref{prelim:cornered-Liouville-sectors}). Then the stopped inclusion functor gives us a diagram of wrapped Fukaya categories 
\begin{equation}\label{eq:diag}
    \begin{tikzcd}
        &\mathcal{W} (X_1)\arrow[rd]&\\ 
        \mathcal{W} (X_{12})\arrow[ru]\arrow[rd]&&\mathcal{W} (X)\\ 
        &\mathcal{W} (X_2)\arrow[ru]
    \end{tikzcd}.
\end{equation}
In general, this diagram is an almost homotopy pushout, and if all the sectors are Weinstein, then it is a genuine homotopy pushout \cite[Theorem 1.28]{ganatra2018sectorial}. More generally, suppose that $X$ can be covered by a collection of cornered Liouville sectors, i.e. $\displaystyle X=\bigcup_{i=1}^k X_i$, the intersection poset of the cornered Liouville sectors $\{X_i\}_{i=1}^k$ gives a diagram of wrapped Fukaya categories, leading to the following is a generalization of (\ref{eq:diag}). 

\begin{theorem}[Sectorial Descent \cite{ganatra2018sectorial}, Theorem 1.35]\label{Prelim:sectorial-descent}
 For any Weinstein sectorial covering $X_1,\dotsb ,X_k$ of a Liouville sector $X$, the induced functor 
    \[
        \hocolim_{\emptyset\neq I\subseteq\{1,\dotsb ,k\}}\mathcal{W}\left(\bigcap_{i\in I} X_i\right)\xrightarrow{\simeq}\mathcal{W} (X) 
    \]
    is a pre-triangulated equivalence.
\end{theorem}

    \section{Symplectic Geometry of Hyperplane complements and Generating Lagrangians}\label{sec:symplecticgeometry of hypargg}

In this section, we first introduce our main object of study, a polarized hyperplane arrangement $\mathbb{V}=(V, \eta, \xi)$ (see Definition \ref{def:hyparr}) and establish some basic geometric properties for the complement $M(\mathbb{V})$ of the complexified hyperplanes in $\mathbb{V}$. Then we apply Theorem \ref{Prelim:sectorial-descent} together with the generation results reviewed in Examples \ref{Prelim:generation-cotangent-fiber} and \ref{Prelim:generation-algebraic-torus} to get a generating set of Lagrangians for the partially wrapped Fukaya category $\mathcal{W} (M(\mathbb{V}),\xi)$.
%

\subsection{Hyperplane Arrangements.}\label{sec:hyparr} Let $(V,\eta)$ be a pair of $d$-dimensional subspace $V \subset \mathbb{R}^n$ and an element  $\eta \in \mathbb{R}^n/V$. This determines a hyperplane arrangement of $n$ hyperplanes in $V \cong \mathbb{R}^d$. If we present $V$ as the image of a linear map $\mathbb{R}^d \to \mathbb{R}^n$ induced by an $n \times d$ matrix $A=(a_{ij})$ and $w=(w_1, \dots, w_n) \in \R^n$ respresents $\eta$, then the hyperplanes in the arragement are given by 
\begin{equation}\label{eq:description of hyperplane}
H_{\mathbb{R},i}=\left\{s=(s_1, \dots, s_d) \in \R^d\ \middle|\ \sum_{j=1}^da_{ij}s_j+w_i=0\right\},
\end{equation}
for $i=1, \dots, n$. 
We assume that $(V,\eta)$ is chosen such that $H_{\mathbb{R},i}\neq\emptyset$ for all $i$ so none of the hyperplane is redundant.

The positive half-spaces of $\R^d$ induce co-orientations on the hyperplanes so that we can assign a region (possibly empty) of the arragement a sequence $\alpha$ of signs $\{+,-\}$ of length $n$. More precisely, given $\alpha=(\alpha(1), \dots, \alpha(n)) \in \{+,-\}^n$, the corresponding region (also called chamber) $\Delta_\alpha \subset V + \eta$ is the set of points $s$ with 
\begin{equation*}
\alpha(i)\left(w_i+\sum_{j=1}^da_{ij}s_j \right) \geq 0, \quad 1 \leq i \leq n.
\end{equation*} 
Such a sign sequence $\alpha$ is called \emph{feasible}, if $\Delta_\alpha \neq \emptyset$. We denote the set of feasible sign sequences for $\mathbb{V}$ by $\mathscr{F}(\mathbb{V})$. Given two sign sequences $\alpha$ and $\beta$, define
\begin{equation}\label{eq:d}
d_{\alpha\beta}:=\#\left\{1\leq i\leq n|\alpha(i)\neq\beta(i)\right\}
\end{equation}
to be the number of sign changes required to turn $\alpha$ into $\beta$. Due to the bijective correspondence, sign sequences and chambers will often be mentioned interchangeably.

\begin{definition}\label{def:pol.hyp.arr}
Let $\mathbb{V}=(V, \eta)$ be an arrangement of $n$ hyperplanes as above. A \emph{polarization} is an element $\xi \in V^{\ast} \cong (\mathbb{R}^n)^\ast/V^{\perp}$ such that each element of $(\mathbb{R}^n)^\ast$ representing $\xi$ has at least $d$ non-zero entries. We call the triple $\mathbb{V}=(V,\eta, \xi)$ a \emph{polarized hyperplane arrangement} of $n$ hyperplanes in $V$. 
\end{definition}

\begin{remark}
The condition that $\xi$ has at least $d$ non-zero entries is added here to ensure that $\xi$ is generic enough for our later purposes. This is usually not required in the literature.
\end{remark}

Choose a lift $\tilde{\xi} \in (\mathbb{R}^n)^\ast$ of $\xi$. We say a sign sequence $\alpha \in \{+,-\}^n$ is \emph{bounded} if the affine linear functional $\tilde{\xi}$ on $V+\eta$ is bounded above on $\Delta_\alpha$. Note that this is independent of the choice of the lift $\tilde{\xi}$. Denote the set of bounded sign sequences for $\mathbb{V}$ by $\mathscr{B}(\mathbb{V})$. Let $\mathscr{P}(\mathbb{V}):=\mathscr{B}(\mathbb{V})\cap \mathscr{F}(\mathbb{V})$ be the set of \emph{bounded and feasible} sign sequences.
%
%

\begin{definition}\label{definition:simple}
A hyperplane arrangement $(V,\eta)$ is called \emph{simple} if for every subset of $k$ hyperplanes $H_{\mathbb{R},i_1}, \dots, H_{\mathbb{R},i_k}$, the intersection $\bigcap_{l=1}^k H_{\mathbb{R},i_l}$ is either empty or has codimension $k$.
\end{definition}

In fact, this is equivalent to saying that the intersection of hyperplanes is either transverse or empty (when they are parallel). The following combinatorial lemma will play a crucial role in the proof of our main result. 

\begin{lemma}\label{Prelim:NumberOfCrossings}
Let $\mathbb{V}$ be a simple hyperplane arrangement equipped with a generic polarization $\xi$. Let $\mathcal{H}=\{H_{\mathbb{R},1}, \dots, H_{\mathbb{R},n}\}$ be the corresponding collection of hyperplanes in $\R^d$. Then the number of bounded and feasible chambers is equal to the number of $0$-dimensional strata in the union of all hyperplanes $\bigcup_{i=1}^n H_{\mathbb{R},i}$.
\end{lemma}
\begin{proof}
The proof is based on double induction on the pair $(d,n)\in\mathbb{N}^2$. Denote $N_\mathbb{V}$ the number of bounded-feasible chambers of $\mathbb{V}$. Due to the genericity of $\xi$, it is straightforward that $N_\mathbb{V}=0$ if and only if there is no $0$-dimensional strata in $\bigcup_{i=1}^n H_{\mathbb{R},i}$. Also, when $d=1$, it is clear that $N_\mathbb{V}=n$. 
	%
	
We first perform induction on $d$. Suppose that the claim holds when $d=r$. Consider a polarized arrangement $\mathbb{V}$ of $n\geq r$ hyperplanes in $\mathbb{R}^{r+1}$ and apply the induction on $n$. The initial step $n=r$ is obvious. Suppose that the claim holds for $n =k>r+1$. We add a new hyperplane $H_\mathbb{R}$ to $\mathcal{H}$ which is in general position, so that the restrictions of the hyperplanes in $\mathcal{H}$ to $H_\mathbb{R}$ is still a hyperplane arrangement in $H_\mathbb{R}\cong\mathbb{R}^r$. By induction hypothesis, we know that the number of bounded and feasible chambers on the induced hyperplane arrangement in $H_\mathbb{R}$ is the same as the number of 0-dimensional strata. Moreover, the simplicity assumption implies that each such chamber in $H_\mathbb{R}$ either divides one of the bounded and feasible chambers in $\mathbb{V}$ or adds a new bounded and feasible chamber to $\mathbb{V}$. This implies that after adding $H_\mathbb{R}$, the number of bounded and feasible chambers in the hyperplane arrangement is increased by the number of 0-dimensional strata in $H_\mathbb{R}$ created by the intersection $H_\mathbb{R}\cap \left(\bigcup_{i=1}^kH_{\mathbb{R},i}\right)$. 
\end{proof}

\begin{definition}\label{def:hyparr}
Let $\mathbb{V}=(V, \eta, \xi)$ be a polarized hyperplane arrangement. The complexified polarized hyperplane arrangement associated to $\mathbb{V}$ is a linear extension of $\mathbb{V}$ over $\C$, denoted by  $\mathbb{V}_\C=(V_\C, \eta_\C, \xi_\C)$. If it is clear from the context, we simply write it as $\mathbb{V}$. 
\end{definition}

For example, if we write one of the hyperplanes $H_{\mathbb{R},i} \subset \R^d$ as in (\ref{eq:description of hyperplane}), then its complexification $H_i\subset\mathbb{C}^d$ is given by 
\begin{equation}\label{eq:hyp}
H_i=\left\{z \in \C^d\ \middle|\ \sum_{j=1}^da_{ij}z_j+w_i=0\right\},
\end{equation}
where $a_{ij}, w_j\in\mathbb{R}$.

\subsection{Essential Geometry.}\label{ss:essential} 
In this subsection, we describe the Weinstein structure of $M(\mathbb{V})$ and show that $M(\mathbb{V})$ is of finite type (Corollary \ref{generation:finiteness-critical-points}).

Let $\ell_i(z)$ be the defining equation of the (complex) hyperplane $H_i$ in the arrangement $\mathbb{V}$ and $\xi (z)=\xi \cdot z$.  We assume that the polarization $\xi$ satisfies the following properties:
\begin{enumerate}[\indent (i)]
	\item $\xi$ does not belong to the span of the normal vectors of any $k$ hyperplanes in $\mathbb{V}$ if $k \leq d-1$;
	\item no two $0$-dimensional strata in the intersections of the complex hyperplanes $\{H_i\}_{i=1}^n$ lie in the same level set $\mathrm{Re}(\xi)^{-1}(c)$ for some $c \in \mathbb{R}$. 
\end{enumerate}

\noindent The set of polarizations satisfying these two properties are generic in the space of all linear forms on $V^\ast$. Consider the affine canonical embedding  $\iota:\mathbb{C}^d \to \mathbb{C}^n$ given by $\iota(z)=(\ell_1(z), \dots, \ell_n(z))$. This induces an embedding of the complement $M(\mathbb{V}):=\mathbb{C}^d \setminus \bigcup_{i=1}^nH_{i}$ into $ (\mathbb{C}^\ast)^n$, which we will denote by $\iota:M(\mathbb{V}) \to (\mathbb{C}^\ast)^n$. 

Next, we equip $M(\mathbb{V})$ with a symplectic form $\omega$ as follows: First, let $\omega_\phi$ e the symplectic form on $(\mathbb{C}^\ast)^n$ defined by the pullback of the standard symplectic form on $\mathbb{C}^{2n}$ via the embedding $i\colon (\mathbb{C}^\ast)^n \hookrightarrow \mathbb{C}^{2n}$ given by $i(z_1,\dots,z_n)=(z_1,1/z_1,\dots,z_n,1/z_n)$. We have
\begin{align*}
\omega_{\phi}&=2\sqrt{-1}\sum_{i=1}^n\left(d z_i\wedge d\bar{z}_i +d\left(\frac{1}{z_i}\right) \wedge d\left(\frac{1}{\bar{z}_i}\right)\right)  =2\sqrt{-1}\sum_{i=1}^n\left(1+\frac{1}{\vert z_i\vert^4}\right) d z_i\wedge d\bar{z}_i \nonumber\\
&=4\sum_{i=1}^n\left(1+\frac{1}{(x_i^2+y_i^2)^2}\right) d x_i\wedge d y_i.
\end{align*}
Equivalently, we equip $(\mathbb{C}^\ast)^n$ with the K\"{a}hler potential $\phi=\sum_{i=1}^n \left( |z_i|^2+\frac{1}{|z_i|^2}\right)$ so that
\begin{align*}
-d^\mathbb{C}\phi&=2\sum_{i=1}^n\left(1-\frac{1}{(x_i^2+y_i^2)^2}\right) x_i d y_i-2\sum_{i=1}^n\left(1-\frac{1}{(x_i^2+y_i^2)^2}\right) y_i d x_i.
\end{align*}
and $\omega_{\phi}=-dd^\mathbb{C} \phi$.
We then take $\omega$ to be the pullback of the form $\omega_\phi$ under the embedding $\iota:M(\mathbb{V}) \hookrightarrow (\C^\ast)^n$, which gives
\[\omega:=\iota^*\omega_{\phi}=4\sum_{i=1}^n\left(1+\frac{1}{\vert \ell_i\vert^4}\right) d \mathrm{Re}(\ell_i)\wedge d \mathrm{Im}(\ell_i).
\]
Notice that $\omega$ depends on $M(\mathbb{V})$ as well as the choice of the embedding $\iota:\mathbb{C}^d \to \mathbb{C}^n$. However, different choices yield deformation-equivalent Weinstein structures.

\begin{lemma}\label{l:rescaling}
For any $\vec{c}=(c_1,\dots,c_n) \in \mathbb{R}_{>0}^n$, the rescaled embedding $\iota_{\vec{c}}:=\vec{c}\iota=(c_1 \ell_1,\dots,c_n\ell_n): M(\mathbb{V}) \to (\mathbb{C}^\ast)^n$ gives us a new Weinstein structure $(M(\mathbb{V}), \iota_{\vec{c}}^*\phi)$ that is Weinstein deformation equivalent to $(M(\mathbb{V}),\iota^*\phi)$.
\end{lemma}

\begin{proof}
Consider an isotopy of Stein embeddings $\iota_t\colon M(\mathbb{V})\to\mathbb{C}^n$ via 
    \[
        \iota_t (z)=t\vec{c}\iota(z)+(1-t)\iota(z)
    \]
    for $ t \in [0,1]$. Then we have $\iota_t^{\ast}\phi (z)=\phi (t\vec{c}\iota(z)+(1-t)\iota(z))$. 
It is straightforward to compute the differential and easy to see that the critical points of $d\phi_t (z)$, for all $0\leq t\leq 1$, lie in a compact subset (by rescaling the previous critical points). Therefore $(M(\mathbb{V}),\iota^*\phi)$ and $\left(M(\mathbb{V}),\iota_{\vec{c}}^*\phi\right)$ are Weinstein deformation equivalent.
\end{proof}

From now on, if the context is clear, we will abuse notation by denoting the restriction of $\phi$ to $M(\mathbb{V})$ simply as $\phi$. 

Let $J \subset \{1, \dots, n\}$ be an index set such that $\bigcap_{i \in J} H_i \neq \emptyset$. Since all the equation $\ell_i$'s are linear, near the intersection $\bigcap_{i \in J} H_i$, the symplectic form $\omega$ can be written as the sum of the dominant terms and bounded terms:
\[\omega=4\sum_{i\in J}\left(1+\frac{1}{\vert \ell_i\vert^4}\right) d \mathrm{Re}(\ell_i)\wedge d \mathrm{Im}(\ell_i) + 4\sum_{i\notin J}\left(1+\frac{1}{\vert \ell_i\vert^4}\right) d \mathrm{Re}(\ell_i)\wedge d \mathrm{Im}(\ell_i).
\]

Geometrically, the symplectic form $\omega$ asymptotically behaves like the standard symplectic form near $\bigcap_{i \in J} H_i \neq \emptyset$. 
The following lemma gives precise formulae for this fact.

\begin{lemma}\label{sector:Standardization}
Let $J\subset\{1, \dots, n\}$ with $\vert J\vert=d$. Enumerate the elements in $J$ by $1,\dots,d$.
Consider the linear change of coordinates $\Phi: \mathbb{C}^d \to \mathbb{C}^d$ which pulls $\ell_i$ back to $z_i$ for $i \in J$. Write 
\begin{align*}
\ell_j =\displaystyle\sum_{i\in J} a_{ij}z_i,
\end{align*}
where the coefficients $a_{ij} \in \mathbb{R}$ are as in (\ref{eq:hyp}). 

Then $M(\Phi^{-1}(\mathbb{V}))=\Phi^{-1}\left(M(\mathbb{V})\right)$ and
we have the following formulae:
\begin{equation}\label{eq:sympform}
\begin{split}
\Phi^{\ast}\omega&=4\sum_{i\in J}\left(1+\frac{1}{\vert z_i\vert^4} +\sum_{j\not\in J}a_{ij}^2\left(1+\frac{1}{\vert\ell_j\vert^4}\right)\right)dx_i\wedge dy_i \\
&+4\sum_{\substack{i\neq j\\ i,j\in J}}\sum_{k\not\in J} a_{ik} a_{jk}\left(1+\frac{1}{\vert\ell_k\vert^4}\right) dx_i \wedge dy_j. 
\end{split}
\end{equation}

\begin{equation*} 
\begin{split}
\Phi^*d\phi &=\sum_{i=1}^d2\left(\left(1-\frac{1}{\vert z_i\vert^4}\right)x_i+\sum_{j\not\in J}\left(1-\frac{1}{\vert\ell_j\vert^4}\right)\rRe(\ell_j) a_{ij}\right)dx_i \\
&+\sum_{i=1}^d2\left(\left(1-\frac{1}{\vert z_i\vert^4}\right)y_i+\sum_{j\not\in J}\left(1-\frac{1}{\vert\ell_j\vert^4}\right)\rIm(\ell_j)a_{ij}\right)dy_i, 
\end{split}
\end{equation*}

\begin{equation} \label{eq:dphic}
\begin{split}
\Phi^*d^\mathbb{C}\phi&=-\sum_{i=1}^d2\left(\left(1-\frac{1}{\vert z_i\vert^4}\right)x_i+\sum_{j\not\in J}\left(1-\frac{1}{\vert\ell_j\vert^4}\right)\rRe(\ell_j) a_{ij}\right)dy_i \\
&+\sum_{i=1}^d2\left(\left(1-\frac{1}{\vert z_i\vert^4}\right)y_i+\sum_{j\not\in J}\left(1-\frac{1}{\vert\ell_j\vert^4}\right)\rIm(\ell_j)a_{ij}\right)dx_i.
\end{split}
\end{equation}
\end{lemma}

The proof of Lemma \ref{sector:Standardization} follows from a straightforward computation, so we omit the details. 
For simplicity, we will drop $\Phi^*$ from the notations if it is clear from the context. We will find it convenient to use Lemma \ref{sector:Standardization} when doing computations.

To write the formulae in the lemma above in more compact forms, let $A$ be the $(n-d)\times d$-matrix whose $(i,j)$-th entry is $\left(a_{ji}\sqrt{1+\dfrac{1}{\vert\ell_i\vert^4}}\right)$ and let $D=\left(\left(1+\dfrac{1}{\vert z_i\vert^4}\right)\delta_{ij}\right)$ be the diagonal matrix. Under the standard basis $\left\{\dfrac{\partial}{\partial\vec{x}},\dfrac{\partial}{\partial\vec{y}}\right\}$, we can write the symplectic form $\omega$ in the matrix form
\begin{align}\label{eq:symp}
    \omega=\begin{pmatrix}
        0&4(D+A^TA)\\ 
        -4(D+A^TA)&0
    \end{pmatrix},
\end{align}
where $D+A^TA$ is invertible.

Similarly, consider the matrices
    \[
        \tilde{A}=\left(a_{ji}\sqrt{1-\frac{1}{\vert\ell_i\vert^4}}\right),\quad \tilde{D}=\left(\left(1-\dfrac{1}{\vert z_i\vert^4}\right)\delta_{ij}\right),
    \]
and 
    \[
        [d^\mathbb{C}\phi]_x=2\vec{x}^T(\tilde{D}+\tilde{A}^T\tilde{A}), \quad [d^\mathbb{C}\phi]_y =2\vec{y}^T\left(\tilde{D}+\tilde{A}^T\tilde{A}\right),
    \]
then we have $d^\mathbb{C}\phi=- [d^\mathbb{C}\phi]_x d\vec{y} + [d^\mathbb{C}\phi]_y d\vec{x}$. We can also compute the Hamiltonian vector fields $X_{\rRe(\xi)}$ and $X_{\rIm(\xi)}$ associated to the real and imaginary parts of $\xi$ as follows.

\begin{lemma}
Write $\xi=\sum_{i=1}^d b_i z_i$, where $b_i \in \mathbb{R}$.
Then  $X_{\rRe(\xi)}$ is a purely imaginary vector field given by the matrix formula 
\begin{align}
X_{\rRe(\xi)} =\frac{1}{4}\left(\frac{\partial}{\partial \vec{y}}\right)^T(D+A^TA)^{-1}\vec{b}, \label{eq:Xreal}
\end{align}
where $\vec{b}$ refers to the vector $(b_1, \dots, b_d)^T$.  Similarly, $X_{\rIm(\xi)}$ is given by
\begin{align}
X_{\rIm(\xi)} =-\frac{1}{4}\left(\frac{\partial}{\partial \vec{x}}\right)^T\left(D+A^TA\right)^{-1}\vec{b}.  \label{eq:Xim}
\end{align}
As a consequence, the square of the length of the vector field $X_{\rRe(\xi)}$ is
\[
    \left\Vert X_{\rRe(\xi)}\right\Vert^2=\omega\left(X_{\rRe(\xi)}, X_{\rIm(\xi)}\right) =\frac{1}{4}\left(\vec{b}\right)^T\left(D+A^TA\right)^{-1}\vec{b}.
\]
\end{lemma}

\begin{proof}
Since $d\rRe(\xi) =\displaystyle\sum_{i=1}^d b_i dx_i$ and $d\rIm(\xi)=\displaystyle\sum_{i=1}^d b_i dy_i$, we get
\begin{equation} \nonumber
\begin{split}
b_i&=d\rRe(\xi)\left(\frac{\partial}{\partial x_i}\right)=\omega\left(\frac{\partial}{\partial x_i},X_{\rRe(\xi)}\right) \\ &=4\sum_{j=1}^d\left(1+\frac{1}{\vert z_j\vert^4}+\sum_{k\not\in I}\left(1+\frac{1}{\vert\ell_k\vert^4}\right)a_{ik} a_{jk}\right) dy_j\left(X_{\rRe(\xi)}\right),
\end{split}
\end{equation}
\begin{equation} \nonumber
\begin{split}
0&=d\rRe(\xi)\left(\frac{\partial}{\partial y_i}\right)=\omega\left(\frac{\partial}{\partial y_i}, X_{\rRe(\xi)}\right) \\
&=-4\sum_{j=1}^d\left(1+\frac{1}{\vert z_j\vert^4}+\sum_{k\not\in I}\left(1+\frac{1}{\vert\ell_k\vert^4}\right)a_{ik}a_{jk}\right)dx_j\left(X_{\rRe(\xi)}\right),
\end{split}
\end{equation}
Using the matrices $A$ and $D$, the equations above can be rewritten as 
\[
    \vec{b}=4\left(D+A^TA\right)d\vec{y}\left(X_{\rRe(\xi)}\right),\quad\vec{0}=4\left(D+A^TA\right)d\vec{x}\left(X_{\rRe(\xi)}\right).
\]
The case of $X_{\mathrm{Im}(\xi)}$ can be dealt with similarly. Using \eqref{eq:symp} and the non-degeneracy of the matrix $D+A^TA$, we conclude the result.
\end{proof}

\begin{lemma}\label{l:Iestimate_pre}
Suppose that  
\begin{align}
\bigcap_{\substack{j\in K\\ \vert K\vert=d+1}}\left\{\vert\ell_j\vert\leq 1\right\}=\emptyset. \label{eq:convercontrol}
 \end{align}
Then there are functions $C_I(z), C_R(z)>0$ and two constants $0<c_1<c_2$ such that for $\vert \rIm \xi \vert$ sufficiently large, we have $c_1<C_I(z)<c_2$ and 
    \[
        -d^\mathbb{C}\phi (X_{\rIm(\xi)})=C_I(z) \rIm (\xi).
    \]
Similarly, for $\vert \rRe(\xi)\vert$ sufficiently large, we have $c_1<C_R(z)<c_2$ and 
    \[
        -d^\mathbb{C}\phi (X_{\rRe(\xi)})=C_R(z) \rRe (\xi).
    \]
\end{lemma}

\begin{proof}
For any $J \subset \{1,\dots,n\}$ such that $|J|=d$ and $\epsilon>0$, we define
\begin{align}
M(\mathbb{V})_{J,\epsilon} =\{ z \in M(\mathbb{V})| |\ell_j(z)|>\epsilon \text{ for all }j \in J\}. \label{eq:covering}
\end{align}
By assumption \eqref{eq:convercontrol}, the open sets $\{M(\mathbb{V})_{J,1} | |J|=d\}$ form a finite cover of $M(\mathbb{V})$.
Consider a single open set $M(\mathbb{V})_{J,1}$ and let $\{\ell_j | j \in J\}$ be the coordinate functions $\{z_i|i=1,\dots,d\}$.


By  \eqref{eq:dphic}, \eqref{eq:Xreal} and \eqref{eq:Xim}, we can write $d^\mathbb{C}\phi (X_{\rRe(\xi)})$ and $d^\mathbb{C}\phi(X_{\rIm(\xi)})$ as 
    \begin{align*}
&        d^\mathbb{C}\phi (X_{\rRe(\xi)})=-\frac{1}{2}\vec{x}^T\left(\tilde{D}+\tilde{A}^T\tilde{A}\right)\left(D+A^TA\right)^{-1}\vec{b},\\
& d^\mathbb{C}\phi (X_{\rIm(\xi)})=-\frac{1}{2}\vec{y}^T\left(\tilde{D}+\tilde{A}^T\tilde{A}\right)\left(D+A^TA\right)^{-1}\vec{b}.
    \end{align*}
    When $\vert \rIm(\xi)\vert\to\infty$, the inequality $|\rIm(\xi)|=\displaystyle\left\vert\sum_{i=1}^d b_{i}y_i\right\vert\leq\left(\sum b_{i}^2\right)^{1/2}\left(\sum y_i^2\right)^{1/2}$ implies that there must be at least one coordinate $y_i$ with $y_j\to\infty$. Using the decomposition \[
        \tilde{D}+\tilde{A}^T\tilde{A} =\left(D+A^TA\right)-\left(D-\tilde{D}+A^TA-\tilde{A}^T\tilde{A}\right),
    \] we can write $-d^\mathbb{C}\phi (X_{\rIm(\xi)})$ as 
\begin{align*}
        -d^\mathbb{C}\phi (X_{\rIm\xi})&=\vec{y}^T\vec{b}-\underbrace{\vec{y}^T\left(D-\tilde{D}+A^TA-\tilde{A}^T\tilde{A}\right)\left(D+A^TA\right)^{-1}\vec{b}}_{\text{remainder term }R} =\rIm(\xi)-R.
    \end{align*}
Note that the term
\[
        \vec{y}^T\left(D-\tilde{D}+A^TA-\tilde{A}^T\tilde{A}\right)=\left(\frac{2y_i}{\vert z_i\vert^4}+\sum_{j\not\in J}\frac{2}{\vert\ell_j\vert^4}\rIm(\ell_j)a_{ij}\right)
    \] 
in the remainder matrix $R$ is bounded when all the $\vert z_i\vert$'s are bounded away from $0$ since $\vert\ell_j\vert\geq 1$ for all $j\not\in J$. Since the vector $\left(D+A^TA\right)^{-1}\vec{b}$ is also bounded, it follows that the remainder term is bounded as $\vert z_i\vert$ are bounded away from $0$. 
Let $K \subset \{1,\dots,d\}$ be a subset such that 
 $\vert z_i\vert\to 0$ for $i \in K$ and $|z_j|$ is bounded away from $0$ for $j \notin K$, we can write $D+A^TA$ as a $2 \times 2$ block matrix
	\begin{equation*}
	D+A^TA={\begin{pmatrix}
		E & F \\
		G & H 
		\end{pmatrix}}
	\end{equation*}
 where the first block $E$ records all the entries coming from $K$. By \cite[Theorem 2.1]{inversematrix2000}, the inverse $(D+A^TA)^{-1}$ is given by 
	\begin{equation}\label{eq:inverse}
	\begin{pmatrix}
	E^{-1}+E^{-1}F(H-GE^{-1}F)^{-1}GE^{-1} & -E^{-1}F(H-GE^{-1}F)^{-1} \\
	-(H-GE^{-1}F)^{-1}GE^{-1} & (H-GE^{-1}F)^{-1}
	\end{pmatrix}
	\end{equation} 
When the coordinates $(z_i)_{i \in K}$ are sufficiently small, $E^{-1}$ is close to a zero matrix and the determinants $|F|, |G|, |H|$ are bounded so $(D+A^TA)^{-1}$ is close to a $2 \times 2$ block matrix where only the bottom right block is non-zero (except when $K=\{1,\dots,d\}$, in which case the whole matrix $(D+A^TA)^{-1}$ is close to the zero matrix).
 On the other hand, when $(z_i)_{i \in K}$ are sufficiently small, the blow-up entries in $\left(D-\tilde{D}+A^TA-\tilde{A}^T\tilde{A}\right)$ are the diagonal entries $\frac{1}{\vert z_i\vert^4}$ for $i \in K$, which go to infinity at the same rate that $E^{-1}$ is going to $0$.
As a result, the terms in $\vec{y}^T\left(D-\tilde{D}+A^TA-\tilde{A}^T\tilde{A}\right)(D+A^TA)^{-1}$ are still bounded even when some $z_i$'s are approaching $0$.
 Therefore $R$ is bounded for all possible $\vert z_i\vert$'s, which implies that as $\vert \rIm(\xi)\vert\to\infty$, $-d^\mathbb{C}\phi (X_{\rIm(\xi)})\approx \rIm(\xi)$.
This proves the claim for $-d^\mathbb{C}\phi (X_{\rIm(\xi)})$. The case of $-d^\mathbb{C}\phi (X_{\rRe(\xi)})$ can be dealt with in a  similar way.
\end{proof}

\begin{corollary}\label{generation:finiteness-critical-points}
The critical points of the function $\phi\colon M(\mathbb{V})\to\mathbb{R}$ lie in a compact subset of $M(\mathbb{V})$. As a consequence, $(M(\mathbb{V}),\omega,\nabla\phi ,\phi)$ is a Weinstein manifold of finite type. 
\end{corollary}

\begin{proof}
By Lemma \ref{l:rescaling}, at the cost of replacing $\ell_i$ by $c_i\ell_i$, we can assume that \eqref{eq:convercontrol} is satisfied without changing the Weinstein deformation type.
Then by Lemma \ref{l:Iestimate_pre}, $d\phi$ is non-zero when $|\rRe(\xi)|$ or $|\rIm(\xi)|$ is large. 
This is true for any polarization $\xi$.
Since the zeros of $d\phi$ do not depend on $\xi$, we conclude that they lie in a compact set. 
\end{proof}

\subsection{Some Useful Lemmas.} Before discussing the symplectic geometry of $M(\mathbb{V})$, we collect here some lemmas that will be used later. We start with a simple fact in linear algebra.

\begin{lemma}\label{Prelim:linear-algebra-fact}
The (pointwise) inverse matrix $\omega^{-1}$ of $\omega$, as a section of $\Hom(T^*M(\mathbb{V}), TM(\mathbb{V}))$ over $M(\mathbb{V})$, extends smoothly over $\mathbb{C}^d$.
Moreover, for a point $p \in \bigcap_{i \in J} H_{i}$, the image of $\omega^{-1}|_{p}$ lies in $\bigcap_{i \in J} T_pH_{i}$.
\end{lemma}

\begin{proof}
Since the complex structure $J_\mathbb{V}$ on $M(\mathbb{V})$ is $\omega$-compatible, we have $-g \circ J_\mathbb{V}=\omega:TM(\mathbb{V}) \to T^*M(\mathbb{V})$, where $g$ is the Riemannian metric.
It follows that $\omega^{-1}=J_\mathbb{V} \circ g^{-1}:  T^*M(\mathbb{V}) \to TM(\mathbb{V})$.
Since $J_\mathbb{V}$ extends smoothly over $\mathbb{C}^d$, it suffices to check that  $g^{-1}$ extends smoothly over $\mathbb{C}^d$.

Let $p \in \bigcap_{i \in J} H_{i}$ and $p \notin H_{i}$ if $i \notin J$. 
In a neighborhood of $p$, we write $g=(g_{ij})$ as a  $2 \times 2$ block matrix
	\begin{equation*}
	R={\begin{pmatrix}
		E & F \\
		G & H 
		\end{pmatrix}}
	\end{equation*}
such that $E=(g_{ij})_{i,j \in J}$. Note that $R$ is only well-defined on the complement $M(\mathbb{V})$ of the hyperplanes. By \cite[Theorem 2.1]{inversematrix2000}, $R^{-1}$ is given by \eqref{eq:inverse}.
	
On the other hand, we know that as we approach the point $p$ from the complement $M(\mathbb{V})$,  $E^{-1}$ converges to the zero matrix and the determinants $|F|, |G| \text{ and } |H|$ are bounded. Thus the inverse matrix $R^{-1}$ converges to $\begin{pmatrix}
	0 & 0 \\
	0 & H^{-1} 
	\end{pmatrix}$,
which clearly extends smoothly over $p$.

Moreover, the image of $R^{-1}$ at $p$ is spanned by $\{\partial_{x_i}, \partial_{y_i}\}_{i \notin J}$, which gives precisely the subspace $\bigcap_{i \in J} T_pH_{i}$. Since $\bigcap_{i \in J} T_pH_{i}$ is $J_\mathbb{V}$-invariant, the image of $\omega^{-1}$ is contained in $\bigcap_{i \in J} T_pH_{i}$.
\end{proof}

\begin{corollary}\label{c:extend}
For any smooth function $H:[0,1] \times \mathbb{C}^d \to \mathbb{R}$, the Hamiltonian vector field of $H|_{M(\mathbb{V})}$ on $M(\mathbb{V})$ extends to a smooth vector field over $\mathbb{C}^d$ that is tangent to the strata in the union $\bigcup_{i=1}^nH_{i}$.
\end{corollary}

\begin{proof}
The $1$-form $dH_t$ is smooth over $\mathbb{C}^d$, so the result follows from Lemma \ref{Prelim:linear-algebra-fact}.
\end{proof}

The following lemma deals with the completeness of vector fields away from lower-dimensional strata.

\begin{lemma}\label{sector:CompletenessCriterion}
Let $X\subseteq\bar{X}$ be an open submanifold of a smooth manifold $\bar{X}$ (possibly with boundary) so that $\bar{X}\setminus X$ is a stratified space of lower dimension and $V$ a vector field on $X$. If $V$ can be extended to a complete $C^1$ (e.g. smooth) vector field $\bar{V}$ on $\bar{X}$ such that $\bar{V}\vert_{\bar{X}\setminus X}$ is tangent to the strata of $\bar{X}\setminus X$, then the flow of $V$ on $X$ will not escape to $\bar{X}\setminus X$.
\end{lemma}

\begin{proof}
Since $\bar{V}$ is complete in $\bar{X}$, for every $t>0$ the flow $\psi^t\colon\bar{X}\to\bar{X}$ is a well-defined diffeomorphism sending the stratified space $\bar{X}\setminus X$ to itself. Therefore $\psi^t$ restricts to a diffeomorphism on $X$, which is also the flow of $V$ on $X$. This shows that the flow of $V$ on $X$ will not escape to $\bar{X}\setminus X$.
\end{proof}

\subsection{Homotopy Sectoriality.}\label{section:homotopy}
In Section \ref{ss:essential}, we equipped $M(\mathbb{V})$ with a natural Weinstein structure. In order to apply the Sectorial Descent (Theorem \ref{Prelim:sectorial-descent}) to study the partially wrapped Fukaya category $\mathcal{W}(M(\mathbb{V}),\xi)$, we need to find a sectorial collection of real hypersurfaces in $M(\mathbb{V})$. In this subsection, we will prove a slightly more general result, which is stated as follows. 

Given any collection of transversely intersecting $k ( \leq d)$ complex hyperplanes $\left\{S_i=\xi_i^{-1}(c_i)\right\}_{i=1}^k$ in $\mathbb{C}^d$ for some $c_i \in \mathbb{R}$, which are also transverse to the hyperplanes at infinity i.e. the complex hyperplanes $H_{1},\cdots,H_{n}\subset\mathbb{C}^d$, there exists a Liouville deformation of $M(\mathbb{V})$ such that the real hypersurfaces $\rRe(\xi_i)^{-1} (c_i)$'s form a sectorial collection after deformation.

The proof is based on a lengthy calculation, so it is helpful to start with a sketch of the argument. Pick a subcollection $\left\{S_{i_1},\cdots,S_{i_s}\right\}$ of the hyperplanes $\{S_i\}_{1\leq i\leq k}$ such that the intersection $\bigcap_{j=1}^s S_{i_j}\neq\emptyset$, we explain how to deform the Liouville structure near $\bigcap_{j=1}^s S_{i_j}$ in a way that is compatible with all other subcollections with non-empty intersections.
We first construct a collection of commuting vector fields $\{X_{\rRe(\xi_j)}, V_j\}_{j=1}^s$ near $\bigcap_{j=1}^sS_{i_j}$ such that $\omega(V_j,X_{\rRe(\xi_{j'})})=\delta_{j,j'}$,
and $\omega(X_{\rRe(\xi_j)}, X_{\rRe(\xi_{j'})})=\omega(V_j, V_{j'})=0$ (cf. Lemma \ref{sector:TrivialSymplecticSplitting} and \ref{sector:GramSchmidtProcess}).
By integrating along (the normalization of) these vector fields, it gives the corresponding $s$ pairs of $R$ and $I$ functions (in the sense of Definition \ref{prelim:cornered-Liouville-sectors}), denoted by $(R_j,I_j)_{1\leq j\leq s}$, which in turn produce a symplectic trivial fibration from a neighborhood of 
$\bigcap_{j=1}^s S_{i_j}$ to $\mathbb{C}_{\varepsilon}^s:=((-2\varepsilon ,2\varepsilon) \times \mathbb{R})^s$ (cf. Corollary \ref{c:TrivialSymplecticSplitting}).
We will also show that $I_j$ is close to $\rIm(\xi_j)$ when $\varepsilon$ is small (cf. Corollary \ref{c:compareI} and Lemma \ref{l:Iestimate}). We can then apply Lemma \ref{l:Iestimate_pre} to get control of $d^\mathbb{C}\phi$ in terms of $I_{j}$ outside of a compact subset. We write down an explicit deformation of the Liouville structure using the product coordinates obtained from the trivial fibration to $\mathbb{C}_{\varepsilon}^s$ (cf. \eqref{eq:explicit1} and \eqref{eq:explicit2}), after an explicit calculation together with the estimates of $d^{\mathbb{C}}\phi$, we show that the deformation is a Liouville deformation such that the skeleton lies in a compact set for all time (cf. Proposition \ref{sector:CompactWeinsteinDeformation} and \ref{sector:CompactWeinsteinDeformationMultiple}). The construction of the commuting vector fields and the explicit deformation of the Liouville structure are canonical enough in the sense that they are compatible when we consider different subcollections of the complex hyperplanes $\{S_i\}_{1\leq i\leq k}$.

We start with the case of a single hyperplane.

\begin{proposition}\label{sector:CompactWeinsteinDeformation}
Let $c\in\mathbb{R}$ be a real number so that $\rRe(\xi)^{-1} (c)$ intersects all hyperplanes $H_{1},\cdots,H_{n}$ transversely. Then there is a Liouville deformation $(M(\mathbb{V}),\omega, Z_t)$ with $Z_0=\nabla\phi$ and $Z_1$ is tangent to the sectorial hypersurface $\{\rRe(\xi)=c\}$ outside of a compact subset. Moreover, the skeleton $\displaystyle\overline{\bigcup_{0\leq t\leq 1}\Skel (M(\mathbb{V}),Z_t)}$ is compact.
\end{proposition}

As explained earlier, the first step is to construct the commuting vector fields $X_R$ and $V$.

\begin{lemma}\label{sector:TrivialSymplecticSplitting}
Let $c\in\mathbb{R}$  so that $\rRe(\xi)^{-1} (c)$ intersects all the hyperplanes $H_{1},\cdots,H_{n}$ transversely. 
Then there is a small neighborhood $\Nbd\left(\left\{\rRe(\xi)=c\right\}\right)$ of $\{ \rRe(\xi)=c\}$ and a vector field $V$ in $\Nbd\left(\left\{\rRe(\xi)=c\right\}\right)$
such that $\omega(V,X_{\rRe(\xi)})=1$ and the flow of $V$ is well-defined in the punctured neighborhood $\Nbd\left(\left\{\rRe(\xi)=c\right\}\right) \setminus \left(\bigcup_{i=1}^nH_i \right)$ for some positive and negative time.
\end{lemma}

\begin{proof}
Take $V=-\frac{JX_{\rRe(\xi)}}{\| X_{\rRe(\xi)}\|^2}$, then
\begin{equation}\nonumber
\omega\left(V,X_{\rRe(\xi)}\right)=g\left(\frac{X_{\rRe(\xi)}}{\| X_{\rRe(\xi)}\|^2}, X_{\rRe(\xi)}\right)=1.
\end{equation}
Thus it suffices to show that
\begin{itemize}
	\item[(i)] $\|X_{\rRe(\xi)}\|$ is non-vanishing;
	\item[(ii)] the flow of $V$ is well-defined for some small time $t\neq0$
\end{itemize}
in $\Nbd\left(\left\{\rRe(\xi)=c\right\}\right) \setminus \bigcup_{i=1}^nH_i$.

To see this,  we consider the open cover $\left\{M(\mathbb{V})_{J,\epsilon}\right\}$ of $M(\mathbb{V})$ defined by \eqref{eq:covering}
and identify the linear forms in $\ell_j$, $j\in J$ with the coordinate functions $z_1,\dotsb ,z_d$. Then the $1$-form $d^\mathbb{C}\phi$ on $M(\mathbb{V})_{J,\epsilon}$ takes the form \eqref{eq:dphic}. Looking at the entries, we see that there is a constant $c>0$ such that the matrix $(D+A^TA)-c \Id$ is positive definite. Recall that we are working over the subset $M(\mathbb{V})_{J,\epsilon}$, where $|\ell_j| >\varepsilon$ for $j \notin J$, therefore the matrix $D+A^TA$ blows up only when some of the $z_i$'s approaches $0$. We apply the same trick used in the proof of Lemma \ref{l:Iestimate_pre}. 

Let $K \subset \{1,\dots,d\}$ and consider a region $\mathcal{D}_K\subset M(\mathbb{V})$ where $(z_i)_{i\in K}$ can be arbitrarily close to $0$ but $(z_i)_{i \notin K}$ are bounded away from $0$. Write $D+A^TA$ as a $2 \times 2$ block matrix, where the first block records all the coordinates with indices in $K$. As in the proof of Lemma \ref{Prelim:linear-algebra-fact}, we can apply \cite[Theorem 2.1]{inversematrix2000} to compute the inverse of 
$D+A^TA$.
When the coordinates $(z_i)_{i \in K}$ are sufficiently small, $(D+A^TA)^{-1}$ converges to a $2 \times 2$ block matrix where only the bottom right block, denoted by $P$, is non-zero (except when $K=\{1,\dots,d\}$, in which case $(D+A^TA)^{-1}$ converges to the $0$ matrix). In fact, we also know that there are constants $C'>c'>0$ (depending on $\mathcal{D}_K$) such that $P - c'\Id$ and $C'\Id -P$ are positive definite.
The genericity of our choice of the polarization $\xi=\sum_{i=1}^db_iz_i$ (see Definition \ref{def:pol.hyp.arr}) implies that all the $b_i$ are non-zero.
It follows that we have
\begin{equation} \label{eq:equality}
0<(\min_i |b_i|)\sqrt{c'} \le \|X_{\rRe(\xi)}\| \le (\max_i |b_i|)\sqrt{C'}
\end{equation}
in $\mathcal{D}_K$ as long as $K \neq \{1,\dots,d\}$.
The transversality assumption on $\{ \rRe(\xi)=c\}$ implies that $\{ \rRe(\xi)=c\} \subset \mathbb{C}^d$ is away from $0$-dimensional strata (i.e $\bigcap_{i \in J'} H_i$ for some $|J'|=d$) formed by the hyperplanes $\{H_i\}_{i=1}^n$.
As a consequence, we can cover a small neighborhood of $\{ \rRe(\xi)=c\}$ using the regions $\mathcal{D}_K$ with $K \neq \{1,\dots,d\}$. Altogether, they give us a uniform upper bound and a uniform lower bound on the norm $\|X_{\rRe(\xi)}\|$. This proves the existence of an open neighborhood of $\{ \rRe(\xi)=c\}$ satisfying (i).

By Corollary \ref{c:extend}, we know that $JX_{\rRe(\xi)}$ extends to a vector field on $\mathbb{C}^d$. By abuse of notations, we denote its extension by $JX_{\rRe(\xi)}$.
Moreover, by (\ref{eq:equality}) and the $2 \times 2$ block matrix argument above, we see that $\|X_{\rRe(\xi)}\|$ also extends to $\mathbb{C}^d$ and it is $0$ only at the $0$-dimensional strata.
As a result, $V=\frac{-JX_{\rRe(\xi)}}{\| X_{\rRe(\xi)}\|^2}$ extends to a well-defined vector field in $\mathbb{C}^d$ except possibly at the $0$-dimensional strata, and it is tangent to the strata in $\bigcup_{i=1}^nH_i$.
By Lemma \ref{sector:CompletenessCriterion}, to show the completeness of $V$ on $M(\mathbb{V})$, it suffices to exclude the possibility that its flowlines escape from the open subset $M(\mathbb{V})\subset\mathbb{C}^d$ in a finite time. Since in a neighborhood of $\{ \rRe(\xi)=c\}$, the norm $\|V\|=\frac{1}{\|X_{\rRe(\xi)}\|}$ is uniformly bounded above by a constant, the flow of  $V$ on $M(\mathbb{V})$ is complete by the Gr\"onwall's inequality \cite[Theorem 2.8]{teschl2012ordinary}.
\end{proof}

\begin{remark}\label{r:all time}
The proof of Lemma \ref{sector:TrivialSymplecticSplitting} shows that there are positive constants $c_R$ and $C_R$ such that $c_R < \|X_{\rRe(\xi)}\| <C_R$ in $\Nbd\left(\left\{\rRe(\xi)=c\right\}\right)$ (see \eqref{eq:equality}).
Moreover, if we choose $\Nbd\left(\left\{\rRe(\xi)=c\right\}\right)$ such that it is a union of $X_{\rRe(\xi)}$-integral curves, then the flow of $X_{\rRe(\xi)}$ in $\Nbd\left(\left\{\rRe(\xi)=c\right\}\right)$ exist for all positive and negative time.
\end{remark}

\begin{corollary}\label{c:TrivialSymplecticSplitting}
Under the assumption of Lemma \ref{sector:TrivialSymplecticSplitting}, there is a constant $\varepsilon>0$ and a pair of functions
\begin{align}
(R,I)\colon\Nbd\left(\left\{\rRe(\xi)=c\right\}\right)\to \mathbb{C}_{\varepsilon}:=(-2\varepsilon ,2\varepsilon) \times \mathbb{R} \label{eq:PComm}
\end{align}
such that $(R,I)^{-1}(0)=\{\xi=c\}$, $\nabla I=X_{\rRe(\xi)}$ and $\nabla R=V$.

In particular, the vector fields $X_R$ and $X_I$ communte and hence \eqref{eq:PComm} defines a trivial symplectic fibration.
\end{corollary}

\begin{proof}
Since $\omega(V,X_{\rRe(\xi)})=1$ is a constant, the vector fields $V$ and $X_{\rRe(\xi)}$ commute and define an integrable 2-plane bundle transverse to $\{\xi=c\}$.
By Lemma \ref{sector:TrivialSymplecticSplitting}, the flow of $V$ is defined for some non-zero )(positive or negative) time.
On the other hand, the flow of $X_{\rRe(\xi)}$ is defined for all time (see Remark \ref{r:all time}).
Therefore, by integrating the vector fields $V$ and $X_{\rRe(\xi)}$, we get the functions $R$ and $I$ respectively. Since 
\[\omega(X_R,X_I)=\omega(JX_R,JX_I)=\omega(-\nabla R, -\nabla I)=\omega(V,X_{\rRe(\xi)})=1,\]
\eqref{eq:PComm} defines a trivial symplectic fibration.
\end{proof}

\begin{corollary}\label{c:compareI}
The function $I$ in Corollary \ref{c:TrivialSymplecticSplitting} agrees with $\rIm(\xi)$.
\end{corollary}

\begin{proof}
The gradient vector field of $\rIm(\xi)$ is given by $\nabla \rIm(\xi)=-JX_{\rIm(\xi)}=X_{\rRe(\xi)}$, which is the same as the gradient vector field of $I$.

Recall from Lemma \ref{sector:TrivialSymplecticSplitting} that $V=-JX_{\rRe(\xi)}/\| X_{\rRe(\xi)}\|^2=-X_{\rIm(\xi)}/\| X_{\rRe(\xi)}\|^2$, which is tangent to $\{\rIm(\xi)=0\}$. As a result, $I|_{\{\rIm(\xi)=0\}}=0$. The uniqueness of integral curve with the same initial condition implies that $I=\rIm(\xi)$.
\end{proof}

\begin{lemma}\label{l:Iestimate}
Let $R$ and $I$ be functions as in Lemma \ref{sector:TrivialSymplecticSplitting} and Corollary \ref{c:TrivialSymplecticSplitting}. Suppose also that  
\begin{align*}
\bigcap_{\substack{j\in K\\ \vert K\vert=d+1}}\left\{\vert\ell_j\vert\leq 1\right\}=\emptyset.
\end{align*}
Then inside the open neighborhood $\Nbd\left(\{\rRe(\xi)=c\}\right)\cong F\times\mathbb{C}_\varepsilon$, we have a function $C(z)>0$ and two positive constants $c_1,c_2$ such that $c_1<C(z)<c_2$ and 
    \[
        -d^\mathbb{C}\phi (X_I)=C(z)I
    \]
    for $\vert I\vert$ sufficiently large. 
\end{lemma}

\begin{proof}
This is a consequence of Lemma \ref{l:Iestimate_pre} and Corollary \ref{c:compareI}.
\end{proof}

We will now start to prove Proposition \ref{sector:CompactWeinsteinDeformation}, the main result of this subsection, using Lemma \ref{l:Iestimate}.

\begin{proof}[Proof of Proposition \ref{sector:CompactWeinsteinDeformation}]
The proof is divided into six steps.
\begin{paragraph}{Step 0: preparation.}
Pick $c'\gg0$ large enough so that $(M(\mathbb{V}),c'\phi)$ is covered by the open subsets $M(\mathbb{V})_{J,1}$ (i.e. the assumption \eqref{eq:convercontrol} holds). 
By Lemma \ref{l:rescaling}, working with $(M(\mathbb{V}),c'\phi)$ does not change the Weinstein deformation type, so it suffices to prove the proposition for $(M(\mathbb{V}),c'\phi)$ instead.

Using the cover $\left\{M(\mathbb{V})_{J,1}\right\}$, and under the convention of Lemma \ref{sector:TrivialSymplecticSplitting}, we write $\xi =\displaystyle\sum_{i=1}^d b_i \ell_i=\displaystyle\sum_{i=1}^d b_i z_i$. 
Recall $R$ and $I$ from  Lemma \ref{sector:TrivialSymplecticSplitting} and Corollary \ref{c:TrivialSymplecticSplitting}.
Let $F=(R,I)^{-1} (0)$, we have a symplectic product decomposition
    \[
        \Nbd\left(\{\mathrm{Re}(\xi)=c\}\right)\cong F\times (-\varepsilon ,\varepsilon)\times\mathbb{R},
    \]
which is given by symplectic parallel transports along the second and the third factors.
Denote the Liouville $1$-form on $M(\mathbb{V})$ by $\lambda_\mathbb{V}$.
Let $\lambda_F =\lambda_\mathbb{V}\vert_{F\times\{0\}}$ and recall the definition of the $1$-form $\lambda_{\mathbb{C}}^{\alpha} $ from \eqref{eq:lambdaa}. For any $0<\alpha<1$, the difference between our Liouville $1$-form $\lambda_\mathbb{V}$ and $\lambda_F +\lambda_{\mathbb{C}}^{\alpha}$
is exact in $\Nbd\left(\{\mathrm{Re}(\xi)=c\}\right)$, so we can write
    \[
        \lambda_\mathbb{V}=\lambda_F +\lambda_{\mathbb{C}}^{\alpha} + df,
    \]
    where $f\colon\Nbd\left(\{\mathrm{Re}(\xi)=c\}\right)\to\mathbb{R}$ is a not necessarily compactly supported smooth function in $\Nbd\left(\{\mathrm{Re}(\xi)=c\}\right)$ which depends on $\alpha$.
\end{paragraph}

\begin{paragraph}{Step 1: an explicit deformation.}
Let $\rho\colon \mathbb{R}\to\mathbb{R}$ be a smooth function depicted on the right-hand side of Figure \ref{bump-function-construction}, which is zero near the origin, particularly $\rho'(0)=\rho''(0)=0$, and $\rho (x)=x -c''$ for some $c''>0$ near $x=\varepsilon$, $\rho (x)=x+c''$ near $x=-\varepsilon$ and $\rho(x)=x$ when $\vert x\vert \geq 2\varepsilon$. 
We assume that $0\le \rho' \le 1$ when $|x| \le \varepsilon$ and $1 \le \rho' \le 1+ \delta_{c''}$ when $|x| \ge \varepsilon$, where $\delta_{c''}>0$ is a small number depending on $c''$.
By possibly replacing $c''$ with a smaller one, we can make $\delta_{c''}$ as close to $0$ as we want.
Set 
\begin{align}\label{eq:explicit1}
\tilde{f} (x,p)\coloneqq f\left(x,\rho (R(p)),\rho'(R(p))I(p)\right)
\end{align}
 for $|R(p)|\leq\varepsilon$ and 
\begin{align}\label{eq:explicit2}
\tilde{f}(x,p)\coloneqq f\left(x,\rho (R(p)),I(p)\right)
\end{align}
for $\varepsilon\leq |R(p)|\leq 2\varepsilon$. This is well-defined because $\rho'=1$ near $|R|=\varepsilon$.
Moreover, $\tilde{f}=f$ when $|R| \ge 2\varepsilon$.

\begin{figure}[ht]
    \centering 
    \subfigure[Derivative of $\rho$]{\begin{tikzpicture}[scale=0.5]
        \draw [->] (-4,0) -- (4,0) node [below] {$x$};
        \draw [->] (0,-4) -- (0,4) node [left] {$y$};
        \draw (0,3) node [left] {$1$} -- (0.25,3); 
        \draw [red, thick] (-4,3) --(-3,3) .. controls (-2.5,3) and (-2.5,4) .. (-2,4) .. controls (-1.5, 4) and (-1.5, 3) .. (-1,3) .. controls (-0.5,3) and (-0.25,0) .. (0,0) .. controls (0.25,0) and (0.5,3) .. (1,3) .. controls (1.5,3) and (1.5,4) .. (2,4) .. controls (2.5,4) and (2.5,3) .. (3,3) -- (4,3);
        \draw (-1,0) node [below] {\small$-\varepsilon$} -- (-1,0.25);
        \draw (-2,0) node [below]{\small$-2\varepsilon$} -- (-2,0.25);
        \draw (1,0) node [below] {\small$\varepsilon$} -- (1,0.25);
        \draw (2,0) node [below] {\small$2\varepsilon$} -- (2,0.25);
    \end{tikzpicture}}\hspace{2cm}
    \subfigure[The Function $\rho$]{\begin{tikzpicture}[scale=0.5]
        \draw [->] (-4,0) -- (4,0) node [below] {$x$};
        \draw [->] (0,-4) -- (0,4) node [left] {$y$};
        \draw [red, thick] (-4,-4) -- (-3,-3) .. controls (-2.5,-2.5) and (-2.5,-1.5) .. (-2,-1) .. controls (-1,0) .. (0,0) .. controls (1,0) .. (2,1) .. controls (2.5, 1.5) and (2.5, 2.5) .. (3,3) -- (4,4);
        \draw (-2,0) node [below] {\small$-\varepsilon$} -- (-2,0.25);
        \draw (-3,0) node [below]{\small$-2\varepsilon$} -- (-3,0.25);
        \draw (2,0) node [below] {\small$\varepsilon$} -- (2,0.25);
        \draw (3,0) node [below] {\small$2\varepsilon$} -- (3,0.25);
    \end{tikzpicture}}
    \caption{\scshape The Bump Function $\rho$ and its Derivative.}
    \label{bump-function-construction}
\end{figure}
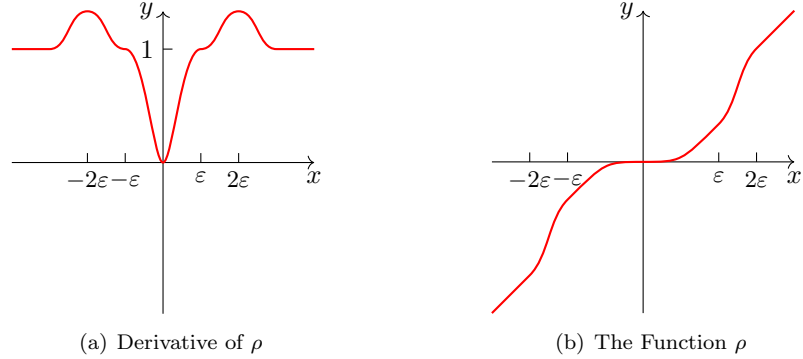

Let $\lambda_1 :=\lambda_F +\lambda_{\mathbb{C}}^{\alpha} + d\tilde{f}$. Since $\tilde{f}=0$ near $R=0$, the symplectic dual of $\lambda_1$ is tangent to the hypersurface $F\times\{0\}\times\mathbb{R}$ near infinity, such that $F\times\{0\}\times\mathbb{R}$ is sectorial with defining function $I$. 
For $0\leq t\leq 1$, consider the family of $1$-forms
    \[
        \lambda_t =t\lambda_1 +(1-t)\lambda_\mathbb{V} =\lambda_F +\lambda_{\mathbb{C}}^{\alpha} + d\left(t\tilde{f} +(1-t)f\right).
    \] 
which induces a Liouville homotopy from $(X,\omega ,\lambda_\mathbb{V})$ to $(X,\omega,\lambda_1)$. 
Our goal is to show that for an appropriate choice of $\alpha \in (0,1)$, the skeleton of this family of Liouvile $1$-forms is compact.  

Recall the notation $\mathbb{C}_{\varepsilon}=(-2\varepsilon ,2\varepsilon)\times\mathbb{R}$.
To show that the skeleton is compact, our strategy is first to compute
\[
d\tilde{f}\vert_{\{z\} \times \mathbb{C}_{\varepsilon}} \quad \text{ and } \quad df\vert_{\{z\} \times \mathbb{C}_{\varepsilon}}
\]
which shows that for every $z$, the zeros of $\lambda_t\vert_{\{z\} \times \mathbb{C}_{\varepsilon}}$ lie in a compact set.
Then we will argue that the zeros of $\lambda_t\vert_{F \times \{p\} }$ lie in a compact subset for every $p$ to conclude that the zeros of $\lambda_t$ lie in a compact subset.
\end{paragraph}

\begin{paragraph}{Step 2: analyzing the critical points using Lemma \ref{l:Iestimate}.}
We implement the strategy described above, and consider first the region where $|R(p)| \le \varepsilon$.

On $\{z\} \times \mathbb{C}_{\varepsilon}$, $\partial_R$ is by definition the unique vector field in $\operatorname{Span}\{V,X_{\rRe(\xi)}\}=\operatorname{Span}\{X_R,X_I\}$ such that $dR(\partial_R)=1$ and $dI(\partial_R)=0$.
In other words, $\omega(X_R,\partial_R)=-1$ and $\omega(X_I,\partial_R)=0$.
Therefore, $\partial_{R}=-X_I$ and similarly, $\partial_I=X_R$.
 Since $\partial_R \tilde{f}=\rho' \partial_R f+\rho'' I \partial_Rf$ and $\partial_I \tilde{f}=\rho' \partial_I f$, we have
    \begin{align}
     d\tilde{f}\vert_{\{z\} \times \mathbb{C}_{\varepsilon}}= \rho' \left(-df(X_I)dR+df(X_R)d I\right) -\rho'' Idf(X_I)dR. \label{eq:dftilde}
    \end{align}
Recall that
    \begin{align}
        \lambda_{\mathbb{C}}^{\alpha} =-\alpha I dR +(1-\alpha)Rd I, \label{eq:lamalphaRI}
    \end{align}
    so $df$ is given by
    \begin{align}
    df\vert_{\{z\} \times \mathbb{C}_{\varepsilon}} &=-\left. d^{\mathbb{C}}\phi\right\vert_{\{z\} \times \mathbb{C}_{\varepsilon}} -\lambda_{\mathbb{C}}^{\alpha} =\underbrace{\left(-d^\mathbb{C}\phi (X_R)-(1-\alpha)R\right)}_{df(X_R)}dI -\underbrace{\left(-d^\mathbb{C}\phi (X_I)-\alpha I\right)}_{df(X_I)}dR. \label{eq:df}
    \end{align}
    Consider the zeros of the restriction of the $1$-form $\lambda_t$ to $\{z\} \times \mathbb{C}_{\varepsilon}$ for $0\leq t\leq 1$. By \eqref{eq:lamalphaRI}, \eqref{eq:df} and \eqref{eq:dftilde},
    \begin{equation} \label{eq:DRI}
    \begin{split}
  &  \lambda_t\vert_{\{z\}\times\mathbb{C}_{\varepsilon}} \\
=&\lambda_{\mathbb{C}}^{\alpha} +td\tilde{f}\vert_{\{z\}\times\mathbb{C}_{\varepsilon}} +(1-t)df\vert_{\{z\}\times\mathbb{C}_{\varepsilon}}\\ 
     =&\left((1-\alpha)R+(1+t(\rho'-1))df(X_R)\right)dI -\left(\alpha I+t\rho'' Idf(X_R)+(1+t(\rho'-1))df(X_I)\right)dR.
     \end{split}
    \end{equation}
    It suffices to look at the zeroes of the coefficients of $dR$ and $dI$. For $dI$, the coefficient being zero implies that
    \[
        df(X_R)=-\frac{(1-\alpha)R}{1-t(1-\rho')}.
    \]
    When $|I|$ is sufficiently large, it follows from Lemma \ref{l:Iestimate} that the coefficient of $dR$ in \eqref{eq:DRI} being $0$  means 
    \begin{align*}
0&=\left(\alpha I+t\rho'' Idf(X_R) +(1+t(\rho'-1))df(X_I)\right)\\        
&=\left(\alpha I-t\rho'' I\frac{(1-\alpha)R}{1-t(1-\rho')} +(1+t(\rho'-1))\left(C(z)I-\alpha I\right)\right)\\
&=\left(t(1-\rho')\alpha -t\rho'' \frac{(1-\alpha)R}{1-t(1-\rho')} +(1-t(1-\rho'))C(z)\right)I.
    \end{align*}
   The only possible solution is 
\begin{equation} \label{eq:alpha}
\alpha =\dfrac{t\rho'' R-C(1-t(1-\rho'))^2}{t\rho'' R+t(1-\rho')(1-t(1-\rho'))}=1-\dfrac{C(1-t(1-\rho'))^2 +t(1-\rho')(1-t(1-\rho'))}{t\rho'' R+t(1-\rho')(1-t(1-\rho'))}. 
\end{equation}
\end{paragraph}

\begin{paragraph}{Step 3: analyzing the critical points by a direct calculation.}
Let $A=t(1-\rho')$, then the right-hand side of \eqref{eq:alpha} reads 
    \begin{equation}\label{eq:estimate-alpha}
        1-\dfrac{C(1-A)^2 +A(1-A)}{t\rho'' R+A(1-A)}.
    \end{equation}
Since we are considering the region where $|R| \le \varepsilon$ and $|\rho''|$ is bounded (cf. Figure \ref{bump-function-construction}(a)), we have $0 \le \rho' \le 1$, so $\dfrac{C(1-A)^2 +A(1-A)}{t\rho'' R+A(1-A)} \ge 0$.

When $A\to 0$, we have $\rho'\to 1$ or $t\to 0$, hence $\dfrac{C(1-A)^2+A(1-A)}{t\rho'' R+A(1-A)}\to\infty$, while if $A\to 1$, we must have $\rho'\to 0$ (thus also $\rho''\rightarrow0$) and $t\to 1$, and a variation of the formula for $\alpha$ gives \[
        \frac{1-\alpha}{1-A}=\frac{C(1-A)+A}{t\rho''R+A(1-A)},
    \] 
    which approaches $\infty$ as $A\to 1$. Looking back into the coefficient of $dI$, we have \[
        (1-A)df(X_R)=-(1-\alpha )R,
    \]
    which is equivalent to the equation $\dfrac{A(1-\alpha)}{1-A}R=d^\mathbb{C}\phi (X_R)$. Since $\vert d^\mathbb{C}\phi (X_R)\vert\leq C<\infty$ is bounded, $\dfrac{R}{1-A}$ must also be bounded. Traveling back to the coefficient for $dR$, we have
    \[
        1=\frac{C(1-A)+A}{t\left(\dfrac{1-\alpha}{1-A}R\right)\rho''+A(1-\alpha)}.
    \]
If $A\to 1$, the numerator converges to $1$, while $t\dfrac{1-\alpha}{1-A}R\rho''\to 0$ and $A(1-\alpha)<1$, so we derive that $1<1$ as $A\to 1$, a contradiction. Thus $A$ cannot converge to $1$, so there is some $0<M<1$ with $0\leq A\leq M<1$, such that \eqref{eq:estimate-alpha} is away from $1$. Therefore if $\alpha$ is sufficiently close to $1$, the equation (\ref{eq:alpha}) has no solution.
  
We conclude that there is an $\alpha\in(0,1)$ such that  the Liouville $1$-form $\lambda_t|_{\{z\}\times\mathbb{C}_{\varepsilon}}$ has no critical points in the region $|R| \le \varepsilon$ when $|I|$ is large.
Note that, Lemma \ref{l:Iestimate} and the estimates above are uniform in $z$ so there is an $\alpha\in(0,1)$ such that when $|I|$ is large, 
$\lambda_t|_{\{z\}\times\mathbb{C}_{\varepsilon}}$ has no critical points in the region $|R| \le \varepsilon$ for all $z \in F$.
\end{paragraph}

\begin{paragraph}{Step 4: the easier region.}
Next we consider the region where $|R(p)|\geq\varepsilon$. 
When $|R(p)|\geq\varepsilon$, we have 
\[ d\tilde{f}|_{\{z\} \times \mathbb{C}_{\varepsilon}} =-\rho'(R)df(X_I)dR +df(X_R)dI. \]
In this case,
    \begin{align*}
        \lambda_t\vert_{\{z\}\times\mathbb{C}_{\varepsilon}} &
=\lambda_{\mathbb{C}}^{\alpha} +td\tilde{f}\vert_{\{z\}\times\mathbb{C}_{\varepsilon}} +(1-t)df\vert_{\{z\}\times\mathbb{C}_{\varepsilon}}\\ 
        &=\left((1-\alpha)R+df(X_R)\right)dI -\left(\alpha I+(1+t(\rho'-1))df(X_I)\right)dR.
    \end{align*}
Applying Lemma \ref{l:Iestimate} again, the coefficient of $dR$ being zero means, for $|I|$ sufficiently large,
\begin{align*}
0&=-d^\mathbb{C}\phi(X_I)-t(\rho'-1)d^\mathbb{C}\phi(X_I) - t(\rho'-1) \alpha I \\
&=\left( (1+t(\rho'-1))C-t(\rho'-1) \right)\alpha I\\
&=\left(C+t(\rho'-1)(C-1)\right)\alpha I.
\end{align*}
Since $C$ is uniformly bounded from above, we can choose $\rho'$ to be very close to $1$ such that the equation above has no solution for all $t \in [0,1]$.
Therefore, $\lambda_t|_{\{z\}\times\mathbb{C}_{\varepsilon}}$ has no critical points in the region $|R(p)|\geq\varepsilon$ as well.
\end{paragraph}

\begin{paragraph}{Step 5: concluding the proof.}
We conclude that there is an $\alpha\in(0,1)$ such that the critical points of the Liouville $1$-form $\lambda_t$ lie in a compact subset of $\{z\}\times\mathbb{C}_{\varepsilon}$ for any $z\in F$, since the previous discussion does not depend on the choice of $z$. Compactness of the critical points of $\lambda_t$ along the fiber follows from the fact that we are deforming the Liouville structure so that it remains complete in the complement $M(\mathbb{V})$, which we know from Corollary \ref{generation:finiteness-critical-points} to have compact set of critical points.
\end{paragraph}
\end{proof}

We now generalize the argument above to the case where there are transversely intersecting hyperplanes $\xi_i^{-1} (c_i^{\pm})\subset\mathbb{C}^d$ with $c_i^- <c_i^+$ for $i=1, \cdots, d$. 

\begin{lemma}\label{l:productDe}
For any subcollection of hypersurfaces in $\{\rRe(\xi_i)^{-1} (c_i^{\pm})| i=1, \cdots, d\}$, their intersection is either coisotropic or empty.
\end{lemma}

\begin{proof}
This is a direct consequence of the fact that the vector fields $X_{\mathrm{Re}(\xi_i)}$ are purely imaginary and the characteristic foliation of $\rRe(\xi_i)^{-1} (c_i^{\pm})$ is generated by $X_{\rRe(\xi_i)}$.
\end{proof}

Lemma \ref{l:productDe} verifies the first item of  Definition \ref{prelim:cornered-Liouville-sectors}.
To proceed, we need to show the orthogonality \eqref{eq:decomp} in the definition of sectorial collections (cf. Definition \ref{prelim:cornered-Liouville-sectors}). 
In fact, we can prove the following generalization of Lemma \ref{sector:TrivialSymplecticSplitting}.

\begin{lemma}\label{sector:GramSchmidtProcess}
There exists a $d$-tuple of vector fields $\{V_1^{\pm} ,\dotsb ,V_d^{\pm}\}$, defined in the product neighbourhood of $\bigcap_{i=1}^d\left\{\rRe(\xi_i)=c_i^{\pm}\right\}$, satisfying the following properties:
    \begin{enumerate}[\indent i)]
        \item $\omega (V_i^{\pm}, V_j^{\pm})=0$ for all $i,j$;
        \item $\omega\left(V_i^{\pm}, X_{\rRe(\xi_j)}\right)=\delta_{ij}$;
        \item The flow of $V_i^{\pm}$ exists for some positive time.
    \end{enumerate}
\end{lemma}

\begin{proof}
With the same conventions as in Lemma \ref{sector:Standardization}, we write $\xi_i =\displaystyle\sum_{j=1}^d b_{ij} z_j$, then the Hamiltonian vector field of $\rRe(\xi_i)$ is given by 
    \[
        X_{\rRe(\xi_i)} =\frac{1}{4}\vec{b}_i^T(D+A^TA)^{-1}\frac{\partial}{\partial\vec{y}},
    \]
    and similarly
    \[
        X_{\rIm(\xi_i)} =-\frac{1}{4}\vec{b}_i^T(D+A^TA)^{-1}\frac{\partial}{\partial\vec{x}},
    \]
    so that the symplectic pairing between them is given by
    \[
        \omega_{ij}\coloneqq\omega\left(X_{\rRe(\xi_i)}, X_{\rIm(\xi_j)}\right) =4X_{\rRe(\xi_i)}^T(D+A^TA)X_{\rIm(\xi_j)}=\frac{1}{4} \vec{b}_i^T (D+A^TA)^{-1}\vec{b}_j,
    \]
    which is bounded for all $z\in M(\mathbb{V})$. Given a non-empty subset $J\subseteq\{1,\dotsb ,d\}$ with $\vert J\vert=r<d$, without loss of generality we may assume that $J=\{1,\dotsb ,r\}$. The restriction of the symplectic form $\omega$ to the subbundle of $T^\ast M(\mathbb{V})$ generated by $\left\{X_{\rRe(\xi_i)} ,X_{\rIm(\xi_i)}\right\}_{1\leq i,j\leq r}$ is given by the matrix 
    \[
        \Omega =\begin{pmatrix}
            0&(\omega_{ij})\\ 
            (-\omega_{ij})&0
        \end{pmatrix}.
    \]
    Let $\{V_i\}_{i=1}^r$ be a family of vector fields defined as linear combinations of $\left\{X_{\rIm(\xi_j)}\right\}$ such that $\omega\left(V_i ,X_{\rRe(\xi_i)}\right)=1$ and $\omega\left(V_i,X_{\rRe(\xi_j)}\right)=0$ for $i\neq j$. Such $V_i$'s can be obtained by taking the inverse matrix of $\Omega$. Also, since the coefficients of $\Omega$ are polynomials in $\omega_{ij}$, which are bounded for $z\in M(\mathbb{V})$, the coefficients of the vector field $V_i$ are Laurant polynomials in $\omega_{ij}$. These coefficients are bounded unless the determinant of $\Omega$ goes to $0$. However, Lemma \ref{Prelim:linear-algebra-fact} shows that this cannot happen near $\bigcap_{j \in J} H_{j}$. This allows us to apply the completeness criterion (Lemma \ref{sector:CompletenessCriterion}) to conclude that all such $V_i$'s admit flows in the complement of $\bigcup_{j=1}^dH_{j}$ such that they do not escape to $H_{j}$.
Now we have obtained $V_i$'s for each open neighborhood of the intersection $\displaystyle\bigcap_{j\in J}\rRe(\xi_j)^{-1} (c_j^{\pm})$, removing a neighborhood of some deeper strata, and we can patch them together using partition of unity to get vector fields $V_i$ defined in a neighborhood of $\rRe(\xi_i)^{-1} (c_i^{\pm})$, which still have well-defined flows in the complement of $\bigcup_{j=1}^dH_{j}$ such that they do not escape to $H_{j}$. 
There is a uniform upper bound for the norms $\|V_i\|$ so they are complete by Gr\"onwall's inequality. It follows from the definition that $\omega\left(V_i^{\pm}, X_{\rRe(\xi_j)}\right)=\delta_{ij}$, which completes the proof.
\end{proof}

As a consequence, by the same reasoning in Corollary \ref{c:TrivialSymplecticSplitting}, we can integrate the vector fields $V_i^{\pm}$ and $X_{\rRe(\xi_i)}$ to obtain functions $\{R_i,I_i\}$, and
a symplectic product decomposition of a neighborhood of the intersection $\displaystyle S^{\sigma}\coloneqq\bigcap_{i=1}^d\rRe(\xi_i)^{-1} (c_i^{\sigma_i})$
\[
\left((R_i ,I_i)_{i=1,\dots,d}\right):\Nbd (S^{\sigma})\rightarrow\mathbb{C}_{\varepsilon}^d
\]
for each $\sigma =(\sigma_1,\dotsb ,\sigma_d)\in\{\pm\}^d$.
This verifies that Definition \ref{prelim:cornered-Liouville-sectors}(ii) holds for the collection of hypersurfaces $\left\{\rRe(\xi_i)^{-1}(c_i^{\pm})\right\}_{i=1,\dots,d}$.  
Note that both $I_i$ and $\rIm(\xi_i)$ are obtained by integrating $X_{R_i}$.
This time, $I_i|_{\rIm(\xi_i)=0}$ is not necessarily $0$ because $X_{\rIm(\xi_j)}$ is not necessarily tangent to $\{\rIm(\xi_i)=0\}$.
However, $I_i$ and $\rIm(\xi_i)$ are arbitrarily $C^{\infty}$ close to each other when $\varepsilon$ is small because $I_i|_{S^{\sigma}}=0=\rIm(\xi_i)|_{S^{\sigma}}$ and they are obtained by integrating the same vector field.
The direct generalization of Lemma \ref{l:Iestimate} is therefore also true by replacing $\xi$ with $\xi_i$ and $I$ with $I_i$.


We can adapt the proof of Proposition \ref{sector:CompactWeinsteinDeformation} to the collection of hypersurfaces $\left\{\rRe(\xi_i)^{-1}(c_i^{\pm})\right\}_{i=1}^d$ because we can use an explicit product type deformation in Step 1 (with respect to the product decomposition of $\mathbb{C}_{\varepsilon}^d$), apply the generalization of Lemma \ref{l:Iestimate} in Step 2, and the rest follows by the same computations.
Thus we have

\begin{proposition}\label{sector:CompactWeinsteinDeformationMultiple}
There exists a Liouville deformation $\left(M(\mathbb{V}),\omega ,Z_t\right)$ with compact skeleton such that $Z_0=\nabla\phi$ and the real hypersurfaces $\rRe(\xi_i)^{-1} (c_i^{\pm})\subset M(\mathbb{V})$ are all sectorial with respect to the deformed Liouville structure $(M(\mathbb{V}),\omega ,Z_1)$.
\end{proposition}

We can actually decompose this Liouville deformation into a finite number of steps, each of which makes one of the hypersurfaces $\rRe(\xi_i)^{-1} (c_i^{\pm})$ sectorial. By \cite[Proposition 11.8]{cieliebak2012stein}, we know that there exists an exact symplectomorphism
\begin{equation} \nonumber
f\colon (M(\mathbb{V}),\omega,\lambda_{\phi})\rightarrow (M(\mathbb{V}),\omega,\lambda_1)
\end{equation}
such that $f^{\ast}\lambda_1 =\lambda_{\phi} +dg$ where $g\colon X\to\mathbb{R}$ is a compactly supported function. This implies that the wrapped Fukaya category is unchanged up to quasi-equivalence under this deformation.

\subsection{Sectorial Decomposition.}\label{sec:sectorial decomposition} 
Given a polarized hyperplane arrangment $\mathbb{V}=(V, \eta, \xi)$, in this subsection we introduce sectorial cuts of the Weinstein manifold $M(\mathbb{V})$, so that it is divided into standard pieces with known generating Lagrangians. Note that the existence of sectorial cuts relies on the simplicity assumption of $\mathbb{V}$. The construction is divided into two steps.

\subsubsection*{Step 1:} Let $\{\xi_i\}_{1\leq i\leq d}$ be a basis of $V^*$ with $\xi_1 =\xi$. Note that each $\xi_i$ can be identified with a linear hyperplane in $V$, and we require that none of the linear hyperplanes associated to the $\xi_i$'s are parallel to any of the hyperplanes $H_{\mathbb{R},j}$ in the arrangement, where $1\leq j\leq n$. Let $N_\mathbb{V}$ be the number of $0$-dimensional strata in the hyperplane arrangement $\mathbb{V}$. Pick real numbers $(c_{1j})_{1\leq j\leq N_\mathbb{V}}$ and $(c_{ik})_{2\leq i\leq d, 1\leq k\leq 2N_\mathbb{V}}$ with the following properties:

\begin{enumerate}[\indent i)]
    \item For each $1\leq j\leq N_\mathbb{V}$, there exists a unique $q_{1j}$ so that $\rRe(\xi_1)^{-1} (q_{1j})$ intersects with the hyperplanes in $\mathbb{V}$ at a unique $0$-dimensional stratum. In this case, choose real numbers $(c_{1j})$ so that $c_{1,j-1}<q_{1j}<c_{1j}$ holds (when $j=1$, we set $c_{1,0}=-\infty$).
    \item For each $1\leq j\leq N_\mathbb{V}$ and $i>1$, there exists a $q_{ij}$ so that $\rRe(\xi_i)^{-1} (q_{ij})$ intersects with some hyperplanes in $\mathbb{V}$ at a $0$-dimensional stratum. Pick real numbers $(c_{ik})$ with $c_{i,2j-1} <q_{ij}<c_{i,2j}$, so that
    \begin{equation} \nonumber
    \left(\rRe(\xi_i)^{-1} (c_{i,2j-1})\cup\rRe(\xi_i)^{-1} (c_{i,2j})\right)\cap\rRe(\xi_1)^{-1}\left([c_{1j},c_{1,j+1}]\right)
    \end{equation}
    intersects transversely with at most $d-2$ hyperplanes in $\mathbb{V}$.
\end{enumerate}

By Proposition \ref{sector:CompactWeinsteinDeformationMultiple}, after a compactly supported Liouville deformation, we can arrange that all the real hypersurfaces $\left\{\mathrm{Re}(\xi_i)=c_{ik}\right\}_{i,k}$, including $i=1$, form a sectorial collection. Thus we can use them to cut the Weinstein manifold $M(\mathbb{V})$ into subsectors.

\subsubsection*{Step 2:} We first cut $M(\mathbb{V})$ using the hypersurfaces defined by $\xi_1$, which leads to the decomposition
\[
    M(\mathbb{V}) =\bigcup_{j=1}^{N_\mathbb{V}}\rRe(\xi_1)^{-1}\left([c_{1,j-1},c_{1j}]\right)\eqqcolon\bigcup_{j=1}^{N_\mathbb{V}} P_j,
\]
so that we get a total number of $N_\mathbb{V}$ sectors with similar behaviors as depicted in Figure \ref{sector:decompose-pair-of-pants 2}. See also the related discussions in the introduction (Section \ref{sec:ideaofproof}).

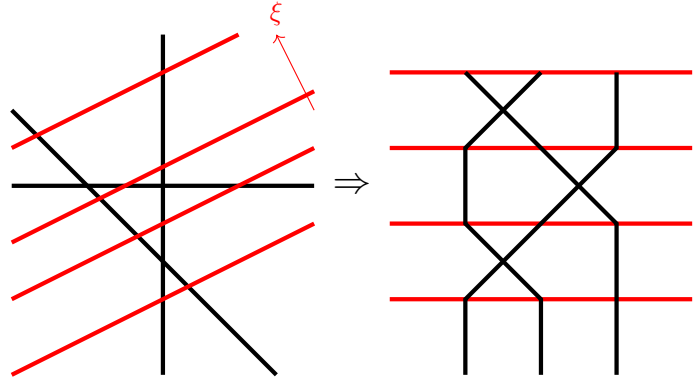
\begin{figure}[ht]
    \centering
    \begin{tikzpicture}[scale=0.5]
        \draw [ultra thick] (-4,0) -- (4,0);
        \draw [ultra thick] (0,-5) -- (0,4);
        \draw [ultra thick] (-4,2) -- (3,-5);
        \draw [->, red] (4,2) -- (3,4) node[above] {$\xi$};
        \draw [ultra thick, red] (-4,-5) -- (4,-1);
        \draw [ultra thick, red] (-4,-3) -- (4,1);
        \draw [ultra thick, red] (-4,-1.5) -- (4,2.5);
        \draw [ultra thick, red] (-4,1) -- (2,4);
        \node at (5,0) {\Large $\Rightarrow$};
        \draw [red, ultra thick] (6,1) -- (14,1);
        \draw [red, ultra thick] (6,-1) -- (14,-1);
        \draw [red, ultra thick] (6,3) -- (14,3);
        \draw [red, ultra thick] (6,-3) -- (14,-3);
        \draw [ultra thick] (8,3) -- (12,-1) -- (12,-5);
        \draw [ultra thick] (10,3) -- (8,1) -- (8,-1) -- (10,-3) -- (10,-5);
        \draw [ultra thick] (12,3) -- (12,1) -- (8,-3) -- (8,-5);
    \end{tikzpicture}
\caption{Decomposing $2$-dimensional pair-of-pants into sectors\label{sector:decompose-pair-of-pants 2}}
\end{figure}

Next, we perform further cuts on each subsector $P_j$ using the real hypersurfaces $\rRe(\xi_i)^{-1} (c_{i,2j-1})$ and $\rRe(\xi_i)^{-1} (c_{i,2j})$ for $i>1$, which gives the following decomposition
\[
    P_j =\bigcup_{\alpha\in\{-,0,+\}^{d-1}}\left(\bigcap_{i=2}^d\rRe(\xi_i)(\alpha_i)\right)\cap\rRe(\xi_1)^{-1} ([c_{1,j-1} ,c_{1j}])\eqqcolon\bigcup_{\alpha\in\{-,0,+\}^{d-1}}P_{j\alpha},
\]
where $\alpha =(\alpha_2, \dotsb ,\alpha_d)$ and
\begin{equation}\nonumber
\begin{split}
\rRe(\xi_i)(-)&=\rRe(\xi_i)^{-1} ((-\infty, c_{i,2j-1}]), \\
\rRe(\xi_i)(0)&=\rRe(\xi_i)^{-1} ([c_{i,2j-1}, c_{i,2j}]), \\
\rRe(\xi_i)(+)&=\rRe(\xi_i)^{-1} ([c_{i,2j},+\infty)). \\
\end{split}
\end{equation}
An intuitive picture in dimension two is shown in Figure \ref{sector:second-cut2}, where the green lines represent the second family of sectorial hypersurfaces that we introduced above, and the dashed parts indicate that we do not use these hypersurfaces to cut other irrelevant pieces.

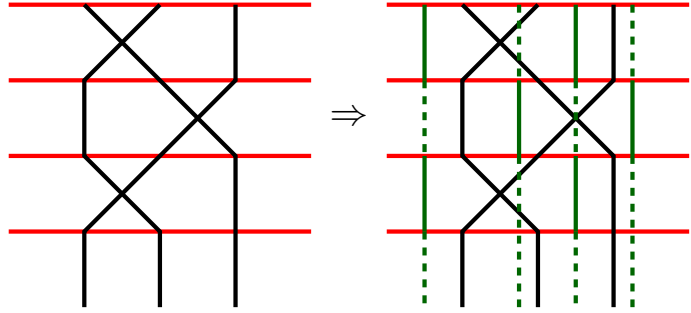
\begin{figure}[ht]
    \centering
    \begin{tikzpicture}[scale=0.5]
        \draw [red, ultra thick] (-4,1) -- (4,1);
        \draw [red, ultra thick] (-4,-1) -- (4,-1);
        \draw [red, ultra thick] (-4,3) -- (4,3);
        \draw [red, ultra thick] (-4,-3) -- (4,-3);
        \draw [ultra thick] (-2,3) -- (2,-1) -- (2,-5);
        \draw [ultra thick] (0,3) -- (-2,1) -- (-2,-1) -- (0,-3) -- (0,-5);
        \draw [ultra thick] (2,3) -- (2,1) -- (-2,-3) -- (-2,-5);
        \node at (5,0) {\Large $\Rightarrow$};
        \draw [red, ultra thick] (6,1) -- (14,1);
        \draw [red, ultra thick] (6,-1) -- (14,-1);
        \draw [red, ultra thick] (6,3) -- (14,3);
        \draw [red, ultra thick] (6,-3) -- (14,-3);
        \draw [ultra thick] (8,3) -- (12,-1) -- (12,-5);
        \draw [ultra thick] (10,3) -- (8,1) -- (8,-1) -- (10,-3) -- (10,-5);
        \draw [ultra thick] (12,3) -- (12,1) -- (8,-3) -- (8,-5);
        \draw [black!60!green, ultra thick] (7,3) -- (7,1);
        \draw [black!60!green, ultra thick, dashed] (7,1) -- (7,-1);
        \draw [black!60!green, ultra thick] (7,-1) -- (7,-3);
        \draw [black!60!green, ultra thick, dashed] (7,-3) -- (7,-5);
        
        \draw [black!60!green, ultra thick] (11,3) -- (11,1);
        \draw [black!60!green, ultra thick, dashed] (11,1) -- (11,-1);
        \draw [black!60!green, ultra thick] (11,-1)--(11,-3);
        \draw [black!60!green, ultra thick, dashed] (11,-3) -- (11,-5);
        
        \draw [black!60!green, ultra thick, dashed] (9.5, 3) -- (9.5, 1);
        \draw [black!60!green, ultra thick] (9.5,1) -- (9.5,-1);
        \draw [black!60!green, ultra thick, dashed] (9.5,-1) -- (9.5,-5);
        
        \draw [black!60!green, ultra thick, dashed] (12.5,3)--(12.5,1);
        \draw [black!60!green, ultra thick] (12.5,1) -- (12.5,-1);
        \draw [black!60!green, ultra thick, dashed] (12.5,-1)--(12.5,-5);
    \end{tikzpicture}
    \caption{Second cut of $P_j$'s into standard pieces.\label{sector:second-cut2}}
\end{figure}

Now we can apply the Sectorial Descent (Theorem \ref{Prelim:sectorial-descent}) with respect to the covering $\{P_{j\alpha}\}$, which gives the following.

\begin{proposition}
The partially wrapped Fukaya category $\mathcal{W} (M(\mathbb{V}),\xi)$ is generated by that of its covering sectors $\mathcal{W}(P_{j\alpha})$, where $1\leq j\leq N_\mathbb{V}$ and $\alpha\in\{-,0,+\}^{d-1}$.
\end{proposition}

This reduces the generation problem of $\mathcal{W}(M(\mathbb{V}),\xi)$ to that of the Fukaya categories $\mathcal{W}(P_{j\alpha})$.

\subsection{Generation for Each Piece.}\label{section:piece} 
In this subsection, we study the generation for the wrapped Fukaya categories of the Weinstein subsectors $P_{j\alpha}\subset M(\mathbb{V})$ defined in the last subsection. Notice that there are two types of such subsectors: those containing a crossing (i.e. a $0$-dimensional intersection point formed by $d$ complex hyperplanes in $\mathbb{V}$), denoted simply by $P_\times$, and those do not contain a crossing, which we denote by $P_\vert$.

\begin{proposition}\label{sector:GenerationResult}
    \begin{enumerate}[i)]
        \item The Liouville sector $P_\times$ is symplectomorphic to
        \[
            \left(\left(\mathbb{C}^{\ast}\right)^d, \bigcup_{i=1}^d\left(f_i\right)^{-1} (c_i^{\pm})\right),
        \]
        where $f_i$ is the pull-back of $\xi_i$ under the symplectic embedding $P_\times \hookrightarrow\left(\mathbb{C}^{\ast}\right)^d$ and $c_i^{\pm} \in \{c_{ik}:1 \le k \le 2N_{\mathbb{V}}\}$ 
        \item The other Liouville sectors $P_{\vert}$ are symplectomorphic to the stabilizations
        \[
            \left(P_{\vert}\cap (\xi_1)^{-1} (c_1^+)\right)\times T^{\ast} [c_1^-,c_1^+],
        \]
where $c_1^{\pm} \in \{c_{1j}: 1 \le j \le N_{\mathbb{V}}\}$.
    \end{enumerate}
\end{proposition}

\begin{remark}
It follows from i) of the proposition above that the wrapped Fukaya category $\mathcal{W} (P_\times)$ is generated by a cotangent fiber of $\left(\mathbb{C}^{\ast}\right)^d\cong T^{\ast}{T}^d$, together with the linking disks associated to the stops $(\xi_i)^{-1} (c_i^{\pm})$. The cotangent fibre $L\subset P_\times$ will be called a \emph{standard Lagrangian} associated to $P_\times$. Order the crossing points from the bottom to the top according to the values of $\rRe(\xi)$, the sector $P_\times$ corresponds to a unique crossing point, say the $j$-th crossing, and we will write the corresponding standard Lagrangian as $L_j$. 

Note that here we use the term ``standard Lagrangian'' so that it is consistent with the term used in geometric representation theory, where the standard objects in the hypertoric category $\mathcal{O}$ (cf.\cite{BLPW2010}) correspond exactly (by our main theorem) to these Lagrangians.

On the other hand, ii) of the proposition above implies that the wrapped Fukaya category $\mathcal{W} (P_{\alpha})\cong\mathcal{W}\left((\xi_1)^{-1} (c_1^+)\cap P_{\alpha}\right)$ is generated by the linking disks associated to the stop $P_{\alpha}\cap(\xi_1)^{-1} (c_1^+)$.
\end{remark}

\begin{proof}[Proof of Proposition \ref{sector:GenerationResult} i)]
Let $J\subset\{1,\cdots,n\}$ be the set of hyperplanes going through the dimension $0$ crossing point contained in $P_\times$. We apply Lemma \ref{sector:Standardization} to get a corresponding symplectic structure on $P_\times$, regarded as a (codimension $0$) symplectic submanifold of $M(\mathbb{V}[J])$, where $\mathbb{V}[J]$ is the hyperplane arrangement whose first $d$ members are coordinate hyperplanes. Recall also the symplectomorphism $\Phi:M(\mathbb{V}[J])\rightarrow M(\mathbb{V})$ from Lemma \ref{sector:Standardization}. Consider the symplectic isotopy
    \[
      \omega_t\coloneqq 4\begin{pmatrix}
        0&D+(1-t)A^TA\\ 
        -D-(1-t)A^TA&0
        \end{pmatrix},
    \]
whose derivative (with respect to $t$) is
\begin{equation} \nonumber
\begin{split}
 \dot{\omega}_t &=4\begin{pmatrix}
            0&-A^TA\\ 
            A^TA&0
        \end{pmatrix} \\
    &=-4d\left(\sum_{j\not\in J}\left(1-\frac{1}{\vert\ell_j\vert^4}\right)\left(\rRe(\ell_j)d\rIm(\ell_j) -\rIm(\ell_j)d\rRe(\ell_j)\right)\right)=d\theta_t.
\end{split}
\end{equation}
Let $V$ be the symplectic dual to the $1$-form $\theta_t$, we need to prove that $V$ is complete.
    
Consider a partial compactification $M\left(\mathbb{V}_J\right)$ of $M(\mathbb{V})$, where $\mathbb{V}_J$ is hyperplane arrangement obtained by removing those with indices in $J$. The stratification of the complement $M\left(\mathbb{V}_J\right)\setminus M(\mathbb{V})$ is determined by equations $\{\ell_i =0\vert \forall i\in J'\}_{J' \subset J}$. With respect to the symplectic structure $\omega_t$, the vector field $V$ is given by the formula
\begin{equation}\label{eq:V}
V=-[d\phi_J](D+(1-t)A^TA)^{-1}\begin{pmatrix}
\frac{\partial}{\partial\vec{x}}\\ \frac{\partial}{\partial\vec{y}}
\end{pmatrix},
\end{equation}
    where $\phi_J =\displaystyle\sum_{j\not\in J}\left(\vert\ell_j\vert^2 +\frac{1}{\vert\ell_j\vert^2}\right)$ is the squared sum and $[d\phi_J]$ is the matrix corresponding to its differential, which is given by
    \[
        d\phi_J =\sum_{j\not\in J}\sum_{i=1}^d\left(1-\frac{1}{\vert\ell_j\vert^4}\right)\rRe(\ell_j)a_{ij}dx_i.
    \]

    In the sector $P_\times$, we have $\vert\ell_j\vert$ bounded away from zero and $\rRe(\ell_j)$ is bounded for all $j\notin J$. If $\vert z_i\vert\to 0$ for some $1\leq i\leq d$, the denominator of the coefficients of $\partial_{x_i}$ and $\partial_{y_i}$ in the expression (\ref{eq:V}) go to $\infty$, so the vector field $V$ extends to a vector field over the open strata $\{z_i=0\}$ and is tangent to it. To show the completeness of $V$, it suffices to look at the case when some $\rIm(\ell_i)\to\infty$. Under this limit, the coefficient of $\partial_{x_i}$ is always bounded. Since all $\vert\ell_j\vert$'s are bounded away from $0$, for any sequence of points in $\Phi (P_\times)$ converging to any strata at infinity,
    \[
        [d\phi_J] (D+(1-t)A^TA)^{-1}
    \]
    converges. It follows that near infinity, we have
    \[
        \left\Vert [d\phi_J](D+(1-t)A^TA)^{-1}\right\Vert\leq M
    \]
    for some constant $M>0$, which then implies that the coefficient of $\partial_{y_i}$ is controlled by $CM\vert z\vert$, where $C>0$ is another constant. It is a consequence of Gr\"onwall's inequality \cite[Theorem 2.8]{teschl2012ordinary} that the vector field $V$ is complete. Thus the time-$1$ flow of $V$ induces a symplectic embedding $\phi_V^1\colon \Phi (P_\times)\hookrightarrow\left(\mathbb{C}^{\ast}\right)^d$, which can be extended to an exact symplectomorphism. It implies that
    \begin{equation} \nonumber
    \mathcal{W}(P_\times)\simeq\mathcal{W} (\Phi (P_\times))\simeq\mathcal{W}\left(\left(\mathbb{C}^{\ast}\right)^d,\bigcup_{i=1}^d(f_i)^{-1} (c_i^{\pm})\right)
    \end{equation}
    is generated by a cotangent fiber of $T^\ast T^d$ and linking disks associated to the stops $(\xi_i)^{-1} (c_i^{\pm})$.
\end{proof}

\begin{proof}[Proof of Proposition \ref{sector:GenerationResult} ii)]
Recall that in the proof of Lemma \ref{sector:GramSchmidtProcess}, we only use the fact that the hypersurfaces $\left\{\rRe(\xi_1)^{-1} (c_1^{\pm})\right\}$ intersects with the removed hyperplanes transversely, and we cut the Liouville sector so that $\rRe(\xi_1)^{-1} (r)$, for $c_1^-\leq r\leq c_1^+$, intersect the removed hyperplanes in the closure of $P_{\alpha}$ transversely, so we conclude that for every $c_1^-\leq r\leq c_1^+$, there exists a product decomposition of the $Z$-invariant neighbourhood 
\begin{equation}\label{eq:decomposition}
\Nbd^Z\left(\{\rRe(\xi_1) =r\}\right) \cap P_\vert\xrightarrow{\simeq}\{\xi_1 =r\}\times T^{\ast} [-\varepsilon_r ,\varepsilon_r] \cap P_\vert
\end{equation}
for some small $\varepsilon_r >0$. The product decomposition (\ref{eq:decomposition}) for each $r$ gives us a locally defined vector field $X_{I_i^r}$, for some function $I_1^r:\Nbd^Z\left(\{\rRe(\xi_1) =r\}\right)\cap P_\vert\rightarrow\mathbb{R}$ from Lemma \ref{sector:GramSchmidtProcess}. The open intervals $(r-\varepsilon_c, r+\varepsilon_c)$ form a covering of $[c^-,c^+]$, so we can find a finite sub-collection of the $(c-\varepsilon_c, c+\varepsilon_c)$'s covering the same interval. Let $\{\psi_r\}$ be a partition of unity defined on $P_{\alpha} =\bigcup_r\Nbd^Z\left(\{\rRe(\xi_1)=r\}\right) \cap P_\vert$ that is subordinated to the finite cover, then the vector field $V=\sum_r\psi_r X_{I_1^r}$ satisfies $d\rRe(\xi_1) (V)=1$ everywhere and is locally complete by Lemma \ref{sector:CompletenessCriterion}. Therefore the vector fields $\left(X_{\rRe(\xi_1)}, V\right)$ induce the required product decomposition in the second part of the Proposition \ref{sector:GenerationResult}. 
\end{proof}

Proposition \ref{sector:GenerationResult} together with Sectorial Descent (Theorem \ref{Prelim:sectorial-descent}) imply that  the wrapped Fukaya category of $(M(\mathbb{V}), \xi=\xi_1)$ is generated by the union of all standard Lagrangians $\{L_{j}\}_{j=1}^{N_\mathbb{V}}$ and the image under the cup functors $\cup:\mathcal{W}\left(\xi_i^{-1} (c_{ik})\cap P_{j\alpha}\right)\rightarrow\mathcal{W} (P_{j\alpha}) \to \mathcal{W}(M(\mathbb{V}), \xi_1)$ for all possible $i$ and $j$ (i.e. the linking disks corresponding to the cornered stops $P_{j\alpha}$). Here we want to rule out the generators coming from the stops $\xi_i^{-1} (c_{ik})$ for $i\geq 2$ (see Theorem \ref{generation:generating-by-standard-Lags}). In the next section, contributions to the generating collection of $\mathcal{W} (M(\mathbb{V}),\xi)$ from the stops $\xi_1^{-1} (c_{1j})$ will also be ruled out based on several formulae provided by Lagrangian surgery, so that the wrapped Fukaya category $\mathcal{W} (M(\mathbb{V}) ,\xi)$ will be shown to be generated by the standard Lagrangian submanifolds. In particular, since the standard Lagrangians are quasi-isomorphic to iterated mapping cones of bounded-feasible chamber Lagrangians (Section \ref{sec:Decomposition of Lag}), this concludes the proof Theorem \ref{theorem:generation1}.

We end this subsection with the following theorem.

\begin{theorem}\label{generation:generating-by-standard-Lags}
The partially wrapped Fukaya category $\mathcal{W}(M(\mathbb{V}),\xi)$ is generated by the standard Lagrangians $\{L_{j}\}_{1\leq j\leq N_\mathbb{V}}$ constructed in Proposition \ref{sector:GenerationResult} and the linking disks associated to the stops $\xi_1^{-1}(c_{1j})$. 
\end{theorem}
\begin{proof}
For simplicity, we call a sector $P_{j\alpha}$ of type $\times$ if it contains a crossing point  (i.e. $P_\times$), and of type $\vert$ if it is a stabilization (i.e. $P_\vert$). The wrapped Fukaya categories of sectors of type $\vert$ are generated by the image of the cup functor $\cup\colon\mathcal{W}\left(P_{j\alpha}\cap \xi_1^{-1} (c_{1j})\right)\to\mathcal{W}(P_{j\alpha})$. 
On the other hand, for each $j$, the wrapped Fukaya categories of the type $\times$ sector $\mathcal{W}(P_{j0})$ is generated by the distinguished Lagrangian $L_j$ associated to the crossing point together with the images of the cup functors from $\mathcal{W}\left(P_{j\alpha}\cap\xi_i^{-1}(c_{i,2j-1})\right)$, $\mathcal{W}\left(P_{j\alpha}\cap\xi_i^{-1}(c_{i,2j})\right)$ and $\mathcal{W}\left(P_{j\alpha}\cap\xi_1^{-1}(c_{1j})\right)$.

To see that the  contributions from the functors $\cup\colon\mathcal{W} (P_{j0}\cap\xi_i^{-1} (c_{i,2j-1}))\to\mathcal{W} (P_{j0})\to\mathcal{W} (P_j)$ and $\cup\colon\mathcal{W} (P_{j0}\cap\xi_i^{-1} (c_{i,2j}))\to\mathcal{W} (P_{j0})\to\mathcal{W} (P_j)$ for $i\geq 2$ are redundant, we consider the following gluing diagram
\begin{equation}\label{diag:gluing-horizontal-pieces}
    \begin{tikzcd}
        \mathcal{W} (P_{j0}\cap\xi_i^{-1} (c_{i,2j-1}))\arrow[r]\arrow[d]&\mathcal{W} (P_{j,(0,\dotsb, 0,-,0,\dotsb, 0)})\arrow[d]\\ 
        \mathcal{W} (P_{j0})\arrow[r]&\mathcal{W} (P_{j0}\cup P_{j,(0,0,\dotsb, 0,-,0,\dotsb, 0)}),
    \end{tikzcd}
\end{equation}
where the top arrow and the left arrows are cup functors.
Since the sector $P_{j,(0,\dotsb, 0,-,0,\dotsb, 0)}$ is of type $\vert$ , the cup functor $\mathcal{W} (P_{j,(\dotsb)}\cap\xi_1^{-1} (c_{1j}))\to\mathcal{W} (P_{j,(\dotsb)})$ is a quasi-equivalence, and hence images of the top arrow of \eqref{diag:gluing-horizontal-pieces}
for $i\geq 2$, are generated by images of Lagrangians in $P_{j,(\dotsb)}\cap\xi_1^{-1} (c_{1j})$ in $\mathcal{W} (P_{j,(\dotsb)})$.

Therefore by taking homotopy colimit, the wrapped Fukaya category $\mathcal{W} (M(\mathbb{V}),\xi)$ is generated by the union of all images of the cup functor $\cup\colon\mathcal{W}\left(P_{j\alpha}\cap\xi_1^{-1}(c_{1j})\right)\to\mathcal{W} (P_{j\alpha})\to\mathcal{W} (M(\mathbb{V}),\xi)$ together with $\{L_j\}_{1\leq j\leq N_{\mathbb{V}}}$. 
The remaining task is to compare these cup functors with linking disks associated to $\xi_1^{-1}(c_{1j})$. 


Fix $j$, note that the union $\displaystyle\bigcup_{\alpha\in\{-,0,+\}} P_{j\alpha}\cap\xi_1^{-1} (c_{1j})=\xi_1^{-1} (c_{1j})$ provides a sectorial cover of $\xi_1^{-1} (c_{1j})$, so by Sectorial Descent (Theorem \ref{Prelim:sectorial-descent}) the wrapped Fukaya category $\mathcal{W}\left(\xi_1^{-1} (c_{1j})\right)$ is the homotopy colimit of the Fukaya categories $\mathcal{W}\left(P_{j\alpha}\cap\xi_1^{-1}(c_{1j})\right)$ via the inclusion functor $\mathcal{W}\left(P_{j\alpha}\cap\xi_1^{-1}(c_{1j})\right)\to\mathcal{W}\left(\xi_1^{-1}(c_{1j})\right)$. Since the maps of Liouville sectors form a commutative diagram
\begin{equation*}
    \begin{tikzcd}
        P_{j\alpha}\cap\xi_1^{-1} (c_{1j})\arrow[r, hook]\arrow[d, hook] & \xi_1^{-1} (c_{1j})\arrow[d, hook] \\
        P_{j\alpha}\arrow[r,hook]& P_j,
    \end{tikzcd}
\end{equation*}
the corresponding functors of wrapped Fukaya categories also fit into a homotopy commutative diagram 
\begin{equation*}
    \label{diag:commutative-fukaya-cat}
    \begin{tikzcd}
    \mathcal{W}\left(P_{j\alpha}\cap\xi_1^{-1}(c_{1j})\right)\arrow[r]\arrow[d]&\mathcal{W}\left(\xi_1^{-1}(c_{1j})\right)\arrow[d]\\ 
    \mathcal{W} (P_{j\alpha})\arrow[r]&\mathcal{W} (P_j).
    \end{tikzcd}
\end{equation*}
Taking the homotopy colimit for $\alpha$, we get a diagram 
\begin{equation*}
    \label{diag:commutative-fukaya-colim}
    \begin{tikzcd}
    \hocolim_{\alpha}\mathcal{W}\left(P_{j\alpha}\cap\xi_1^{-1}(c_{1j})\right)\arrow[r]\arrow[d]&\mathcal{W}\left(\xi_1^{-1}(c_{1j})\right)\arrow[d]\\
    \hocolim_{\alpha}\mathcal{W} (P_{j\alpha})\arrow[r]&\mathcal{W} (P_j),
    \end{tikzcd}
\end{equation*}
where the top and bottom functors are quasi-equivalences of $A_\infty$-categories, and the diagram is still commutative up to homotopy. The bottom left composition contributes linking disks associated to the sectorial corners to the generating collection, and we can invert the top arrow to see that these generating collections actually lie in the full $A_{\infty}$-subcategory of $\mathcal{W} (P_j)$ spanned by generating collections of $\mathcal{W}\left(\xi_1^{-1}(c_{1j})\right)$.
\end{proof}

\section{Exact triangle, Surgery and Isomorphic Objects}\label{Exact triangle, surgery and isomorphic objects}

In Theorem \ref{generation:generating-by-standard-Lags}, we showed that the partially wrapped Fukaya category $\mathcal{W} (M(\mathbb{V}),\xi)$ is generated by the standard Lagrangians together with the all the linking disks associated to the stop $\xi_1=\xi$. In this section, we will use surgery exact sequences to rule out the contributions of all the linking disks and show that the standard Lagrangians are generated by Lagrangians associated to bounded feasible chambers of $\mathbb{V}$. We start with a discussion about McLean type neighborhoods in the special case corresponding to a (complex) hyperplane arrangement, which is a slight modification of \cite{mclean2012growth}. 

\subsection{A McLean Type Neighborhood.}\label{sec:McLean}

Consider a projective completion of $M(\mathbb{V}) \subset \mathbb{C}^d$, denoted by $\overline{M}(\mathbb{V}) \subset \mathbb{CP}^d$. Let $D_\infty$ denote the complement $\mathbb{CP}^d \setminus \mathbb{C}^d \cong \mathbb{CP}^{d-1}$. By abuse of notations, we will still use $\{H_i\}_{i=1}^n$ to denote the collection of the completions of the complexified hyperplanes in the arrangement $\mathbb{V}$. The intersection loci $\bigcup_{i=1}^nH_i\cap D_\infty$ can be seen as a generic hyperplane arrangement $\mathbb{V}_\infty$ in $\mathbb{CP}^{d-1}$. In general, this contatins a singular locus of $\overline{M}(\mathbb{V})$, denoted by $\mathbb{V}_{sing}$. For example, singular locus will arise when there are parallel hyperplanes in $\mathbb{V}$.

One can desingularize $\overline{M}(\mathbb{V})$ by iterated blow-ups along $\mathbb{V}_{sing}$, starting from the deepest strata, which we describe as follows. First, blow up $\overline{M}(\mathbb{V})$ along all 0-dimensional strata of the intersections of the hyperplanes $H_i$ in $D_\infty$. The proper transforms of the hyperplanes can only intersect along dimension $1$ or higher. We then blow up along $1$-dimensional strata of the intersections of proper transforms of the hyperplanes $H_i$ and $D_\infty$, and so on. Since the blow-up loci do not intersect with each other, this iterated blow-up is independent of the choice of order among the strata of the same dimension in $\mathbb{V}_\infty$. Performing this procedure inductively until the $(d-2)$-th step, we obtain a non-singular projective variety $X_\mathbb{V}$ with a simple normal crossing divisor $D_\mathbb{V}$ such that $M(\mathbb{V})=X_\mathbb{V}\setminus D_\mathbb{V}$. The divisor $D_\mathbb{V}$ is linearly equivalent to 
\begin{equation} \nonumber
H_1 + \cdots + H_n - \sum{a_i E_i},
\end{equation}
for some positive integers $a_i\in\mathbb{Z}_{>0}$, where the $E_i$'s are exceptional divisors. This is because at each step, there are at least two proper transforms of the hyperplanes $H_i$ intersecting with the blow-up locus, hence the coefficient of the corresponding exceptional divisor of the proper transform of $H_i$ is always negative. It follows that $D_\mathbb{V}$ supports an ample divisor, after a rescaling of the coefficients if necessary. The corresponding ample line bundle $\mathcal{L}\rightarrow X_\mathbb{V}$ has a section $s$ with $s^{-1}(0)=D_\mathbb{V}$. We equip $X_\mathbb{V}$ with a symplectic form $\omega_X:=-dd^\mathbb{C}\log ||s||$, where $||\cdot||$ is a metric on $\mathcal{L}$ whose associated curvature form $F$ has the property that $iF$ is a positive $(1,1)$-form. 

Let $\omega_{\log}$ be the restriction of $\omega_X$ to $M(\mathbb{V})$, and write  $M(\mathbb{V})_{\log}$ for $M(\mathbb{V})$ equipped with the $2$-form $\omega_{\log}$. We remove a small regular neighborhood $N_D$ of $D_\mathbb{V}$ from $X_\mathbb{V}$, so that its complement is a Liouville domain with respect to the $1$-form $\lambda_{\log}:=-d^\mathbb{C}\log ||s||$. Denote its completion by $\widehat{M}(\mathbb{V})_{\log}$, which is a Liouville manifold.

Recall that the hyperplanes $H_i$ are complexification of real hyperplanes $H_{\mathbb{R},i}$ (see Definition \ref{def:hyparr}).
Therefore, the blow-ups considered above can be defined over $\mathbb{R}$ and $X_\mathbb{V}$ has an induced real structure.
In other words, there is a real algebraic variety $X_{\mathbb{V}}(\mathbb{R})$ such that $X_\mathbb{V}=X_{\mathbb{V}}(\mathbb{R}) \times_{\operatorname{Spec}(\mathbb{R})} \operatorname{Spec}(\mathbb{C})$, so $X_\mathbb{V}$ comes with an anti-holomorphic involution (lifting the complex conjugation on $\mathbb{CP}^d$), which we denote by $\tau$.
By construction, we have $\tau(D_\mathbb{V})=D_\mathbb{V}$.
By possibly replacing $||\cdot||$ with a $\tau$-equivariant metric on $\mathcal{L}$ (after lifting $\tau$ to $\mathcal{L}$),  we may assume that $\tau^*\lambda_{\log}=-\lambda_{\log}$ and $\tau^*\omega_X=-\omega_X$.
By choosing the neighborhood $N_D$ to be $\tau$-invariant and using the cone coordinates, we can extend $\tau$ to an anti-symplectic involution over $\widehat{M}(\mathbb{V})_{\log}$, which we continue to denote by $\tau$ by abuse of notations.

\begin{lemma}
For any two sufficiently small choices of the neighborhood $N_D$, the resulting Liouville manifolds $\widehat{M}(\mathbb{V})_{\log}$ are $\tau$-equivariantly symplectomorphic.
\end{lemma}

\begin{proof}
Suppose that $N_{D,1}$ and $N_{D,2}$ are two different choices.
There is an even smaller $\tau$-invariant neighborhood  $N_D' \subset N_{D,1} \cap N_{D,2}$, so by transitivity of an equivalence relation, it suffices to assume that $N_{D,2}=N_D' \subset N_{D,1}$.
In this case, there is a canonical isomorphism between the respective completions, given by the identity outside of $N_{D,1}$, and matches with the Liouville flow starting from the boundary of $N_{D,1}$. It is clearly $\tau$-equivariant.
\end{proof}

\begin{lemma}[\cite{mclean2012growth} Lemma 5.18]\label{lem:conv deform equiv}
$\widehat{M}(\mathbb{V})_{\log}$ is $\tau$-equivariantly convex deformation equivalent to $M(\mathbb{V})$, where the latter is equipped with the Liouville structure specified in Section \ref{ss:essential}.
\end{lemma}

\begin{proof}
The fact that $\widehat{M}(\mathbb{V})_{\log}$ is convex deformation equivalent to $M(\mathbb{V})$ is proved in \cite[Lemma 5.18]{mclean2012growth}. The proof directly generalizes to the $\tau$-equivariant setting.
\end{proof}

We recall the McLean type neighborhood theorem.

\begin{proposition}[\cite{mclean2012growth} Lemma 5.14]\label{prop:Mclean}
Let $\widehat{M}(\mathbb{V})_{\log}$ be defined as above. For each $i \in \{1,\dots,n\}$, there is an open neighborhood $U_i \subset \mathbb{C}^d$ of $H_i$, together with a projection $\pi_i: U_i \to H_i$,
such that the following statements are true:
\begin{enumerate}[(i)]
\item  For all $i,j$, we have $\pi_i \circ \pi_j=\pi_j \circ \pi_i$ in $U_i \cap U_j$.
\item For any $I \subset \{1,\dots,n\}$, the compositions of $(\pi_i)_i\in I$, which we denote by $\pi_{I}:U_{I}:=\bigcap_{i \in I} U_i \to H_{I}:=\bigcap_{i \in I} H_i$, has fibres symplectomorphic to 
\begin{align}\label{eq:localform2}
\left(\prod_{i \in I} \mathbb{D}, r_idr_i\wedge d \theta_i\right),
\end{align}
where $\mathbb{D}$ is the unit disk centered at the origin, and the intersection of a fibre with $U_{I} \cap H_j$ is mapped to $\{0\}_j \times \prod_{i \in I \setminus \{j\}} \mathbb{D}$ under the symplectomorphism.
\item For any $j \in I$, the restriction of the map $\pi_j$ to any fibre $\prod_{i \in I} \mathbb{D}$ of $\pi_{I}$ is the natural projection
\[\prod_{i \in I} \mathbb{D}_{\epsilon} \to \{0\}_j \times \prod_{i \in I \setminus \{j\}} \mathbb{D},
\]
where $\mathbb{D}_\epsilon$ is the disk of radius $\epsilon$ centered at the origin.
\item  The symplectic parallel transport map of  $\pi_{I}:U_{I} \to H_{I}$ has holonomy in $\prod_{i \in I} U(1)$.
\item Let $\tau:\mathbb{C}^d \to \mathbb{C}^d$ and $\tau_{\mathbb{D}}:\mathbb{D} \to \mathbb{D}$ be complex conjugations, and let $\tau_I:H_I \to H_I$ be the restriction of $\tau$.
For any $I \subset \{1,\dots,n\}$, 
there is an open $\tau_I$-invariant neighborhood $B_{I}$ of $L_{I}^{\circ}:=\Fix(\tau_I) \setminus (\bigcup_{I \subset J} U_{J}$) in $H_{I}$, together with a $(\tau_I \times \prod_{i \in I} \tau_{\mathbb{D}},\tau)$-equivariant fibre-preserving symplectomorphism
\begin{align}\label{eq:localTri}
\left(B_I \times \prod_{i \in I} \mathbb{D}, \omega_{H_{I}}|_{B_I}+\sum_{i \in I} r_idr_i\wedge d \theta_i \right) \simeq \left(U_I|_{B_I}, \omega_{M(\mathbb{V})}|_{U_{I}^{\circ}} \right),
\end{align}
where
$U_I|_{B_I} :=\pi_{I}^{-1}(B_I)$.
In particular, it implies that for any sign sequence $\alpha\in\mathscr{F}(\mathbb{V})$ with $U_I \cap L_{\alpha} \neq \emptyset$, after possibly rotating the disk factors by $e^{\sqrt{-1}\pi}$, the $(\tau_I \times \prod_{i \in I} \tau_{\mathbb{D}},\tau)$-equivariant symplectomorphism sends
\begin{equation} \nonumber
L_{I}^{\circ} \times \prod_{i \in I} \left\{(r_i,\theta_i) \in \mathbb{D}^* |\alpha(i) e^{\sqrt{-1}\theta_i}=1\right\},
\end{equation}
which is an open subset of $\Fix(\tau_I \times \prod_{i \in I} \tau_{\mathbb{D}})$, to $U_I|_{B_I} \cap L_{\alpha}$.
\end{enumerate}
\end{proposition}

\begin{proof}
The proposition can be proved by generalizing \cite[Lemma 5.14]{mclean2012growth} to incorporate the anti-symplectic involutions. For the sake of completeness, we present here the details of proof.

Define a total order for subsets of $\{1,\dots,n\}$ as follows: if $|I_1|>|I_2|$, then $I_1 <I_2$, and if $|I_1|=|I_2|$ and $\min (I_1 \setminus I_2) < \min (I_2 \setminus I_1)$, then $I_1<I_2$. 
We are going to construct $U_i$, $\pi_i$ and $B_I$ inductively, starting from its description near $H_I$ for small $I$ and progressing to large $I$.

We will construct $U_i$ so that $U_I =\emptyset$ if $H_I=\emptyset$ so the base case is when $|I|=d$ and $H_I$ is a point. Define $B_{I}=H_I$.
Using a local $\prod_{i \in I} \tau_{\mathbb{D}}$-invariant metric and the corresponding exponential map, we can  find a small neighborhood $U_I$ of the point $H_I$ and a $(\prod_{i \in I} \tau_{\mathbb{D}}, \tau)$-equivariant diffeomorphism  $\phi_I: \prod_{i \in I} \mathbb{D} \to U_I$
such that $\phi_I(\{0\}_j \times \prod_{i \in I\setminus{j}} \mathbb{D})=U_I \cap H_j$ for all $j$.
Moreover, there are constants $a_i>0$ such that $\phi_I^*\omega_{\log}=\sum_{i \in I} a_ir_idr_i\wedge d \theta_i$.
By rescaling $\phi_I$ factor by factor, we may assume that $a_i=1$ for all $i$.
Since $\tau$ is an anti-symplectic involution, by possibly shrinking $\prod_{i \in I} \mathbb{D}$, we can isotope $\phi_I$ to another $(\prod_{i \in I} \tau_{\mathbb{D}}, \tau)$-equivariant symplectomorphism
$\phi_I: \prod_{i \in I} \mathbb{D} \to U_I$
such that $\phi_I(\{0\}_j \times \prod_{i \in I\setminus{j}} \mathbb{D})=U_I \cap H_j$ and
 $\phi_I^*\omega_{\log}$ is 
of the form \eqref{eq:localform2}.
We define $\pi_i:U_I \to H_I$ by $ \phi_I \circ p_i \circ \phi_I^{-1}$, where $p_i:\prod_{j \in I} \mathbb{D} \to \prod_{j \in I \setminus\{i\}} \mathbb{D}$ is the natural projection.
Since $\phi_I$ is $(\prod_{i \in I} \tau_{\mathbb{D}}, \tau)$-equivariant, it sends the fixed point set of $\prod_{i \in I} \tau_{\mathbb{D}}$ to that of $\tau$. By possibly rotating the disk factors by $e^{\sqrt{-1}\pi}$, it implies that  $\phi_I^{-1}(U_I \cap L_{\alpha})$ equals $\prod_{i \in I}\left\{(r_i,\theta_i) \in \mathbb{D}^* |\alpha(i) e^{\sqrt{-1}\theta_i}=1\right\}$. This completes the base case of the induction.

Let $I \subset \{1,\dots,n\}$. Assume the induction hypothesis holds for all $J<I$. More precisely, it means that we have constructed $\{\pi_i\}_{i=1}^n$ and $\{U_i\}_{i=1}^n$ near $H_J$ such that all the listed properties are satisfied. We need to extend the definitions over a neighborhood of $H_I$ such that all the listed properties continue to hold. In fact, we only need to extend $\pi_i$ and $U_i$ for $i \in I$ in a neighborhood of $H_I$.

First note that, as a complete intersection, the structure group of the normal bundle $\nu_{H_I}$ of $H_I\subset X_\mathbb{V}$ is $\prod_{i \in I} U(1)$. 
Using partition of unity, exponential map and the induction hypothesis, we can construct a diffeomorphism $\phi_I$ from a polydisk subbundle of $N_{H_I}$ to a neighborhood of $H_I$ such that near $\phi_I^{-1}(H_I \cap H_J)$, for $J<I$. $\phi_I^*\omega_{\log}$ and $\phi_I^*\pi_i$ already satisfy all the listed properties and $\phi_I^*\pi_I$ coincides with the projection $\nu_{H_I} \to H_I$.
By an abuse of notation, we denote the polydisk subbundle to be $\nu_{H_I}$.
Moreover, since every connected component of $L_{I}^{\circ}$ is contractible, we can choose a smooth trivialization of $\nu_{H_I} \to H_I$ over a $\tau_I$-invariant neighborhood $B_I$ of $L_{I}^{\circ}$, which is of the form $B_I \times \prod_{i \in I} \mathbb{D}$ (where the structure group $\prod_{i \in I} U(1)$ respects the product decomposition), and assume that $\phi_I$ is $(\tau_I \times \prod_{i \in I} \tau_{\mathbb{D}},\tau)$-equivariant over $B_I \times \prod_{i \in I} \mathbb{D}$.

On the other hand, we can construct a symplectic form $\omega'$ on $\nu_{H_I}$ such that
\begin{itemize}
\item $\omega'=\phi_I^*\omega_{\log}$ near $\phi_I^{-1}(H_I \cap H_J)$, for $J$ less than $I$, 
\item over $H_I^{\circ}:=H_{I}\setminus \left(\bigcup_{J<I} \phi_I^{-1}(H_I \cap H_J)\right)$, the bundle $\nu_{H_I}|_{H_I^{\circ}} \to H_I^{\circ}$ has fibres diffeomorphic to \eqref{eq:localform2}, with the structure group respects the product structure,
\item the symplectic parallel transport map of $\nu_{H_I}|_{H_I^{\circ}} \to H_I^{\circ}$ has holonomy in $\prod_{i \in I} U(1)$,
\item over the trivialization $B_I \times \prod_{i \in I} \mathbb{D}$, $\omega'$ is of the form \eqref{eq:localTri}.
\end{itemize}
The construction is called a `bundle compatible' symplectic form in \cite[Lemma 5.14]{mclean2012growth}.
Note that, in our setup, the action of $\tau_I \times \prod_{i \in I} \tau_{\mathbb{D}}$ on $B_I \times \prod_{i \in I} \mathbb{D}$  is anti-symplectic with respect to $\omega'$.

Since $\omega'|_{H_I^{\circ}}=\omega_{H_I^{\circ}}$ and $\omega'=\phi_I^*\omega_{\log}$ near $\phi_I^{-1}(H_I \cap H_J)$, for $J<I$, we can apply Moser's trick to the linear interpolation of $\omega'$ and $\phi_I^*\omega_{\log}$. The upshot is that we are able to modify $\phi_I$ so that $\phi_I^*\omega_{\log}=\omega'$.
Inside $N_{H_I}$, there is a natural projection map corresponding to the $i$-th factor of the structure group, denoted by $p_i$, which extends $\phi_I^*\pi_i$ from an open neighborhood of $\phi_I^{-1}(H_I \cap H_J)$ to the entire bundle.
We can therefore extend $U_i$, for $i \in I$, to a neighborhood of $H_I$ such that $U_I=\phi_I(N_{H_I})$.
The maps $\pi_i$, $i \in I$ can also be extended to a neighborhood of $H_I$ by $\phi_I \circ p_i \circ \phi_I^{-1}$. 
It is immediate to check that they satisfy all the listed properties.
\end{proof}

\subsection{Clean Surgery Formula and Lagrangian Cobordisms.\label{sec:clean-surgery}} 

In this subsection, we briefly review the work on clean surgery formulae for closed Lagrangian submanifolds \cite{mak2018dehn}, and extend it to cylindrical exact Lagrangian submanifolds. Combining this with the recent work of Bosshard \cite{bosshard2022lagrangian}, we show that clean surgery of cylindrical exact Lagrangian submanifolds will induce exact triangles in the wrapped Fukaya category. We begin with a review of Lagrangian cobordisms and Lagrangian surgeries.

A \emph{Lagrangian cobordism} $V\subseteq\mathbb{C}\times X$ is a cylindrical exact Lagrangian submanifold with respect to $\lambda +\lambda_{\mathbb{C}}$, where $\lambda$ is the Liouville $1$-form on $(X,\omega ,Z)$ dual to $Z$, and $\lambda_{\mathbb{C}}$ is the standard Liouville $1$-form on $\mathbb{C}$, with the following additional requirements: there exist (sectorial) Liouville subdomains $\mathbb{C}^{\mathrm{int}}\subseteq\mathbb{C}$ and $X^{\mathrm{int}}\subseteq X$ such that

\begin{enumerate}[(a)]
	\item $V$ is cylindrical outside $\mathbb{C}^{\mathrm{int}}\times X^{\mathrm{int}}$,
	\item $V\cap (\mathbb{C}\setminus\mathbb{C}^{\mathrm{int}})\times X$ consists only of finitely many submanifolds that are cylindrizations $\gamma_j\tilde{\times}L_j$ of product Lagrangian submanifolds of cylindrical rays $\gamma_j$ in $\mathbb{C}$ and exact Lagrangian submanifolds $L_j\subseteq X$. These cylindrized products are called \emph{ends} of the cobordism,
	\item $V$ is disjoint from $\mathbb{C}\times\partial X$ where $\partial X$ is the sectorial boundary of $X$. 
\end{enumerate}

In the above, the \emph{cylindrization} $\gamma_j\tilde{\times}L_j$ is a deformation of the product $\gamma_j\times L_j$, so that it becomes cylindrical in $\mathbb{C}\times X$. We refer the reader to \cite{ganatra2018sectorial}, Section 7.2 for its construction.

\begin{theorem}[\cite{bosshard2022lagrangian}, see also \cite{BC1, BC2, Tanaka18}]
	Let $V$ be a Lagrangian cobordism with ends $L_0, L_1, \dotsb ,L_n$ (oriented counterclockwisely), there is an iterated cone decomposition
	\[
		L_0\cong [L_n\to L_{n-1}\to\dotsb\to L_2\to L_1]
	\]
	in the derived wrapped Fukaya category $D^\mathit{perf}\mathcal{W} (X)$.
\end{theorem}

We recall how to construct Lagrangian cobordisms from surgeries. We follow the terminology in \cite{mak2018dehn}. A function $\nu_{\lambda}\colon [0,\varepsilon]\to [0,\lambda]$ is called \emph{$\lambda$-admissible} if
\begin{enumerate}[(1)]
	\item $\nu_{\lambda} (0)=\lambda >0$ and $\nu_{\lambda}'(r)<0$ for $r\in (0,\varepsilon)$;
	\item $\nu_{\lambda}^{-1} (r)$ and $\nu_{\lambda} (r)$ have vanishing derivatives of all orders at $r=\lambda$ and $r=\varepsilon$, respectively.
\end{enumerate}
The graph of a $\lambda$-admissible function is depicted in Figure \ref{sec:admissiblefunction}.

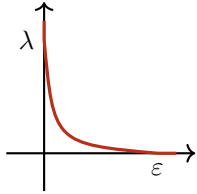
\begin{figure}[ht]
	\centering
	\begin{tikzpicture}[scale=0.5]
		\draw [->, thick] (-1,0) -- (3,0) node [below] {$\varepsilon$} -- (4,0);
		\draw [->, thick] (0,-1) -- (0,3) node [left] {$\lambda$} -- (0,4);
		\draw [BrickRed, very thick] (0,3.5) -- (0,3) .. controls (0.25,0.25) .. (3,0) -- (3.5,0);
	\end{tikzpicture}
	\caption{$\lambda$-admissible Function\label{sec:admissiblefunction}}
\end{figure}

The following lemma is a generalization of Weinstein neighborhood for cylindrical Lagrangian submanifolds.

\begin{lemma}\label{l:nbhd}
Let $L$ be a properly embedded exact Lagrangian submanifold of $(X,\omega ,Z)$ that is cylindrical near infinity. There exists an open neighbourhood $W$ of $L$ that is invariant under the $Z$-action outside of a compact subset and is symplectomorphic to an open neighbourhood of $L$ in $T^{\ast} L$.
\end{lemma}
\begin{proof}
Consider the Lagrangian submanifold with boundary $\overline{L}=L\cap\overline{X}$ in the Liouville domain $\overline{X}$. 
Then $\partial \overline{L}$ is a closed Legendrian submanifold in the contact manifold 
$(\partial \overline{X}, \alpha:=\lambda|_{\partial \overline{X}})$.
By the Legendrian neighborhood theorem (see \cite[Corollary 2.5.9]{geiges}), there is a neighborhood $U$ of $\partial \overline{L}$ such that $U$ is contactomorphic to a neighborhood $V$ of $\partial \overline{L}$ in $J^1(\partial \overline{L})=\mathbb{R} \times T^*(\partial \overline{L})$ with the canonical contact form $\alpha_{can}$.
Denote the contactomorphism by $f:U \to V$ so we have $f^*\alpha_{can}=g\alpha$, where $g:U \to \mathbb{R}_{>0}$ is a positive function.
The map $F:\mathbb{R}_{>0} \times U \to \mathbb{R}_{>0} \times V$ given by $(r,u) \mapsto (g(u)r,f(u))$ satisfies $F^*d(r\alpha_{can})=d(\frac{r}{g}g\alpha)=d(r\alpha)$ so it is a symplectomorphism between the respective symplectizations.
Note that there is a symplectomorphism from the symplectization of $J^1(\partial \overline{L})$ to $T^*(\mathbb{R}_{>0} \times \partial \overline{L})$ which is identity on $\mathbb{R}_{>0} \times \partial \overline{L}$ (see the Lemma \ref{l:symplectizationproduct} below).
Therefore, the neighborhood $\mathbb{R}_{>0} \times U$ of $\mathbb{R}_{>0}\times \partial \overline{L}$ is symplectomorphic to a neighborhood of the zero section in  $T^*(\mathbb{R}_{>0} \times \partial \overline{L})$.
Let $W_{>1}:=\mathbb{R}_{>1} \times U$ which is a $Z$-invariant neighborhood of $L \setminus \overline{L}$ and is symplectomorphic to  a neighborhood of the zero section in  $T^*(\mathbb{R}_{>1} \times \partial \overline{L})$.
We can extend $W_{>1}$ to a neighborhood $W$ of $L$ and apply the Moser trick to get a Weinstein neighborhood of $L$. Since we have already built a symplectomorphism from $W_{>1}$ to a neighborhood of the zero section in $T^*(\mathbb{R}_{>1} \times \partial \overline{L})$, we can ensure that the Moser flow in the Moser trick is compactly supported so the result follows.
\end{proof}

\begin{lemma}\label{l:symplectizationproduct}
Let $K$ be a closed manifold. Then there is a symplectomorphism from the symplectization $\left(\mathbb{R}_{>0} \times J^1(K), d(r\alpha_{can})\right)$ to  the cotagent bundle $T^*(\mathbb{R}_{>0} \times K)$ sending $\mathbb{R}_{>0} \times K$ to the zero section.
\end{lemma}

\begin{proof}
We can wrtie $J^1(K)$ as $\mathbb{R} \times T^*K$ and $\alpha_{can}$ as $-dz+\lambda_\mathit{can}$ where $\lambda_\mathit{can}$ is the canonical Liouville form on $T^*K$. So 
\[
d(r\alpha_{can}))=dz \wedge dr+dr \wedge d\lambda_\mathit{can} + r d\lambda_\mathit{can}.
\]
Use the obvious diffeomorphism to identify $\mathbb{R}_{>0} \times \mathbb{R} \times T^*K$ with $T^* \mathbb{R}_{>0} \times T^*K=T^*(\mathbb{R}_{>0} \times K)$ so the cotagent bundle symplectic form is given by
$dz \wedge dr+ d\lambda_\mathit{can}$.
For $t \in [0,1]$, let 
\begin{align*}
\omega_t&:=dz \wedge dr+(1-t)(dr \wedge d\lambda_\mathit{can} + r d\lambda)+td\lambda_\mathit{can} \\
X_t&:= \frac{1-r}{(1-t)r+t}Z
\end{align*}
where $Z$ is the Liouville vector field satisfying $\iota_{Z} d\lambda_\mathit{can}=\lambda_\mathit{can}$.
It is clear that $d\iota_{X_t}\omega_t=d((1-r)\lambda_\mathit{can})=\frac{d}{dt}\omega_t$ and $X_t$ is integrable.
The result follows from the Moser argument.
\end{proof}

Using Lemma \ref{l:nbhd}, we can identify an open neighbourhood of $L$ in $X$ as a codisk bundle over $L$ of radius $\varepsilon >0$, with respect to the Riemannian metric on $T^{\ast} L$ induced by an appropriately chosen Riemannian metric on $L$ (one can use the symplectomorphism in Lemma \ref{l:symplectizationproduct} to show that there is a strictly positive $\varepsilon$ even though $L$ is non-compact). Let $D\subseteq L$ be a submanifold that is also cylindrical near infinity, and write $N_L^{\ast} D$ to be the conormal bundle of $D$ in $T^{\ast} L$.

\begin{definition}[\cite{mak2018dehn}, Definition 2.20]
We define \emph{flow handle} $H_{\lambda}(L,D)$ associated to the pair $(L,D)$ to be the space 
	\[
		H_{\lambda} (L,D)\coloneqq\left\{\varphi_{\nu_{\lambda} (\Vert\zeta\Vert)}^{\Vert\zeta\Vert} (\zeta)\middle\vert \zeta\in N^{\ast}_{L,\varepsilon} D\setminus D\right\},
	\]
where $\varphi_{\nu_{\lambda} (\Vert\zeta\Vert)}^{\Vert\zeta\Vert}$ is the time-$1$ Hamiltonian flow of a function $\tilde{\nu}_\lambda(||\zeta||)$, with $\tilde{\nu}_\lambda'(r)=\nu_\lambda(r)$.
\end{definition}

Let $L_1, L_2$ be two exact cylindrical Lagrangian submanifolds cleanly intersecting along a submanifold $D$. In \cite{mak2018dehn}, they showed that there is a symplectic automorphism of $T^{\ast} L_1$ so that $L_2$ becomes $N^{\ast}_{L_1} D$ in a given neighbourhood of $L_1$. The \emph{flow surgery} of $L_1$ and $L_2$ along $D$ is then the submanifold given by attaching the flow handle $H_{\lambda} (L,D)$ defined in an open neighbourhood $U$ of $D$ to $(L_1\cup L_2)\setminus U$. It follows from \cite[Corollary 2.22 and Lemma 6.2]{mak2018dehn} (see also \cite[Lemma 3.10]{bosshard2022lagrangian}) that we have

\begin{proposition}\label{prop:flow-surgery-cobordism}
Let $L=L_1\#_D L_2$ be the flow surgery of $L_1$ and $L_2$ along $D$, then there is a Lagrangian cobordism (the trace cobordism) from $L_1$ and $L_2$ to $L$.
\end{proposition}

In particular, taking flow surgery of two cylindrical Lagrangians $L_1, L_2\subset X$ along their clean intersection loci $D$ induces an exact triangle
\[
	L_1\to L_1\#_D L_2\to L_2\xrightarrow{+1} L_1[1]
\]
in $D^\mathit{perf}\mathcal{W}(X)$.

\subsection{Gluing Adjacent Chambers}\label{sec:gluingchambers}

We keep the notations in the previous sections and consider the Weinstein manifold $(M(\mathbb{V}), \omega)$, with the symplectic form $\omega$ defined in Section \ref{ss:essential}. The real locus $M(\mathbb{V}) \cap \mathbb{R}^d$ consists of (the interiors of) feasible chambers $\{\Delta_\alpha\}_{\alpha \in \mathscr{F}(\mathbb{V})}$ (cf. Section \ref{sec:hyparr}). It is straightforward that the interior of each chamber $\Delta_\alpha^\circ$, as a connected component of the real locus, is a Lagrangian submanifold of $(M(\mathbb{V}), \omega)$. We call such a Lagrangian submanifold a \emph{chamber Lagrangian}, denoted by $L_\alpha$. Also, such Lagrangians are preserved under the (complete) convex deformation from $(M(\mathbb{V}), \omega)$ to $(M(\mathbb{V}),\omega_{\log})$ in Lemma \ref{lem:conv deform equiv}, since they are $\tau$-equivariant. We shall abuse notations and denote the corresponding Lagrangians in $(M(\mathbb{V}),\omega_{\log})$ by the same notation $L_\alpha$.

For two sign sequences $\alpha_1, \alpha_2 \in \mathscr{F}(\mathbb{V})$, we say the chamber Lagrangian $L_{\alpha_1}$ is adjacent to $L_{\alpha_2}$ if $\alpha_1$ differs from $\alpha_2$ only at one element. By Lemma \ref{lem:conv deform equiv} and Proposition \ref{prop:Mclean}, we can glue two adjacent chamber Lagrangians via clean surgery. Suppose that $L_{\alpha_1}$ and $L_{\alpha_2}$ are two adjacent chamber Lagrangians with $\alpha_1(k)\neq\alpha_2(k)$ for some $1\leq k\leq n$. Take a neighborhood $U_k$ of $H_k$ as in Proposition \ref{prop:Mclean} and a fiberwise Hamiltonian function $\rho(r)$ supported in $U_k$ with Hamiltonian vector field $X_{\rho(r)}=\frac{\rho'(r)}{r}\frac{\partial}{\partial\theta}$, where $(r,\theta)$ are polar coordinates on the disk fibers of $U_k\rightarrow H_k$. Let $\tilde{L}_{\alpha_1}:=\phi^1_{X_{\rho(r)}}(L_{\alpha_1})$ be the time-$1$ Hamiltonian flow of $L_{\alpha_1}$. If there is another hyperplane $H_j$ intersecting with $H_k$ and bounding both $L_{\alpha_1}$ and $L_{\alpha_2}$ as depicted in Figure \ref{fig:4chamberLag}, we do not perform a modification along the $H_j$-coordinate in the McLean type neighborhood $U_{k,j}$ of the corner $H_{k,j}$. By Proposition \ref{prop:Mclean}(v), up to rotations in the disk fibers of the McLean neighborhood $U_k$, the Lagrangians $L_{\alpha_1}$ and $L_{\alpha_2}$ are given by $\mathbb{R}_+\cap\mathbb{D}$ and $\mathbb{R}_-\cap\mathbb{D}$, respectively, and are identified as the fixed locus of the anti-symplectic involution $\tau_k$ in the hyperplane $H_k$. Thus by choosing the fiberwise function $\rho:U_k\rightarrow\mathbb{R}$ so that $X_{\rho(r)}$ rotates $L_{\alpha_1}$ to $L_{\alpha_2}$ in the subdisk fibers of some shrink of $U_k$, we can achieve that $\tilde{L}_{\alpha_1}$is cylindrical, $\tilde{L}_{\alpha_1}\cap L_{\alpha_2}$ is connected and the intersection is clean. The clean surgery of $\tilde{L}_{\alpha_1}$ with $L_{\alpha_2}$ along their intersection, as introduced in the previous section, gives us a new Lagrangian $L_{\alpha_1} \# L_{\alpha_2}$, which is disjoint from a small tubular neighborhood of the hyperplane $H_k$, and the locus near the other hyperplanes are unchanged. We will consider the image of the Lagrangian $L_{\alpha_1} \# L_{\alpha_2}$ under the (completed) convex deformation from $(M(\mathbb{V}), \omega_{\log})$ back to $(M(\mathbb{V}), \omega)$, and (by the same rule of abusing notations as above) denote the corresponding Lagrangian in $(M(\mathbb{V}), \omega)$ is still denoted by $L_{\alpha_1} \# L_{\alpha_2}$. 

Moreover, by Proposition \ref{prop:Mclean}, the Lagrangian surgery we considered above is consistent in the following sense: consider four nearby feasible sign sequences $\{\alpha_1,\alpha_2,\alpha_3,\alpha_4\}$, where $\alpha_1$ (resp. $\alpha_4$) differs from $\alpha_2$ (resp. $\alpha_3$) at $k$, and differs from $\alpha_3$ (resp. $\alpha_2$) at $j$. See Figure \ref{fig:4chamberLag}. Over the intersection $H_{jk}$, we have a polydisk bundle $U_{jk}$ with fiber $(\mathbb{D}^2, r_jdr_j\wedge d\theta_j+r_kdr_k\wedge d\theta_k)$. With such a choice of fiberwise coordinates, the clean surgery $(L_{\alpha_1} \# L_{\alpha_2}) \# (L_{\alpha_3} \# L_{\alpha_4})$ is Hamiltonian isotopic to $(L_{\alpha_1} \# L_{\alpha_3}) \# (L_{\alpha_2} \# L_{\alpha_4})$.

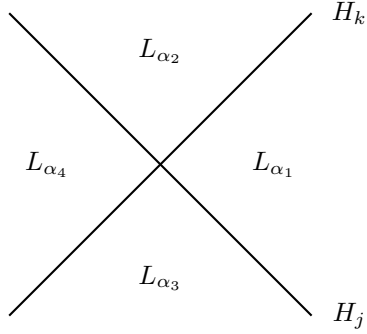
\begin{figure}[ht]
	\centering
	\begin{tikzpicture}
	
	\draw[thick] (-2,2) -- (2,-2);
	\draw[thick] (-2,-2) -- (2,2);
	\fill (2.5,2) node {$H_k$};
	\fill (2.5,-2) node {$H_j$};
	\fill (1.5,0) node {$L_{\alpha_1}$};
	\fill (0,1.5) node {$L_{\alpha_2}$};
	\fill (0,-1.5) node {$L_{\alpha_3}$};
	\fill (-1.5,0) node {$L_{\alpha_4}$};

	\end{tikzpicture}
	\caption{Lagragians $L_{\alpha_1}, L_{\alpha_2}, L_{\alpha_3}, L_{\alpha_4}$ \label{fig:4chamberLag}}
\end{figure}

\begin{proposition}
There is an exact triangle
	\[
		L_{\alpha_1}\to L_{\alpha_1}\# L_{\alpha_2}\to L_{\alpha_2}\xrightarrow{[1]}
	\]
	in $D^\mathit{perf}\mathcal{W}(M(\mathbb{V}),\xi)$, where $L_{\alpha_1}$ and $L_{\alpha_2}$ are adjacent chamber Lagrangians.
\end{proposition}

\begin{remark}\label{rem:surgery in the positive imaginary part}
By possibly replacing $\rho(r)$ with $-\rho(r)$ we can achieve that $L_{\alpha_1} \# L_{\alpha_2}=\tilde{L}_{\alpha_1} \# L_{\alpha_2}$ lie in the region where $\mathrm{Im}(\xi)\geq0$.   
\end{remark}

\subsection{Lagrangians in a Liouville Sector.}\label{sec:construction of Lagrangian}

In Section \ref{sec:sectorial decomposition}, we have constructed the Liouville sectors $\{P_j\}_{1 \leq j\leq N_\mathbb{V}}$, each of which contains a unique 0-dimensional strata in the hyperplane arrangement $\mathbb{V}$. From now on, we simply write $P$ for such a sector $P_j$, and denote the $0$-dimensional intersection point by $\beta$. We write $\xi^{-1}(r) \cup \xi^{-1}(r')$ for the stops coming from the polarization $\xi$, where $-\infty<r'<r$.

We regard the real locus $P(\mathbb{R}):=P \cap \mathbb{R}^d$ as the union of the (open) chambers $\Delta_\bullet^\circ$ bounded by the hyperplanes and the stops $\xi^{-1}(r) \cup \xi^{-1}(r')$. For later purposes, we denote by $\Delta_{\beta^+}$ (resp. $\Delta_{\beta^-}$) the chamber with $\beta$ as one of its vertices and the real projection of $\xi^{-1}(r)$ (resp. $\xi^{-1}(r')$) as its facet. We also denote the facet $\Delta_{\beta^+} \cap \xi^{-1}(r)$ of $\Delta_{\beta^+}$ (resp. the facet $\Delta_{\beta^-} \cap \xi^{-1}(r')$ of $\Delta_{\beta^-}$) by $\Delta^r_\beta$ (resp. $\Delta^{r'}_{\beta}$). See Figure \ref{Lagrangians-Sector}.

\begin{figure}[ht]
	\centering
	\begin{tikzpicture}
	\draw[red, very thick] (0,0) -- (10,0);
	\draw[red, very thick] (0,3) -- (10,3);
	\draw (1,-0.5) -- (5,3.5);
	\draw (1,3.5) -- (5,-0.5);
	\draw (5,-0.5) -- (7,4);
	\draw (6,4) -- (10,-1);
	\fill (11,0) node {\color{red} $\xi (z)=r'$};
	\fill (11,3) node {\color{red} $\xi (z)=r$};
	\fill (2.6, 1.5) node {\color{black} $\beta$};
	\draw[black, very thick] (3,1.5)[point];
	\fill (3,2.3) node {\color{black} $\Delta_{\beta^+}$};
	\fill (3,0.6) node {\color{black} $\Delta_{\beta^-}$};
	\fill (4.5, 1.5) node {\color{black} $\Delta_{\gamma}$};
	\draw[black, very thick] (1.5,3)[point] -- (4.5,3)[point];
	\fill (3,3.3) node {\color{black} $\Delta_\beta^r$};
	\end{tikzpicture}
	\caption{Description of the real locus $P(\mathbb{R})$ of the Liouville sector $P$ \label{Lagrangians-Sector}}
\end{figure}
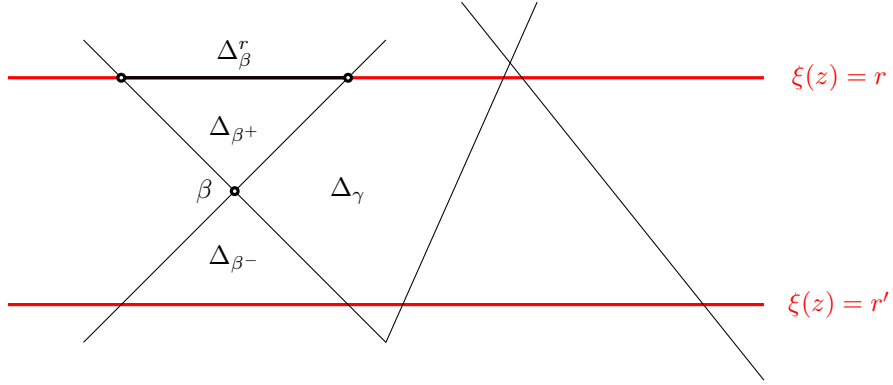

We describe the properly embedded Lagragians in $P$ (disjoint from the sectorial boundary) lying over the interior of each chamber $\Delta_\bullet$. First, consider the case $\Delta_\bullet=\Delta_{\beta^+}$. We further cut $P$ into a union of cornered Liouville sectors described in Section \ref{sec:sectorial decomposition} and use the same notation. By Proposition \ref{sector:GenerationResult}, the cornered Liouville sector $P=P_\times$ is symplectomorphic to $\left((T^*T^d),\bigcup_{i=1}^d(f_i)^{-1}(c_i^\pm)\right)$, where $f_i$ is the pullback of the functions $\mathrm{Re}(\xi_i)$ and $\xi_1=\xi$. Since this symplectomorphism preserves the real locus, the chamber $\Delta_{\beta^+}^\circ$ gets mapped to a cotangent fiber (or equivalently, standard Lagrangian) in $T^*T^d$ under this symplectomorphism, denoted by $T^*_{\beta^+}T^d$. However, the contangent fiber $T^*_{\beta^+}T^d$ intersects the sectorial hypersurface $(f_1)^{-1}(c^+_1)$ corresponding to $\xi^{-1}(r)$, hence one needs to modify it near the stop $F_{c^+_1} \subset (f_1)^{-1}(c^+_1)$ to get an object in the wrapped Fukaya category $\mathcal{W} (P)$. Note that $T^*_{\beta^+}T^d$ is disjoint from other added stops by our assumptions on $P_\times$. Choose a product decomposition of a small neighborhood of the sectorial hypersurface
\begin{equation} \nonumber
(R_1,I_1):\Nbd\left((f_1)^{-1}(c_i^\pm)\right)\cong F_{c_1^\pm} \times \C_{-\epsilon< \mathrm{Re}(z)\leq 0}, 
\end{equation} 
and take the Hamiltonian vector field $X_{\rho (R_1)}$, where $\rho (x)$ is a smooth function that is $0$ near $x=-\varepsilon$ and equals $x^{\alpha}$, $-1<\alpha<0$ near $x=0$, see Figure \ref{fig:function-rho}. We can slightly perturb $T^*_{\beta^+}T^d$ by pushing it off along $X_{\rho (R_1)}$. Note that the function $R_1$ is the pull-back of $\rRe(\xi_1)$ under the symplectomorphism, and we know that 
\begin{equation*}
X_{\rRe(\xi_1)} =\frac{1}{4}\vec{b}^T(D+A^TA)^{-1}\frac{\partial}{\partial\vec{y}},
\end{equation*}
which is bounded as $\vert z_j\vert ,\vert\ell_k\vert\to 0$. Hence the vector field $X_{\rho (R_1)}$ is also complete. Define $L^\pm_{\Delta_{\beta^+}}\subset P_\times$ to be the corresponding Lagrangian submanifold under the perturbation of $X_{\rho (R_1)}$, where the sign $\pm$ indicates that the cotangent fiber $T^*_{\beta^+}T^d$ is pushed off to a positive (resp. negative) direction of $X_{\rho (R_1)}$. When it is clear from the context, we will simply denote these Lagrangian submanifolds as $L^\pm_{\beta^+}$. Note that the Lagrangian $L^+_{\beta^+}$ (resp. $L^-_{\beta^+}$) lies in the area where $\mathrm{Im}(\xi) \geq 0$ (resp. $\mathrm{Im}(\xi) \leq 0$).

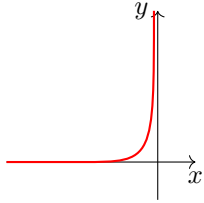
\begin{figure}[ht]
	\centering
	\begin{tikzpicture}
		\draw[->] (-2,0) -- (0.5,0) node [below] {$x$};
		\draw [->] (0,-0.5) -- (0,2) node[left] {$y$};
		\draw [red, thick] (-2,0) -- (-1,0) .. controls (-0.05,0) .. (-0.05,2);
	\end{tikzpicture}
	\caption{The Function $\rho$\label{fig:function-rho}}
\end{figure}

More generally, for any other chamber $\Delta_\gamma\subset P(\mathbb{R})$ which is different from $\Delta_{\beta^{\pm}}$, it intersects non-trivially with other sectorial hypersurfaces $(f_i)^{-1}(c^\pm_i)$. In this case, we choose a compatible product decomposition of the sectorial hypersurfaces
\begin{equation} \nonumber
(R_i,I_i):\Nbd\left((f_i)^{-1}(c_i^\pm)\right) \cong F_{c_i^\pm} \times \C_{-\epsilon<\mathrm{Re}(z)<\epsilon}
\end{equation}
and take the Hamiltonian vector field of $R=\sum R_i$. Then a small perturbation of the cotangent fiber $T^*_{\gamma}T^d$ along $X_{\rho(R)}$ defines a Lagrangian $L^\pm_{\Delta_\gamma}\subset P_\times$ as before. Similar constructions work for the other Liouville subsectors $P_\vert$ by Proposition \ref{sector:GenerationResult} (ii), therefore we can glue these Lagrangians via surgery to obtain a globally defined Lagrangian $L^{\pm}_{\Delta_\gamma}\subset\left(M(\mathbb{V}),\xi\right)$. We will simplify the notation and denote it by $L^{\pm}_{\gamma}$. Again, the convention here is that $L_\gamma^+$ lies in the area where $\mathrm{Im}(\xi)\geq 0$. 

The following lemma will be used frequently in the next subsection.

\begin{lemma}\label{lemma:zero}
Let $L$ be an exact cylindrical Lagrangian submanifold in $(M(\mathbb{V}),\xi)$, such that its closure $\bar{L}$ in $\mathbb{C}^d$ does not contain any $0$-dimensional intersections of the hyperplanes in $\mathbb{V}$, and it lies entirely in the half-space $\{\rIm(\xi)\geq 0\}$, then $L$ is the zero object in the derived Fukaya category $D^\mathit{perf}\mathcal{W}(M(\mathbb{V}),\xi)$.
\end{lemma}

\begin{proof}
Note that the vector field $X_{\rRe(\xi)}$ is purely imaginary and only converges to $0$ when $z$ goes to the $0$-dimensional strata in the hyperplane arrangement.  Flowing along $X_{\rRe(\xi)}$, $L$ is sent to a subspace $\left\{\rIm (z)\geq c\right\}$ for some constant $c>0$. It follows that we can use Hamiltonian isotopy to move $L$ away from any compact subset of $(M(\mathbb{V}),\xi)$, which implies that $L$ is the zero object by \cite{ganatra2018sectorial}.
\end{proof}

\subsection{Chamber Moves.}\label{s:chamber_moves}
In this subsection, we prove several results about chamber moves that will be used in the proof of Theorem \ref{theorem:generation1}.

\subsubsection{Wrapping Exact Triangles.}\label{sec:wrappingtriangle}
Let $L_\beta^{\pm}=L^\pm_{\beta^+}$ be the Lagrangians constructed in Section \ref{sec:construction of Lagrangian}. We want to describe an exact triangle of Lagrangians 
\begin{equation}\label{eq:wet}
L_\beta^+\to L_\beta^-\to D_\beta \xrightarrow{[1]}
\end{equation}
in the wrapped Fukaya category $\mathcal{W}^\mathit{perf}(P)$, where $D_\beta$ is the linking disk corresponding to the facet $\Delta_\beta^r$, see Figure \ref{Lagrangians-Sector}. A similar exact triangle was proved in \cite{ganatra2018sectorial}, which we will refer to as the ``wrapping exact triangle''. In our situation, we do not know a priori whether $L_{\beta}^+$ is obtained from $L_{\beta}^-$ via wrapping through the stop, so we cannot directly apply the result in \cite{ganatra2018sectorial}. Instead, we will prove (\ref{eq:wet}) using clean surgery described in Section \ref{sec:clean-surgery}. Consider two Lagrangians $L_\beta^+$ and $L_\beta^-$ intersecting along the interior of a chamber in the real loci, except near the sectorial boundary, where they are shifted upwards and downwards in the direction of $\mathrm{Im}(\xi)$, respectively, and denote the clean surgery of them by $L_\beta^+ \# L_\beta^-$. Near the sectorial boundary that is identified with $\xi^{-1}(r)\times T^\ast[-\epsilon, 0]$, the Lagrangian  $L_\beta^+ \# L_\beta^-$ cleanly intersects with $L_{\Delta_\beta} \subset P(\mathbb{R})$ along $L_{\Delta^r_\beta} \times \{-\epsilon/2\}$, see Figure \ref{Lagrangians-Sector}. This follows from our choice of the product structure near the sectorial boundary, where $R=\mathrm{Re}(\xi)$. By definition, one can see that $L_{\beta}^+\# L_{\beta}^-$ can be identified with the linking disk associated to the stop $\xi^{-1} (r)$.

\subsubsection{Flipping Lagrangians.}\label{sec:flippingLagrangians}

The Lagrangian submanifold $L^+_{\beta^+}$ is Hamiltonian isotopic to $L^-_{\beta^-}$, since by construction, the corresponding Lagrangians in $T^*T^d$ are Hamiltonian isotopic to each other. Similarly, $L^-_{\beta^+}$ is Hamiltonian isotopic to $L^+_{\beta^-}$.

\subsubsection{Identifying Linking Disks.}\label{sec:Identifying linkingdisc}

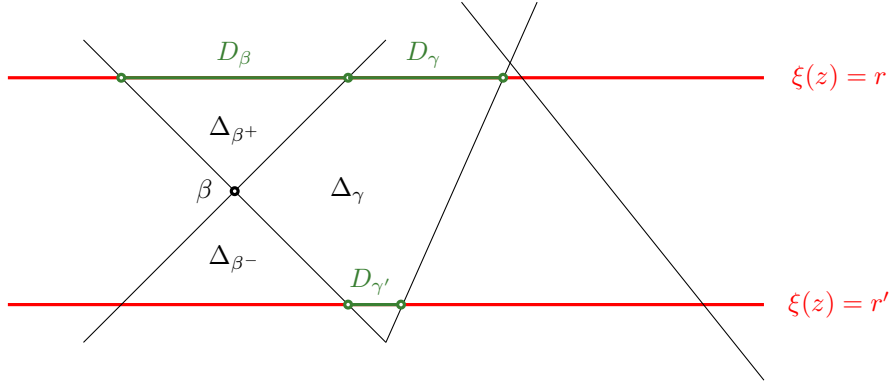
\begin{figure}[ht]
	\centering
	\begin{tikzpicture}
	\draw[red, very thick] (0,0) -- (10,0);
	\draw[red, very thick] (0,3) -- (10,3);
	\draw (1,-0.5) -- (5,3.5);
	\draw (1,3.5) -- (5,-0.5);
	\draw (5,-0.5) -- (7,4);
	\draw (6,4) -- (10,-1);
	\fill (11,0) node {\color{red} $\xi (z)=r'$};
	\fill (11,3) node {\color{red} $\xi (z)=r$};
	\fill (3,2.3) node {\color{black} $\Delta_{\beta^+}$};
	\fill (3,0.6) node {\color{black} $\Delta_{\beta^-}$};
	\fill (2.6, 1.5) node {\color{black} $\beta$};
	\fill (4.5, 1.5) node {\color{black} $\Delta_{\gamma}$};
	\draw[OliveGreen, very thick] (1.5,3)[point] -- (4.5,3)[point];
	\fill (3,3.3) node {\color{OliveGreen} $D_{\beta}$};
	\draw[black, very thick] (3,1.5)[point];
	\draw[OliveGreen, very thick] (4.5,3) -- (6.55,3)[point];
	\fill (5.5,3.3) node {\color{OliveGreen} $D_\gamma$};
	\draw[OliveGreen, very thick] (4.5,0)[point] -- (5.20,0)[point];
	\fill (4.8,0.3) node {\color{OliveGreen} $D_{\gamma'}$};
	\end{tikzpicture}
	\caption{Linking Disks in $\mathcal{W}(P)$ \label{fig:indentifying Linking disks}}
\end{figure}

Let $\Delta_\gamma\subset P(\mathbb{R})$ be a chamber adjacent to both $\Delta_{\beta^+}$ and $\Delta_{\beta^-}$. To this chamber, we have associated two Lagrangians $L_\gamma^{\pm}$ (cf. Section \ref{sec:construction of Lagrangian}), and two linking disks $D_\gamma$ and $D_{\gamma'}$ (cf. Section \ref{sec:wrappingtriangle}). An illustration is given in Figure \ref{fig:indentifying Linking disks}. We claim that $D_{\gamma'}$ is generated by $D_\gamma$ and $D_\beta$ in $\mathcal{W}(P)$. 

We have proved in Section \ref{sec:wrappingtriangle} that the Lagrangian $L^+_\gamma \# L^-_\gamma$ is Hamiltonain isotopic to $D_\gamma \oplus {D}'_\gamma$. Moreover, we have $L_{\beta^+}^+ \# L_{\beta^+}^- \sim D_\beta$. By our discussions in Section \ref{sec:gluingchambers}, we can glue the Lagrangians $L^+_{\beta^+}$ and $L^+_\gamma$ via surgery. It follows from Lemma \ref{lemma:zero} that the Lagrangian $L^+_{\beta^+}\#L^+_\gamma$, which is disjoint from the crossing $\beta$, can be pushed off to infinity, hence becoming the zero object in the derived Fukaya category $D^\mathit{perf}\mathcal{W}(P)$. A parallel argument holds for $L^-_{\beta^+}$ and $L^-_\gamma$, which implies that $D'_\gamma$ is generated by $D_\gamma$ and $D_\beta$. 

On the other hand, if a chamber $\Delta_\gamma$ is disjoint from the crossing $\beta$, the linking disks $D_\gamma$ and $D'_\gamma$ can be identified by the decomposition argument (see Proposition \ref{sector:GenerationResult}).

\subsubsection{Decomposing Chamber Lagrangians.}\label{sec:Decomposition of Lag}
Let $I_\beta \subset \{1, \dots, n\}$ be the index set of the hyperplanes in $\mathbb{V}$ passing through the crossing $\beta$. Consider the corresponding polarized hyperplane arrangement $\mathbb{V}_\beta$, which consists of $d$ hyperplanes $\{H_i\}_{i \in I_\beta}$ in $\R^d$ by the simplicity of $\mathbb{V}$. In other words, $\mathbb{V}_\beta$ is the iterated deletion of $\mathbb{V}$ by the hyperplanes $\{H_j\}_{j \notin I_\beta}$ (see Section \ref{sec:Delandres HypArr}). There exists a unique sign sequence $\tilde{\beta}:I_\beta \to \{+,-\}$, which singles out the bounded and feasible chamber of $\mathbb{V}_\beta$ with respect to the polarization $\xi$. In the Liouville sector $P$, this chamber is given by $\Delta_{\beta^-}$. Let $L_{\beta^-}\subset M(\mathbb{V})$ be the pushforward of $L^-_{\beta^-}\subset P$ under the sectorial gluing.  Consider the restriction map on the set of sign sequences $\mathrm{Res}_{I_\beta}:\{\pm\}^{|I|} \to \{\pm\}^{|I_\beta|}$, there is a collection of chamber Lagrangians $\{L_{\alpha}\}$ with $\mathrm{Res}_{I_\beta}(\alpha)=\tilde{\beta}$. We can glue these Lagrangians together via clean surgery to obtain an object in the Fukaya category $\mathcal{W}(M(\mathbb{V}), \xi)$, which is has the form of an iterated mapping cone built out of the $L_\alpha$'s. We denote the object by $\bigodot_{\alpha: \mathrm{Res}_{I_\beta}(\alpha)=\tilde{\beta}}L_\alpha$. The polarization $\xi$ induces a partial order on the set of sign sequences $\mathrm{Res}^{-1}_{I_\beta}(\tilde{\beta})$: two sign sequences $\alpha_1,\alpha_2\in\mathrm{Res}^{-1}_{I_\beta}(\tilde{\beta})$ satisfy $\alpha_1<\alpha_2$ if $d_{\alpha_1\beta}<d_{\alpha_2\beta}$, i.e. $\alpha_2$ differs with $\beta$ in more entries than $\alpha_1$. We iteratively apply the surgery procedure described in Section \ref{sec:gluingchambers} such that all clean surgeries are carried from the smaller chamber to the larger one. Note that the result of the surgery belongs to the subset $\left\{\mathrm{Im}(\xi)\geq 0\right\}$, see Remark \ref{rem:surgery in the positive imaginary part}.

\begin{proposition}\label{surgery:decomposing-standard-Lag}
For each $0$-dimensional crossing $\beta$, the Lagrangian 
\begin{equation} \nonumber
L:=L_{\beta^-} \# \bigodot_{\alpha: \mathrm{Res}_{I_\beta}(\alpha)=\tilde{\beta}}L_\alpha
\end{equation}
obtained by surgery is a trivial object in $D^\mathit{perf}\mathcal{W}(M(\mathbb{V}), \xi)$.
\end{proposition}

\begin{proof}
We first perform a surgery between $L_{\beta^-}$ and $L_\beta$. Since they intersect cleanly, we can perform a clean surgery as we did in Section \ref{sec:gluingchambers}. Then we glue the result $L_{\beta^-} \# L_\beta$ with other chamber Lagrangians with labels in $\mathrm{Res}^{-1}_{I_\beta}(\tilde{\beta})$. As all the glued objects lie in the region $\left\{\mathrm{Im}(\xi) \geq 0\right\}$, the resulting Lagrangian $L$ also lies in $\left\{\mathrm{Im}(\xi) \geq 0\right\}$. As a surgery, $L$ is also disjoint from the $0$-dimensional intersections of the complex hyperplanes and the stop, so it becomes the zero object by Lemma \ref{lemma:zero}.
\end{proof}

\subsection{Proof of Theorem \ref{theorem:generation1}.} We are ready to prove one of our main theorems (Theorem \ref{theorem:generation1}), restated below for convenience. 

\begin{theorem}\label{thm:generation}
The partially wrapped Fukaya category $\mathcal{W}(M(\mathcal{\mathbb{V}}), \xi)$ is generated by the collection of chamber Lagrangians $\{L_\alpha\}_{\alpha \in \mathscr{P}(\mathbb{V})}$. 
\end{theorem}

Let $P_j\subset(M(\mathbb{V}),\xi)$ be the Liouville subsector containing the 0-dimensional intersection $\beta_j$ defined in Section \ref{sec:sectorial decomposition}, where $1 \leq j \leq N_\mathbb{V}$. Due to the genericity of the polarization $\xi$, the stop $\xi^{-1}(r_{j})$ consists of preimages of the open chambers in $\mathbb{R}^{d-1}$ given by the restriction of the chambers of $\mathbb{V}$. To the open chambers in $\xi^{-1}(r_{j})$, we associate the linking disks, which are denoted by $D_\bullet^{r_{j}}$, where the subscript indicates to which open chamber in the stop the linking disk is associated. For example, we write $D^{r_j}_{\beta_j}$ for the linking disk associated to the chamber corresponding to the $0$-dimensional stratum $\beta_j$.  See Figure \ref{Description of objects in the sector} for a description.

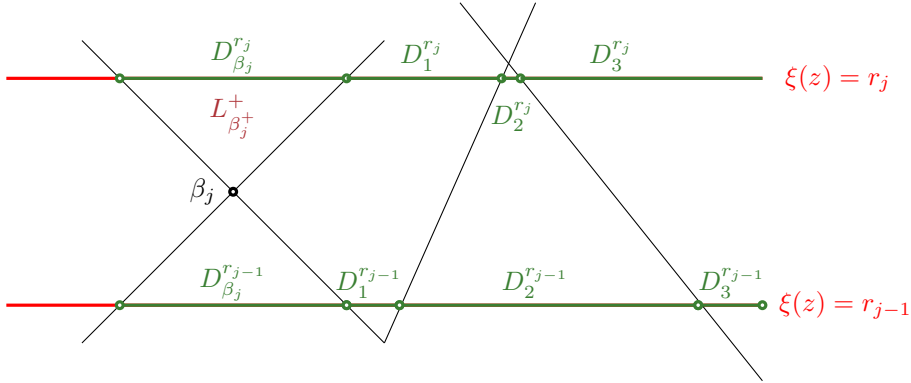
\begin{figure}[ht]
	\centering
	\begin{tikzpicture}
	\draw[red, very thick] (0,0) -- (10,0);
	\draw[red, very thick] (0,3) -- (10,3);
	\draw (1,-0.5) -- (5,3.5);
	\draw (1,3.5) -- (5,-0.5);
	\draw (5,-0.5) -- (7,4);
	\draw (6,4) -- (10,-1);
	\fill (11.1,0) node {\color{red} $\xi (z)=r_{j-1}$};
	\fill (11,3) node {\color{red} $\xi (z)=r_{j}$};
	\fill (3,2.5) node {\color{Maroon} $L^+_{\beta_j^+}$};
	\fill (2.6, 1.5) node {\color{black} $\beta_{j}$};
	\draw[black, very thick] (3,1.5)[point];
	\draw[OliveGreen, very thick] (1.5,3)[point] -- (4.5,3)[point];
	\fill (3,3.3) node {\color{OliveGreen} $D^{r_{j}}_{\beta_{j}}$};
	\draw[OliveGreen, very thick] (4.5,3) -- (6.55,3)[point];
	\fill (5.5,3.3) node {\color{OliveGreen} $D^{r_{j}}_1$};
	\draw[OliveGreen, very thick] (6.55,3) -- (6.8,3)[point];
	\fill (6.7, 2.5) node {\color{OliveGreen} $D^{r_{j}}_2$};
	\draw[OliveGreen, very thick] (6.8,3) -- (10,3);
	\fill (8,3.3) node {\color{OliveGreen} $D^{r_{j}}_3$};
	\draw[OliveGreen, very thick] (1.5,0)[point] -- (4.5,0)[point];
	\fill (3,0.3) node {\color{OliveGreen} $D^{r_{j-1}}_{\beta_{j}}$};
	\draw[OliveGreen, very thick] (4.5,0)[point] -- (5.20,0)[point];
	\fill (4.8,0.3) node {\color{OliveGreen} $D^{r_{j-1}}_1$};
	\draw[OliveGreen, very thick] (5.20,0)[point] -- (9.15,0)[point];
	\fill (7,0.3) node {\color{OliveGreen} $D^{r_{j-1}}_2$};
	\draw[OliveGreen, very thick] (9.15,0)[point] -- (10,0)[point];
	\fill (9.6,0.3) node {\color{OliveGreen} $D^{r_{j-1}}_3$};
	\end{tikzpicture}
	\caption{Description of objects in $P_j$\label{Description of objects in the sector}}
\end{figure}

Recall that we have proved in Section \ref{section:piece} that the wrapped Fukaya category $\mathcal{W}(P_j)$ is generated by all the linking disks $D^{r_j}_\bullet$ and the standard Lagrangian $L^+_{\beta^+_{j}}$. When $j=1$, we can further simplify the set of generating Lagrangians. 

\begin{lemma}\label{lem:thefirstInductionStep}
The wrapped Fukaya category $\mathcal{W}(P_1)$ is generated by a single chamber Lagrangian $L^+_{\beta_1}$.
\end{lemma}
\begin{proof}
We add an auxiliary stop $\xi^{-1}(r')$ for some $-\infty< r'<r_1$, and denote by $P_0'$ the resulting Liouville sector. This gives rise to additional linking disks $D^{r'}_{\beta_1}$ and $D^{r'}_{\bullet}$, which become trivial objects under the stop removal functor $\mathcal{W}(P'_0) \to \mathcal{W}(P_0)$. It is therefore enough to show that the old linking disks $D^{r_1}_{\beta_1}$ and $D^{r_1}_{\bullet}$ are generated by the Lagrangian $L^+_{\beta_1^+}$ and the new linking disks $D^{r'}_{\beta_1}$ and $D^{r'}_{\bullet}$. First, it follows from our discussions in Section \ref{sec:wrappingtriangle} that we have $D^{r'}_{\beta_1^+} \cong\mathit{Cone}(L^+_{\beta_1^+} \to L^-_{\beta_1^+})$. Applying the flipping argument in Section \ref{sec:flippingLagrangians}, this is isomorphic to $\mathit{Cone}(L^-_{\beta_1^-} \to L^+_{\beta_1^-}) \cong D^{r'}_{\beta_1}$. It follows that $D^{r_1}_{\beta_1} \simeq D^{r'}_{\beta_1}$ in $\mathcal{W}(P_1)^\mathit{perf}$. Next, take a linking disk $D^{r_1}_1$ adjacent to $D^{r_1}_{\beta_1}$. By Section \ref{sec:Identifying linkingdisc}, we know that $D^{r_1}_{\beta_1} \oplus D^{r_1}_{1} \simeq D^{r'}_1$. For other linking disks $D^{r_1}_{\bullet}$, we can directly identify them with the corresponding ones in $\xi^{-1}(r')$ by Proposition \ref{sector:GenerationResult} (ii). 
\end{proof}

\begin{proof}[Proof of Theorem \ref{thm:generation}]
We prove the theorem by induction. For $i \geq 1$, we write $P_{\leq i}$ for the union of all subsectors $P_j$ with $1 \leq j \leq i$. We claim that  the wrapped Fukaya category $\mathcal{W} (P_{\leq i})$ is generated by the chamber Lagrangians $L_\alpha\cap P_{\leq i}$, where $\alpha\in\mathscr{P}(\mathbb{V})$. 
	
When $i=1$, the claim follows from Lemma \ref{lem:thefirstInductionStep}. In other words, by the flipping argument in Sections \ref{sec:flippingLagrangians} and \ref{sec:Decomposition of Lag} (cf. Proposition \ref{surgery:decomposing-standard-Lag}), we have quasi-isomorphisms $L^+_{\beta_1^+} \cong L^-_{\beta_1^-} \cong L_{\beta_1}$. Suppose that the statement holds for $i=k$. Since the wrapped Fukaya category $\mathcal{W}(P_{\leq k+1})$ can be realized as the sectorial gluing of $\mathcal{W}(P_{\leq k})$ and $\mathcal{W}(P_{k+1})$, it is enough to show that $\mathcal{W}(P_{k+1})$ is generated by the chamber Lagrangians $\left\{L_\alpha\cap P_{\leq i}\right\}_{\alpha\in\mathscr{P}(\mathbb{V})}$. First, take the  generator  $L^+_{\beta_k^+}$, arguing in the same way as in the proof of Lemma \ref{lem:thefirstInductionStep}, we have
\begin{equation*}
L^+_{\beta_{k+1}^+} \simeq L^-_{\beta_{k+1}^-} \simeq \bigodot_{\alpha: \mathrm{Res}_{I_{\beta_{k+1}}}(\alpha)=\widetilde{\beta}_{k+1}} L_\alpha [1],
\end{equation*}
which verifies the generation of $L^+_{\beta_{k+1}^+}$ by the Lagrangians in $\left\{L_\alpha\cap P_{\leq i}\right\}_{\alpha\in\mathscr{P}(\mathbb{V})}$. Exactly the same argument can be applied to $L^-_{\beta_{k+1}^+}$ and $D^{r_{k+1}}_{\beta_{k+1}}$, which gives the same conclusion. Moreover, in Lemma \ref{lem:thefirstInductionStep} we have proved that each linking disk $D^{r_{k+1}}_\bullet$ can be identified with the corresponding linking disk $D^{r_{k}}_\bullet$ up to quasi-isomorphism in $\mathcal{W}(P_{\leq k+1})$. (For an adjacent one, we take the mapping cone with $D^{r_{k+1}}_\beta$, as in Section \ref{sec:Identifying linkingdisc}.) Then the induction hypothesis implies that the linking disks $D^{r_{k}}_\bullet$ are generated by chamber Lagrangians. 
\end{proof}

\section{Computing Floer Cohomologies}\label{sec:computation}

In this section, we compute Fukaya $A_\infty$-algebra of the Lagrangians $\{L_\alpha\}_{\alpha \in\mathscr{P}(\mathbb{V})}$ in $\mathcal{W}\left(M(\mathbb{V}),\xi\right)$ and prove Theorem \ref{theorem:conv1}. 

\subsection{The Convolution Algebra $\widetilde{B}(\mathbb{V})$.}

We first recall the definition of the convolution algebra  $\widetilde{B}(\mathbb{V})$ associated to a polarized hyperplane arrangement $\mathbb{V}=(V,\eta,\xi)$ introduced in \cite{BLPPW}. For $S=\{i_1,\cdots,i_m\}\subset\{1,\cdots,n\}$, let $u_S$ be the monomial $u_{i_1}\cdots u_{i_m}\in\mathbb{Z}[u_1,\cdots,u_n]$, and let $H_{\mathbb{R},S}\subset V+\eta$ be the intersection $\bigcap_{i\in S}H_{\mathbb{R},i}$ of hyperplanes indexed by the elements of $S$. For $\alpha,\beta\in\mathscr{P}(\mathbb{V})$, introduce the ring
\begin{equation*}
\widetilde{R}_{\alpha\beta}:=\frac{\mathbb{Z}[u_1,\cdots,u_n]}{(u_S,\Delta_\alpha\cap\Delta_\beta\cap H_{\mathbb{R},S}=\emptyset)}.
\end{equation*}
Let $f_{\alpha\beta}\in\widetilde{R}_{\alpha\beta}$ be the element corresponding to $1\in\mathbb{Z}[u_1,\cdots,u_n]$. For $\alpha,\beta,\gamma\in\mathscr{P}(\mathbb{V})$, we introduce the notation
\begin{equation*}
S(\alpha\beta\gamma):=\left\{i\in\{1,\cdots,n\}|\alpha(i)=\gamma(i)\neq\beta(i)\right\}.
\end{equation*}
Define
\begin{equation*}
\widetilde{B}(\mathbb{V}):=\bigoplus_{\alpha,\beta\in\mathscr{P}(\mathbb{V})}\widetilde{R}_{\alpha\beta},
\end{equation*}
with multiplication
\begin{equation}\label{eq:multiplication}
f_{\alpha\beta}\cdot f_{\beta\gamma}=u_{S(\alpha\beta\gamma)}f_{\alpha\gamma},
\end{equation}
which is extended bilinearly over $\mathbb{Z}[u_1,\cdots,u_n]$. The algebra $\widetilde{B}(\mathbb{V})$ admits a $\mathbb{Z}$-grading given by
\begin{equation*}
|f_{\alpha\beta}|=d_{\alpha\beta},\textrm{ }|u_i|=2,
\end{equation*}
where the number $d_{\alpha\beta}\in\mathbb{N}$ is defined by (\ref{eq:d}). This grading can be refined to a multi-grading by $\mathbb{Z}\langle e_1,\cdots,e_n\rangle$, in which case 
\begin{equation*}
|f_{\alpha\beta}|=e_{i_1}+\cdots+e_{i_k},\textrm{ }|u_i|=2e_i,
\end{equation*}
if $\beta$ is obtained from $\alpha$ by changing the signs in positions $i_1,\cdots,i_k$. Thus $\widetilde{B}(\mathbb{V})$ can be viewed as an algebra over $\mathbb{Z}[u_1,\cdots,u_n]$. Note that the single grading on $\widetilde{B}(\mathbb{V})$ is recovered by setting all $e_i$ to be 1. 

\begin{remark}
In \cite[Section4]{BLPW2010}, the geometric description of this convolution algebra comes from the relative core $\mathcal{X}$ in the associated hypertoric variety $\mathfrak{M}_\mathbb{V}$ associated to the polarized hyperplane arrangement $\mathbb{V}$. Note that $\mathcal{X}$ is a union of toric varieties $\{X_\alpha\}_{\alpha \in \mathscr{P}(\mathbb{V})}$ associated to each (closed) chamber $\Delta_\alpha$. Take a normalization of $\mathcal{X}$, denoted by $\pi:\widetilde{\mathcal{X}} \to \mathcal{X}$. Then the equivariant cohomology $H^*_T(\widetilde{\mathcal{X}} \times_\pi \widetilde{\mathcal{X}})$ becomes a $\mathbb{Z}$-graded algebra equipped with the usual convolution product with respect to the orientation on $\widetilde{\mathcal{X}} \times_\pi \widetilde{\mathcal{X}}$, twisted by the combinatorial signs for $\mathbb{V}$. It turns out that there is an $\mathbb{Z}$-algebra isomorphism 
\begin{equation*}
\widetilde{B}(\mathbb{V}) \cong H_T^*(\widetilde{\mathcal{X}} \times_\pi \widetilde{\mathcal{X}}).
\end{equation*}
\end{remark}

\subsection{Proof of Theorem \ref{theorem:conv1}.}

We first compute the partially wrapped Floer cohomologies of the generating chamber Lagrangians $\{L_\alpha\}_{\alpha\in\mathscr{P}(\mathbb{V})}$ in the Liouville sector $(M(\mathbb{V}),\xi)$. To do so, it would be more convenient to view $M(\mathbb{V})$ as the complement of $n+1$ hyperplanes in $\mathbb{CP}^d$. By abuse of notations, we denote these hyperplanes by $H_{0},H_{1},\cdots,H_{n}\subset\mathbb{CP}^d$, where $H_{0}$ is the hyperplane at infinity. After taking the iterated blow-ups as in Section \ref{sec:McLean}, we obtain a non-singular projective variety $X_\mathbb{V}$ with a simple normal crossing divisor $D_\mathbb{V}$ such that $M(\mathbb{V})=X_\mathbb{V} \setminus D_\mathbb{V}$.  We may write $D_\mathbb{V}=\bigcup_{i=0}^{n+m}D_i$, where $D_i=\widetilde{H}_{i}$ is a strict transform of $H_i$ for $0\leq i\leq n$, and the divisors $D_{n+1},\cdots,D_{n+m}$ are exceptional divisors. For $I\subset\{0,\cdots,n+m\}$, denote by $D_I$ the intersection of the divisors $D_i$ for $i\in I$, and introduce the notation
\begin{equation*}
D_I^\circ:=D_I\setminus\bigcup_{j\notin I}D_j.
\end{equation*}

As in Section \ref{sec:McLean}, we can equip $M(\mathbb{V})$ with the structure of a Liouville manifold so that for each $i\in\{0,\cdots,n+m\}$ there is a tubular neighborhood $U_i$ of $D_i$, and the $|I|$-fold intersection of these tubular neighborhoods, $U_I:=\bigcap_{i\in I}U_i$, admits the structure of a symplectic disk bundle with structure group $U(1)^{|I|}$ over $D_I$ (cf. Proposition \ref{prop:Mclean}). Let $\pi_I: U_I \to D_I$ be the projection map, and let $S_I\rightarrow D_I$ be the associated $T^{|I|}$-bundle. Denote by $U_I^\circ$ the restriction of $U_I$, and by $S_I^\circ$ the restriction of $S_I$, to the open stratum $D_I^\circ\subset D_I$. As a convention, we set
\begin{equation*}
D_\emptyset=X_\mathbb{V}, D_\emptyset^\circ=S_\emptyset=S_\emptyset^\circ=M(\mathbb{V}). 
\end{equation*} 

In order to compute the wrapped Floer cohomology with respect to the stop defined by the polarization $\xi$, we need to understand the hypersurface $\rRe(\xi)^{-1}(C)$ for $C \gg 1$ in the McLean neighborhood of the divisor $D_\mathbb{V}$. Let $\bar{\xi}:\mathbb{CP}^d \rightarrow \mathbb{C}$ be the homogenization of the polarization $\xi:M(\mathbb{V})\rightarrow\mathbb{C}$, and let $\ell_{\infty}:\mathbb{CP}^d \rightarrow \mathbb{C} $ be the linear function defining the hyperplane $H_0\subset\mathbb{CP}^d$. We have a rational map $\phi:\mathbb{CP}^d \dashrightarrow \mathbb{CP}^1$, given by $\phi(z):=\left[\bar{\xi}(z):\ell_{\infty}(z)\right]$ with the base locus $B=\{\bar{\xi}=0\} \cap H_0$. Composing $\phi$ with the blow-up map $\pi: X_\mathbb{V} \to \mathbb{CP}^d$, we get a rational map $\phi_\pi:=\pi \circ \phi : X_\mathbb{V} \dashrightarrow \mathbb{CP}^1$ with the base locus $\pi^*B$. 
Notice that
$\pi^*B \subset \bigcup_{i \in \{0,n+1,\dots,n+m\}} D_i$,  
 $\phi_{\pi}^{-1}([1:0])=\bigcup_{i \in \{0,n+1,\dots,n+m\}} D_i\setminus \pi^*B$  and $\overline{\phi_{\pi}^{-1}([1:0])}=\bigcup_{i \in \{0,n+1,\dots,n+m\}} D_i$. 
In particular, a punctured McLean-type neighborhood $\Nbd^\circ(D_\mathbb{V}):=\Nbd(D_\mathbb{V})\setminus D_\mathbb{V}$ lies in $X_\mathbb{V}\setminus \pi^*B$. It follows that we have a well-defined morphism $\phi_\pi: \Nbd^\circ(D_\mathbb{V}) \to \mathbb{CP}^1$. 
By possibly replacing $\xi$ with $\xi+c$ for some $c \in \mathbb{R}$, we may assume that $\xi^{-1}(0) \cap L_{\alpha}= \emptyset $, so that
$\overline{L}_{\alpha} \cap \pi^*B= \emptyset$ for all $\alpha \in \mathscr{P}(\mathbb{V})$.

For $C$ sufficiently large and $I \subset \{0,n+1,\dots,n+m\}$, the restriction of the bundle projection $\pi_{I}|_{\xi^{-1}(C) \cap U_I}: \xi^{-1}(C) \cap U_I \to D_I$ is submersive onto its image.
In fact, the fibres of $\pi_{I}|_{\xi^{-1}(C) \cap U_I}$ are diffeomorphic to $T^{|I|-1}$.
The closure $\overline{\xi^{-1}(C)}$ in $X_{\mathbb{V}}$ is given by $\overline{\xi^{-1}(C)}=\xi^{-1}(C) \cup \pi^*B$.
As a result, $\overline{\xi^{-1}(C)}$ can be written as a union of the image of sections $s_i:\pi_{i}\left(\xi^{-1}(C)\cap U_i\right) \to U_i|_{\pi_{i}\left(\xi^{-1}(C) \cap U_i\right)}$ for $i=0,n+1,\dots,n+m$ such that $s_i(z)=0$ if and only if $z \in \pi^*B \cap D_i$.
To compute the wrapped Floer cohomologies of the chamber Lagrangians with respect to the sectorial hypersurface $\rRe(\xi)^{-1}(C)$, it suffices to compute the wrapped Floer cohomologies with respect to the stop $\mathfrak{s}_\mathbb{V}:=\xi^{-1}(C)=\bigcup_{i=0,n+1,\dots,n+m} \im(s_i) \cap M(\mathbb{V})$.


%

For $\epsilon>0$, let $\rho_{\epsilon}:(0,\infty) \to [0,\infty)$ be a smooth function such that
\begin{align}
 \rho_{\epsilon}(r)=0 \text{ for } r> \epsilon \text{ and } \rho_{\epsilon}(r)=\frac{1}{r}. \label{eq:rhoe}
\end{align}
We choose a Hamiltonian $h:M(\mathbb{V})\rightarrow\mathbb{R}$ so that it takes the form $h_I=\sum_{i\in I}\rho_{\epsilon}(h_{i})$ in the tubular neighborhood $U_I^\circ$ for each $I$, where each $h_{i}$ generates the $S^1$-action which rotates the fibers of $S_i^\circ\rightarrow D_i^\circ$. It follows that the Hamiltonian orbits of $X_h$ near $D_I^\circ$ are given by orbits of the circle actions, therefore they appear in families and the collection of them can then be identified with the torus fibers of $S_I^\circ\rightarrow D_I^\circ$. 
To make sure that the Hamiltonian chords between the chamber Lagrangians are nondegenerate, we further perturb the Hamiltonian $h$ with a $C^2$-small Morse function on the chamber $\Delta_\alpha\subset\mathbb{R}^d$, which reaches its maximum at the corners (i.e. $0$-dimensional strata) and whose restriction to each stratum $\Delta_\alpha\cap D_{I}^\circ$ has a single critical point which is a minimum for any  $I\subset\{1,\cdots,n\}$. Since $\Delta_\alpha$ is contractible, such a function is easy to construct. We denote such a (perturbed) Hamiltonian by $\tilde{h}:M(\mathbb{V})\rightarrow\mathbb{R}$. 

To determine the gradings of the generators of the wrapped Floer cochain complexes, we need to choose a trivialization of the canonical bundle $K_{M(\mathbb{V})}$. The trivialization that we shall use here is the one induced by the restriction of the trivialization of the canonical bundle $K_{\mathbb{C}^d}$. Note that this is different from the grading convention of \cite{lekili2020homological}, in the case of higher-dimensional of pair-of-pants, where one uses the trivialization coming from the restriction of that of the canonical bundle of $(\mathbb{C}^\ast)^d$.

\begin{proposition}\label{proposition:aa}
Given $\alpha\in\mathscr{P}(\mathbb{V})$, we have an isomorphism
	\begin{equation}\label{eq:aa}
	\mathit{HW}^\ast(L_\alpha,L_\alpha)\cong\widetilde{R}_{\alpha\alpha}
	\end{equation}
	as $\mathbb{Z}$-graded rings.
\end{proposition}
\begin{proof}
The proof is divided into three steps.

\paragraph{Step 1: graded module structure}
We first show that (\ref{eq:aa}) is an isomorphism of $\mathbb{Z}\langle e_1,\cdots,e_n\rangle$-graded modules. For each $I\subset\{1,\cdots,n\}$, denote the local coordinates on the fibers of the symplectic disk bundle $U_I^\circ\rightarrow D_{I}^\circ$ by $x_{I,i}$, where $i\in I$. Assume that $D_{I}\cap\Delta_\alpha\neq\emptyset$. By our choice of the wrapping Hamiltonian $\tilde{h}$, for each $|I|$-tuple of positive integers $(k_i)_{i\in I}$, there is a single non-degenerate time-1 chord of $X_{\tilde{h}}$ from $L_\alpha$ to itself which wraps $k_i$ times around the hyperplane $D_{i}\subset X_\mathbb{V}$, which is locally defined by the equation $x_{I,i}=0$. We label the corresponding generator of $\mathit{CW}^\ast(L_\alpha,L_\alpha)$ by $\prod_{i\in I}u_i^{k_i}$, reflecting the fact that it is a product of the generators $u_i^{k_i}$, which wrap $k_i$ times along a single hyperplane $D_{i}$, a fact which will be established shortly. By our grading conventions, each generator $u_i$ of $\mathit{CW}^\ast(L_\alpha,L_\alpha)$ has degree 2. On the other hand, if $D_{I}\cap\Delta_\alpha=\emptyset$, then for $i\in I$ there is no time-1 chord of $X_{\tilde{h}}$ which wraps around each of the hyperplane $D_{i}$ by a positive number of times. Note also that by our choice of the stop $\mathfrak{s}_\mathbb{V}\subset\partial_\infty M(\mathbb{V})$, no wrapping can occur along the compactifying divisors $D_{0},D_{n+1},\cdots,D_{n+m}\subset X_\mathbb{V}$. 

Observe that $H_0(L_\alpha;\mathbb{Z})$ injects into $H_0(M(\mathbb{V});\mathbb{Z})$ so we have the isomorphisms
	\begin{equation}\label{eq:isom}
	H_1(M(\mathbb{V}),L_\alpha)\cong H_1(M(\mathbb{V});\mathbb{Z})\cong\mathbb{Z}^n=\mathbb{Z}\langle e_1,\cdots,e_n\rangle,
	\end{equation}
whose generators correspond to meridian loops around the hyperplanes $H_{1},\cdots,H_{n}\subset\mathbb{C}^d$.
Under this isomorphism, the generator $\prod_{i\in I}u_i^{k_i}\in\mathit{CW}^\ast(L_\alpha,L_\alpha)$ represents the (relative) homology class $(k_1,\cdots,k_n)\in\mathbb{Z}^n$, where $k_i=0$ if $i\notin I$.
This enables to give a topological $\mathbb{Z}\langle e_1,\cdots,e_n\rangle$-grading on the generators of $\mathit{CW}^\ast(L_\alpha,L_\alpha)$. As a $\mathbb{Z}$-graded module, $\mathit{CW}^\ast(L_\alpha,L_\alpha)$ is supported in even degrees, so the Floer differential vanishes and (\ref{eq:aa}) holds as an isomorphism of $\mathbb{Z}\langle e_1,\cdots,e_n\rangle$-graded modules.


\paragraph{Step 2: locality of the Floer product}

Next we consider the triangle product on the Floer cochain complexes $\mathit{CW}^\ast(L_\alpha,L_\alpha)$. Whenever there is a perturbed holomorphic curve contributing to the triangle product
\begin{equation*}
\mu^2:\mathit{CW}^\ast(L_\alpha,L_\alpha)\otimes\mathit{CW}^\ast(L_\alpha,L_\alpha)\rightarrow\mathit{CW}^\ast(L_\alpha,L_\alpha),
	\end{equation*}
	the relative homology class of the output Hamiltonian chord must be equal to the sum of those of the input chords. Any generator which appears in the expression of the product of the generators $\prod_{i\in I}u_i^{k_i}$ and $\prod_{i\in I}u_i^{l_i}$, where $(l_i)_{i\in I}$ is another $|I|$-tuple of positive integers, must have grading $\sum_{i\in I}(2k_i+2l_i)$ and represents the (relative) homology class $(k_1+l_1,\cdots,k_n+l_n)\in\mathbb{Z}^n$ under the isomorphism (\ref{eq:isom}). By our grading convention, we conclude that such a generator must be an integer multiple of $\prod_{i\in I}u_i^{k_i+l_i}$.

To determine the integer coefficients of the Floer product, it is more convenient to modify the wrapping (with the same stop $\mathfrak{s}_\mathbb{V}$) according to the generators.
Indeed, by \cite[Lemma 3.29]{ganatra2020covariantly}, the Floer cochain complex $\mathit{CW}^\ast(L_\alpha,L_\alpha)$ can be computed as the direct limit $\varinjlim_t \mathit{CF}^\ast((L_\alpha)_{t},L_\alpha)$ for a confinal Lagrangian isotopy $(L_\alpha)_{t \ge 0}$ starting from $(L_{\alpha})_0=L_{\alpha}$.
As explained in Step 1, no two generators in $\mathit{CW}^\ast(L_\alpha,L_\alpha)$ have the same $\mathbb{Z}\langle e_1,\cdots,e_n\rangle$-grading, so if we choose a confinal Lagrangian isotopy such that no two generators in 
$\varinjlim_t \mathit{CF}^\ast((L_\alpha)_{t},L_\alpha)$ have the same $\mathbb{Z}\langle e_1,\cdots,e_n\rangle$-grading, then we can guanratee that under the quasi-isomorphism
$\mathit{CW}^\ast(L_\alpha,L_\alpha) \to \varinjlim_t \mathit{CF}^\ast((L_\alpha)_{t},L_\alpha)$, a generator of the former cochain complex is sent to the unique generator of the latter cochain complex with the same $\mathbb{Z}\langle e_1,\cdots,e_n\rangle$-grading.

We can determine the product $\mu^2\left(\prod_{i\in I}u_i^{k_i},\prod_{i\in I}u_i^{l_i}\right)$ by working with a local model near  $\bigcup_{i \in J} D_i$, where $|J|=d$ and $I\subset J\subset\{1,\cdots,n\}$.
By relabeling, we may assume that $J=\{1,\dots,d\}$. Consider the map $\Phi$ from Lemma \ref{sector:Standardization}, under the action of which the symplectic form and its primitive become \eqref{eq:sympform} and \eqref{eq:dphic} respectively. Since the difference between \eqref{eq:dphic} and the Liouville $1$-form
\begin{equation*}
\lambda_0:=-\sum_{i=1}^d2\left(\left(1-\frac{1}{\vert z_i\vert^4}\right)x_i\right)dy_i +\sum_{i=1}^d2\left(\left(1-\frac{1}{\vert z_i\vert^4}\right)y_i\right)dx_i
\end{equation*}
is bounded near $\bigcup_{i \in J} D_i$, by taking a linear interpolation between \eqref{eq:dphic} and $\lambda_0$, based on Lemmas \ref{Prelim:linear-algebra-fact} and \ref{sector:CompletenessCriterion}, we can use Moser's trick to show that $M(\mathbb{V})$ can be equipped with a symplectic form that is symplectomorphic to $\omega$, and equals $d\lambda_0$ in a neighborhood $N_J$ of $\bigcup_{i \in J} D_i$.

Note that locally the symplectic structure $d\lambda_0$ is of product type and for each $i \in J$. We can choose a Hamiltonian $\hbar_i$ which is defined to be $\rho_{\epsilon}(|z_i|)$ (cf. \eqref{eq:rhoe}) near $D_i$ and has its support in $N_j$. The Hamiltonian flow of $\hbar=\sum_{i=1}^d\hbar_i$ wraps $L_{\alpha}$ around $\bigcup_{i=1}^dD_i$. We use a product type perturbation of the form as before so that for each $(n_1,\dots,n_d) \in \mathbb{N}$, there is exactly one time-$1$ Hamiltonian chord with multi-grading $(n_1,\dots,n_d,0,\dots,0)$, where the zeros indicate that there are no wrappings around $D_i$ for $i \notin J$. These Hamiltonian chords correspond to the generators $\prod_{i\in J}u_i^{n_i}$.
Since $I \subset J$, we can compute $\mu^2\left(\prod_{i\in I}u_i^{k_i},\prod_{i\in I}u_i^{l_i}\right)$ under this wrapping.
Since $d\ell_i (X_{\hbar_j})=0$ when $i,j \in J$ and $i \neq j$, for any Floer triangle $u:S \to M(\mathbb{V})$ contributing to $\mu^2\left(\prod_{i\in I}u_i^{k_i},\prod_{i\in I}u_i^{l_i}\right)$, we can show that the image of $u$ is contained in $N_J$ using the open mapping theorem. See for example \cite[Remark 4.4.3]{BC1}. More precisely, choose strip-like ends
\[\varepsilon_{0,+},\varepsilon_{1,+}:[0,\infty)\rightarrow S\textrm{ and }\varepsilon_-:(-\infty,0]\rightarrow S\]
and consider the map $\ell_i \circ u:S\rightarrow\mathbb{C}$. Since
\begin{itemize}
	\item $\ell_i \circ u$ is holomorphic at $z \in S$ as long as $|\ell_i \circ u(z)| > \epsilon$ (recall that $\hbar_i=\rho_{\epsilon}(|z_i|)$, so the $\ell_i$-projection of the support of the perturbation $\hbar_i$ lies in $\left\{|z_i| \le \epsilon\right\}$);
	\item the limits $\lim_{s\rightarrow\infty}\varepsilon_{j,+}^\ast(\ell_i\circ u)(s,t)$ for $j=0,1$ and $\lim_{s\rightarrow-\infty}\varepsilon_-^\ast(\ell_i\circ u)(s,t)$ exist, where $s,t$ are coordinates on the strip-like ends, and they lie inside the neighborhood $\left\{|z_i| \le \epsilon\right\}$;
	\item the projections under $\ell_i$ of the Lagrangian boundaries are contained in $\left\{|z_i| \le \epsilon\right\}$,
\end{itemize}
we conclude that $\mathit{im}(\ell_i\circ u)\cap\{|z_i|>\epsilon\}=\emptyset$. In other words, the image of $\ell_i \circ u$ is contained in $\{|z_i| \le \epsilon\}$.

\paragraph{Step 3: Coefficients of the product structure}
Near the tubular neighborhood $N_J$, the wrapping Hamiltonian is modeled on a standard product Hamiltonian in a neighborhood of the origin in $(\mathbb{C}^\ast)^d$, and the Lagrangian $L_\alpha$ maps locally to one of the orthants in the real locus (see Example \ref{Prelim:generation-algebraic-torus}). As argued above, any perturbed holomorphic curve $u$ contributing to the product $\mu^2$ must remain entirely within the local chart $N_J$. Moreover, the projection of $u$ to each coordinate factor $z_i$ is a perturbed holomorphic curve in a neighborhood of the origin in $\mathbb{C}^\ast$ with boundary on the appropriate arcs. Conversely, every tuple of index 0 perturbed holomorphic curve in the coordinate factors lifts to an index 0 perturbed holomorphic curve in the local chart $N_J$. Note that the generator $\prod_{i\in I}u_i^{k_i}$ projects to the generator $x_{-k_j}$ of the wrapped Floer complex of $\mathbb{R}\times\{1\}\subset\mathbb{R}\times S^1$ on the $z_j$ factor; see Figure \ref{fig:wrapping} and the reference therein. On the cylinder, the generators $x_{-k_j}$ and $x_{-l_j}$ are the inputs of a unique triangle contributing to the Floer product, whose output can be identified with $x_{-k_j-l_j}$. We have proved that
	\begin{equation*}
	\mu^2\left(\prod_{i\in I}u_i^{k_i},\prod_{i\in I}u_i^{l_i}\right)=\prod_{i\in I}u_i^{k_i+l_i},
	\end{equation*}
	hence the ring structure on $\mathit{HW}^\ast(L_\alpha,L_\alpha)$ is as expected.
\end{proof}

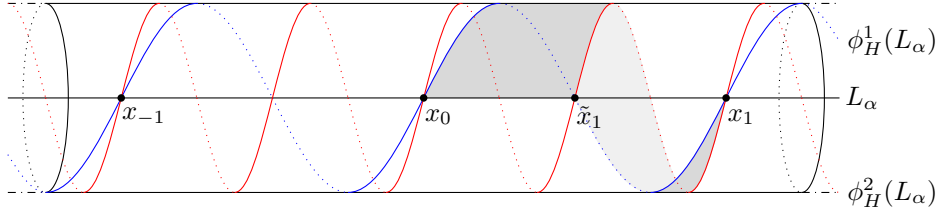
\begin{figure}
	\centering
	\begin{tikzpicture}
    \filldraw [draw=black, color={black!15}] (7.5,0) cos (7,-1.25) sin (5,-1.25) sin (6,0);
    \filldraw [draw=black,color={black!6}] (7,-1.25) sin (7.5,0) cos (8,-1.25) sin (8.5,-2.5)--(8,-2.5) cos (7,-1.25);
    \filldraw [draw=black,color={black!15}] (8,-2.5) cos (9,-1.25) sin (8.5,-2.5)--(8,-2.5);
	\draw (0,0) to (10,0);
	\draw (0,-2.5) to (10,-2.5);
	\draw (-0.5,-1.25) to (10.5,-1.25);
	\draw [dash dot] (-0.5,0) to (0,0);
	\draw [dash dot] (10,0) to (10.5,0);
	\draw [dash dot] (-0.5,-2.5) to (0,-2.5);
	\draw [dash dot] (10,-2.5) to (10.5,-2.5);
	\draw (10.8,-1.25) node {$L_\alpha$};
	\draw [dotted] (0,0) arc(90:270:0.3 and 1.25);
	\draw (0,-2.5) arc(-90:90:0.3 and 1.25);
	\draw [dotted] (10,0) arc(90:270:0.3 and 1.25);
	\draw (10,-2.5) arc(-90:90:0.3 and 1.25);
	\draw [blue,dotted] (-0.5,-2) sin (0,-2.5);
	\draw [blue] (0,-2.5) cos (1,-1.25);
	\draw [blue] (1,-1.25) sin (2,0);
	\draw [blue,dotted] (2,0) cos (3,-1.25);
	\draw [blue,dotted] (3,-1.25) sin (4,-2.5);
	\draw [blue] (4,-2.5) cos (5,-1.25);
	\draw [blue] (5,-1.25) sin (6,0);
	\draw [blue,dotted] (6.,0) cos (7,-1.25);
	\draw [blue,dotted] (7,-1.25) sin (8,-2.5);
	\draw [blue] (8,-2.5) cos (9,-1.25);
	\draw [blue] (9,-1.25) sin (10,0);
	\draw [blue,dotted] (10,0) cos (10.5,-0.5);
	
	\draw [red,dotted] (-0.5,0) cos (0,-1.25);
	\draw [red,dotted] (0,-1.25) sin (0.5,-2.5);
	\draw [red] (0.5,-2.5) cos (1,-1.25);
	\draw [red] (1,-1.25) sin (1.5,0);
	\draw [red,dotted] (1.5,0) cos (2,-1.25);
	\draw [red,dotted] (2,-1.25) sin (2.5,-2.5);
	\draw [red] (2.5,-2.5) cos (3,-1.25);
	\draw [red] (3,-1.25) sin (3.5,0);
	\draw [red,dotted] (3.5,0) cos (4,-1.25);
	\draw [red,dotted] (4,-1.25) sin (4.5,-2.5);
	\draw [red] (4.5,-2.5) cos (5,-1.25);
	\draw [red] (5,-1.25) sin (5.5,0);
	\draw [red,dotted] (5.5,0) cos (6,-1.25);
	\draw [red,dotted] (6,-1.25) sin (6.5,-2.5);
	\draw [red] (6.5,-2.5) cos (7,-1.25);
	\draw [red] (7,-1.25) sin (7.5,0);
	\draw [red,dotted] (7.5,0) cos (8,-1.25);
	\draw [red,dotted] (8,-1.25) sin (8.5,-2.5);
	\draw [red] (8.5,-2.5) cos (9,-1.25);
	\draw [red] (9,-1.25) sin (9.5,0);
	\draw [red,dotted] (9.5,0) cos (10,-1.25);
	\draw [red,dotted] (10,-1.25) sin (10.5,-2.5);
	\draw (1,-1.25) node[circle,fill,inner sep=1pt] {};
	\draw (9,-1.25) node[circle,fill,inner sep=1pt] {};
	\draw (5,-1.25) node[circle,fill,inner sep=1pt] {};
	\draw (7,-1.25) node[circle,fill,inner sep=1pt] {};
	\draw (1.3,-1.5) node {$x_{-1}$};
	\draw (5.2,-1.5) node {$x_0$};
	\draw (9.2,-1.5) node {$x_1$};
	\draw (7.2,-1.5) node {$\tilde{x}_1$};
	\draw (11.2,-0.5) node {$\phi_H^1(L_\alpha)$};
	\draw (11.2,-2.5) node {$\phi_H^2(L_\alpha)$};
	\end{tikzpicture}
	\caption{\cite[Figure 12]{Auroux2014Intro} Wrapping in the cylinder, the shaded region is a triangle contributing to the triangle product between $x_0$ and $x_1$. Under the identification $\mathit{CW}^\ast\left(L_\alpha,\phi_H^1(L_\alpha)\right)\cong\mathit{CW}^\ast\left(L_\alpha,\phi_H^2(L_\alpha)\right)$, the generator $x_1$ maps to $\tilde{x}_1$.\label{fig:wrapping}}
\end{figure}

\begin{proposition}\label{proposition:ab}
Given $\alpha,\beta\in\mathscr{P}(\mathbb{V})$, we have an isomorphism
\begin{equation}\label{eq:ab}
\mathit{HW}^\ast(L_\alpha,L_\beta)\cong\widetilde{R}_{\alpha\beta}\cdot f_{\alpha\beta}
\end{equation}
as $\mathbb{Z}$-graded $\left(\mathit{HW}^\ast(L_\alpha,L_\alpha),\mathit{HW}^\ast(L_\beta,L_\beta)\right)$ and $(\widetilde{R}_{\alpha\alpha},\widetilde{R}_{\beta\beta})$-bimodules.
\end{proposition}
\begin{proof}
	We argue similarly as in the proof of Proposition \ref{proposition:aa}. As before, if $\Delta_\alpha\cap\Delta_\beta\cap H_{I}=\emptyset$, which means the closures of $L_\alpha$ and $L_\beta$ in $X_\mathbb{V}$ do not intersect with each other near $U_I^\circ$, then there is no generator of $\mathit{CW}^\ast(L_\alpha,L_\beta)$ in $U_I^\circ$. 
	
	Now assume that $\Delta_\alpha\cap\Delta_\beta\cap H_{I}\neq\emptyset$. In this case, the local coordinates $z_i=\ell_i$ define a local projection $U_I\rightarrow\mathbb{C}^{|I|}$, under which $L_\alpha$ and $L_\beta$ map to orthants in the real locus. These orthants correspond to real points whose coordinates have the same signs for those $i\in I$ with $\alpha(i)=\beta(i)$, and have different signs otherwise. Let $K\subset I$ be the subset defined by the elements $i\in I$ with $\alpha(i)=\beta(i)$. Given any tuple $(k_i)_{i\in I}$, where $k_i\in\mathbb{N}$ for $i\in K$, and $k_i\in\mathbb{Z}_{\geq0}+\frac{1}{2}$ for $i\in I\setminus K$, near $D_{I}^\circ$ there is a family of time-1 trajectories of $X_h$ from $L_\alpha$ to $L_\beta$ that wraps $k_i$ times around the hyperplane $D_{i}$ for each $i\in I$. After perturbing the Hamiltonian slightly from $h$ to $\tilde{h}$, there is a single non-degenerate such trajectory, and we label the corresponding generator of $\mathit{CW}^\ast(L_\alpha,L_\beta)$ by $\prod_{i\in I}u_i^{k_i}$. Note that a key difference between the situation of Proposition \ref{proposition:aa} and the situation here is that some of the $k_i$ are half-integers, and we introduce the notation $f_{\alpha\beta}$ by requiring
	\begin{equation*}
	\prod_{i\in I}u_i^{k_i}=\prod_{i\in I}u_i^{\lfloor{k_i}\rfloor}\cdot f_{\alpha\beta}.
	\end{equation*}
	Note that by our grading conventions, the generator $u_i$ has degree 2, therefore $u_i^{1/2}$ has degree $1$. As a consequence, $f_{\alpha\beta}$ has degree $d_{\alpha\beta}$. Note that when $I=\{0\}$, $L_\alpha$ and $L_\beta$ represent the same orthant in the local projection to $\mathbb{C}$, so by our choice of the stop $\mathfrak{s}_\mathbb{V}$, wrapping at infinity does not add any new generators to $\mathit{CW}^\ast(L_\alpha,L_\beta)$. We have proved that as a graded vector space, $\mathit{CW}^\ast(L_\alpha,L_\beta)$ is isomorphic to the right-hand side of (\ref{eq:ab}).
	
	We need to prove the vanishing of the Floer differential. Since $L_\alpha$ and $L_\beta$ are contractible, by choosing base points $\ast_\alpha\in L_\alpha$, $\ast_\beta\in L_\beta$, and $\ast\in M(\mathbb{V})$, and reference path from $\ast$ to $\ast_\alpha$, and from $\ast$ to $\ast_\beta$, we can complete any arc connecting $L_\alpha$ to $L_\beta$ into a closed loop in $M(\mathbb{V})$, which is unique up to homotopy. We can use this to assign classes in $H_1(M(\mathbb{V});\mathbb{Z})\cong\mathbb{Z}^n$ to the generators of $\mathit{CW}^\ast(L_\alpha,L_\beta)$. In fact, in our case it would be more convenient to consider $H_1(M(\mathbb{V});\frac{1}{2}\mathbb{Z})\cong(\frac{1}{2}\mathbb{Z})^n$ so that the class associated to the generator $\prod_{i\in I}u_i^{k_i}\in\mathit{CW}^\ast(L_\alpha,L_\beta)$ is $(k_1,\cdots,k_n)\in(\frac{1}{2}\mathbb{Z})^n$, where $k_i=0$ if $i\notin I$. Any two generators of $\mathit{CW}^\ast(L_\alpha,L_\beta)$ related by the Floer differential must represent the same class in $(\frac{1}{2}\mathbb{Z})^n$, and their gradings differ by 1, which is impossible unless the differential vanishes identically.
	
	The claim about the module structures is a consequence of the following proposition.
\end{proof}

\begin{proposition}
Indexing the generators of $\mathit{HW}^\ast(L_\alpha,L_\beta)$ by monomials in $u_1,\cdots,u_n$ and $f_{\alpha\beta}$ as in Proposition \ref{proposition:ab}, the Floer product
\begin{equation*}
\mathit{HW}^\ast(L_\beta,L_\gamma)\otimes\mathit{HW}^\ast(L_\alpha,L_\beta)\rightarrow\mathit{HW}^\ast(L_\alpha,L_\gamma)
\end{equation*}
coincides with the product in the algebra $\widetilde{B}(\mathbb{V})$.
\end{proposition}
\begin{proof}
The argument is similar to the special case of $\alpha=\beta=\gamma$ treated in Proposition \ref{proposition:aa}, except the relation (\ref{eq:multiplication}) involving half-integer powers of $u_i$. We again work near $\bigcup_{i \in J} D_i$, where $|J|=d$ and $I\subset J\subset\{1,\cdots,n\}$. By a Moser argument as in the proof of Proposition \ref{proposition:aa}, the neighborhood $N_J$ of $\bigcup_{i \in J} D_i$ is modeled on the standard product $(\mathbb{C}^\ast)^d$ and the wrapping Hamiltonian $\sum_{i \in J}\hbar_i$ has been chosen accordingly. Now suppose that $\alpha(j)=\gamma(j)\neq\beta(j)$. On the $j$-th coordinate factor $z_j$, the Lagrangians $L_\alpha$, $L_\beta$, and $L_\gamma$ project to the arcs $\mathbb{R}_+$, $\mathbb{R}_-$ and $\mathbb{R}_+$ in  $\mathbb{C}^\ast$, respectively. The generators $f_{\alpha\beta}\in\mathit{CW}^\ast(L_\alpha,L_\beta)$, $f_{\beta\gamma}\in\mathit{CW}^\ast(L_\beta,L_\gamma)$, and $f_{\alpha\gamma}\in\mathit{CW}^\ast(L_\alpha,L_\gamma)$ are projected to $x_{-\frac{1}{2}}\in\mathit{CW}^\ast(\mathbb{R}_+,\mathbb{R}_-)$, $x_{\frac{1}{2}}\in\mathit{CW}^\ast(\mathbb{R}_-,\mathbb{R}_+)$ and $x_{-1}\in\mathit{CW}^\ast(\mathbb{R}_+,\mathbb{R}_+)$, respectively, where the wrapped Floer cochain complexes $\mathit{CW}^\ast$ are taken inside $\mathbb{C}^\ast$. Since there is a unique holomorphic triangle with inputs $x_{-\frac{1}{2}}$, $x_{\frac{1}{2}}$, and output $x_{-1}$, we get an additional $u_j$ factor before $f_{\alpha\gamma}$ in the expression of $f_{\alpha\beta}\cdot f_{\beta\gamma}$. In terms of half-integer powers of $u_i$, this simply corresponds to the identity $u_i^{1/2}\cdot u_i^{1/2}=u_i$. Taking into account of all such $j\in J$ proves (\ref{eq:multiplication}).
\end{proof}

We have proved that the partially wrapped Floer cohomology algebra is isomorphic to $\widetilde{B}(\mathbb{V})$, i.e.
\begin{equation}\nonumber
\widetilde{B}(\mathbb{V})\cong\bigoplus_{\alpha,\beta\in\mathscr{P}(\mathbb{V})}\mathit{HW}^\ast(L_\alpha,L_\beta).
\end{equation}
In order to obtain a statement on the level of $A_\infty$-algebras, we need the following formality result.

\begin{theorem}\label{thm:formality}
	The partially wrapped Fukaya $A_\infty$-algebra
	\begin{equation}\nonumber
	\widetilde{\mathcal{B}}_{\mathbb{V}}:=\bigoplus_{\alpha,\beta\in\mathscr{P}(\mathbb{V})}\mathit{CW}^\ast(L_\alpha,L_\beta)
	\end{equation}
	is formal. As a consequence, it is quasi-isomorphic to the $\mathbb{Z}$-graded associative algebra $\widetilde{B}(\mathbb{V})$.
\end{theorem}
\begin{proof}
	As before, label the generators of $\widetilde{\mathcal{B}}_{\mathbb{V}}$ by monomials in the variables $u_1,\cdots,u_n$ with half-integral powers. Suppose that there is an $A_\infty$-operation
	\begin{equation}\nonumber
	\mu^j\left(\prod_{i\in I_1}u_i^{k_{1,i}},\cdots,\prod_{i\in I_j}u_i^{k_{j,i}}\right)\in\widetilde{\mathcal{B}}_{\mathbb{V}}
	\end{equation}
	with non-trivial output in $\widetilde{\mathcal{B}}_{\mathbb{V}}$. Homological considerations as in the proof of Proposition \ref{proposition:ab} tells us that the right-hand side of the above expression must have the same homology class in $H_1(M(\mathbb{V});\frac{1}{2}\mathbb{Z})\cong(\frac{1}{2}\mathbb{Z})^n$ as the sum of the inputs on the left-hand side. This, however, forces the grading of the right-hand side of $\mu^j$ to be the sum of the gradings of the inputs. Since $\mu^j$ has degree $2-j$, we then conclude that $\mu^j$ vanishes unless $j=2$.
\end{proof}

\begin{corollary}\label{identify-perfect-complex}
There is a quasi-equivalence between $A_\infty$-categories:
\begin{equation*}
\label{eq:Equivalence cat}
\mathcal{W}^\mathit{perf}(M(\mathbb{V}), \xi) \cong \mathit{Perf}\left(\widetilde{B}(\mathbb{V})\right),
\end{equation*}
where $\mathit{Perf}\left(\widetilde{B}(\mathbb{V})\right)$ is the dg category of perfect modules over the graded algebra $\widetilde{B}(\mathbb{V})$.
\end{corollary}

\section{Functoriality}\label{functoriality-section}

In this section, we study the geometric interpretation of the combinatorial operations, such as deletion and restriction, on the polarized hyperplane arrangement $\mathbb{V}$. 
\subsection{Deletion and Restriction.}\label{sec:Delandres HypArr}

We recall the notions from \cite{LLM2020}.

\begin{definition} 
Let $\mathbb{V}=(\mathbb{R}^d,\eta ,\xi)$ be a polarized hyperplane arrangement and $H_{\mathbb{R},i}$ be one of the hyperplanes in $\mathbb{V}$. 
	\begin{enumerate}[\indent (i)]
		\item the \emph{($i$-th) deletion} of $\mathbb{V} =(\mathbb{R}^d,\eta ,\xi)$ is the polarized hyperplane arrangement 
		\begin{equation*}
		\mathbb{V}_i:=(\pi_i\left(\mathbb{R}^d\right),\pi_i (\eta),\xi\circ\pi_i^{-1}),
		\end{equation*}
		where $\pi_i:\R^d \to \R^{d-1}$ is the $i$-th projection. Equivalently, this is obtained by deleting $H_{\mathbb{R},i}$ from the set of hyperplanes in $\mathbb{V}$.  
		\item the \emph{($i$-th) restriction} of  $\mathbb{V} =(\mathbb{R}^d,\eta ,\xi)$ is the polarized hyperplane arrangement 
		\begin{equation*}
		\mathbb{V}^i:=(\iota^{-1}_i\mathbb{R}^d,\iota_i^{-1}\eta,\iota_i^{\ast}\xi),
		\end{equation*}
		where $\iota_i\colon\mathbb{R}^{d-1}\to\mathbb{R}^d$ is the inclusion into the $i$-th coordinate hyperplane. Equivalently, this is obtained by taking the intersection with $H_{\mathbb{R},i}$. 
	\end{enumerate}
\end{definition}

Fix $i$ and choose a sign $s \in \{+,-\}$. Given a sign sequence $\alpha \in \mathscr{P}(\mathbb{V}_i)$, we define $\alpha^i_s$ to be the sign sequence with $s$ inserted at $i$-th position. Consider the map 
\begin{equation*}
\begin{aligned}
\mathrm{del}_i^s:\mathscr{P}(\mathbb{V}_i) &\to && \mathscr{P}(\mathbb{V}) \\
\alpha & \mapsto && \begin{cases}
\alpha^i_s & \alpha^i_s \in \mathscr{P}(\mathbb{V}) \\
0 & \alpha^i_s \notin \mathscr{P}(\mathbb{V})
\end{cases}.
\end{aligned}
\end{equation*}  
Similarly, for $\alpha \in \mathscr{P}(\mathbb{V})$, we define $\alpha^{(i)}$ to be the sign sequence with the $i$-th sign removed. Consider the map 
\begin{equation*}
\begin{aligned}
\rest_i^s:\mathscr{P}(\mathbb{V}) &\to && \mathscr{P}(\mathbb{V}^i) \\
\alpha & \mapsto && \begin{cases}
\alpha^{(i)} & \alpha(i)=s \\
0 & \alpha(i) \neq s
\end{cases}.
\end{aligned}
\end{equation*}  
These maps on the sign sequences induce well-defined graded algebra morphisms between the corresponding convolution algebras.
\begin{proposition}[\cite{BLPW2010}, \cite{LLM2020}]
	Let $\mathbb{V}$ be a polarized hyperplane arrangement as above. For each $i$ and sign $s\in \{+,-\}$,
	\begin{enumerate}[\indent (i)]
		\item the map $\mathrm{del}_i^s:\mathscr{P}(\mathbb{V}_i) \to \mathscr{P}(\mathbb{V}) $ induces a graded algebra morphism 
		\begin{equation*}
		\mathrm{del}_i^s:\widetilde{B}(\mathbb{V}_i) \to \widetilde{B}(\mathbb{V}),
		\end{equation*}
		\item the map $\rest_i^s:\mathscr{P}(\mathbb{V}) \to \mathscr{P}(\mathbb{V}^i) $ induces a graded algebra morphism 
		\begin{equation*}
		\rest_i^s:\widetilde{B}(\mathbb{V}) \to \widetilde{B}(\mathbb{V}^i).
		\end{equation*}
	\end{enumerate}
\end{proposition}

Note that these graded algebra morphisms induce natural functors on the categories of perfect modules. For instance, the restriction morphism $\rest_i^s$ induces a functor $\rest_i^s:\mathit{Perf}\left(\widetilde{B}(\mathbb{V})\right) \to \mathit{Perf}\left(\widetilde{B}(\mathbb{V}^i)\right)$. By Corollary \ref{identify-perfect-complex}, there is also an $A_{\infty}$-functor
\begin{equation} \nonumber
\rest_i^s:\mathcal{W} (M(\mathbb{V}),\xi)\to\mathcal{W}\left(M(\mathbb{V}^i),\iota_i^*\xi\right)
\end{equation}
given as follows. For any chamber Lagrangian $L_{\alpha}$, if $\alpha (i)=-s$, then $\rest_i^s (L_{\alpha})=0$; if $\alpha (i)=s$, then $\rest_i^s (L_{\alpha})=\left(\bar{L}_{\alpha}\cap H_i\right)\setminus\bigcup_{j\neq i} H_j$ is the chamber Lagrangian in $H_i$ which is the boundary strata of $\bar{L}_{\alpha}$. Given two chamber Lagrangians $L_{\alpha}, L_{\beta}$, we can define a chain map $\mathit{CW}^{\ast} (L_{\alpha},L_{\beta})\to\mathit{CW}^{\ast}\left(\rest_i^s(L_{\alpha}), \rest_i^s(L_{\beta})\right)$ as the homomorphism $\rest_i^s\colon\widetilde{R}_{\alpha\beta}\to\widetilde{R}_{\alpha^{(i)}\beta^{(i)}}$. The formality of the Fukaya $A_\infty$-algebra $\widetilde{\mathcal{B}}_\mathbb{V}$, together with the generation of the Fukaya category $\mathcal{W} (M(\mathbb{V}),\xi)$ by chamber Lagrangians ensures that this gives a well-defined $A_{\infty}$-functor. It follows directly from the construction that 

\begin{corollary}\label{restriction-functor-strict}
The functor $\rest_i^s\colon\mathcal{W}(M(\mathbb{V}),\xi)\to\mathcal{W}\left(M(\mathbb{V}^i),\iota^*_i\xi\right)$ is strict.
\end{corollary}
If it is clear from the context, we simply write $\iota_i^*\xi$ as $\xi$. The deletion functor $\mathrm{del}_i^s:\mathcal{W}(M(\mathbb{V}_i), \xi) \to \mathcal{W}(M(\mathbb{V}), \xi)$ can be defined in a similar way. We leave the precise construction for the reader. 

\begin{remark}
It is interesting to ask whether the restriction functor $\rest_i^s$ defined in \cite{LLM2020} has an interpretation in terms of symplectic geometry. However, it turned out that all functors we obtained from the restriction of hyperplane arrangements in the following sections are neither $\rest_i^s$ nor the adjoints to it. It remains to see whether $\rest_i^s$ has geometric interpretation or not.
\end{remark}

\subsection{Geometric Functors from Restriction and Deletion.} In this subsection, we discuss several functors arising from restriction and deletion of polarized hyperplane arrangements. Let $\mathcal{H}_\mathbb{C}=\{H_1,\cdots,H_n\}$ be the collection of the complexifications of the hyperplanes in $\mathbb{V}$. Fix a hyperplane $H_i\in\mathcal{H}_\mathbb{C}$, and let $\ell_i$ be the defining function of $H_i$. Then for $\varepsilon >0$, after a possible Liouville deformation, we have a stopped inclusion
\[
\rRe(\ell_i)^{-1} ([-\varepsilon ,\varepsilon]) \cap (M(\mathbb{V}), \xi)\hookrightarrow (M(\mathbb{V}),\xi),
\]
where $\rRe(\ell_i)^{-1} ([-\varepsilon ,\varepsilon])\cap (M(\mathbb{V}), \xi)$ splits into a product
\begin{equation*}
\begin{aligned}
\rRe(\ell_i)^{-1} ([-\varepsilon ,\varepsilon])\cap (M(\mathbb{V}), \xi) &\simeq \left(H_i\setminus\bigcup_{H\in\mathcal{H}_\mathbb{C}} H_i \cap H,\xi\vert_{H_i}\right)\times \left(T^{\ast} [-\varepsilon ,\varepsilon]^\times \right) \\
&\simeq (M(\mathbb{V}^i), \xi)\times \left(T^{\ast} [-\varepsilon ,\varepsilon]^\times \right),
\end{aligned}
\end{equation*}
so the K\"unneth functor gives us a quasi-isomorphism 
\[
	\mathcal{W}\left(\rRe(\ell_i)^{-1} ([-\varepsilon ,\varepsilon])\cap (M(\mathbb{V}),\xi)\right)\xrightarrow{\simeq}\mathcal{W}\left(M(\mathbb{V}^i),\xi\right)\otimes\mathcal{W} (T^{\ast} [-\varepsilon ,\varepsilon]^{\times}).
\]
\begin{figure}[ht]
    \centering
    \subfigure[Cocore and Linking Disks of the Sector.]{\hspace{2cm}\begin{tikzpicture}[scale=0.5]
        \draw [red, ultra thick] (1,-3) -- (1,3);
        \draw [red, ultra thick] (-1,-3) -- (-1,3);
        \draw [ultra thick, green!50!black] (0,0.25) -- (0,3) node [above,green!50!black] {$L^+$};
        \draw [ultra thick, green!50!black] (0,-0.25) -- (0,-3) node [below, green!50!black] {$L^-$};
        \draw [ultra thick] (0,0) circle [radius=0.25];
		\draw [ultra thick, black!50!white] (-0.75, 3) -- (-0.75,-3) node [below left] {$D^+$};
		\draw [ultra thick, black!50!white] (0.75, 3) -- (0.75,-3) node [below right] {$D^-$};
    \end{tikzpicture}\hspace{2cm}}\hspace{1.5cm}\subfigure[Core of the Sector.]{\hspace{1.5cm}\begin{tikzpicture}[scale=0.5]
        \draw [ultra thick, orange] (-0.75,0) -- (-1,0);
        \draw [ultra thick, orange] (0.75,0) -- (1,0);
        \draw [red, ultra thick] (1,-3) -- (1,3);
        \draw [red, ultra thick] (-1,-3) -- (-1,3);
        \draw [ultra thick] (0,0) circle [radius=0.25];
        \draw [ultra thick, orange] (0,0) circle [radius=0.75];
    \end{tikzpicture}\hspace{1.5cm}}
    \caption{Core, Cocore and Linking Disks of the Sector $T^{\ast} [-\varepsilon,\varepsilon]^{\times}$.\label{func:cotangent-bundle-interval-removing-origin}}
\end{figure}
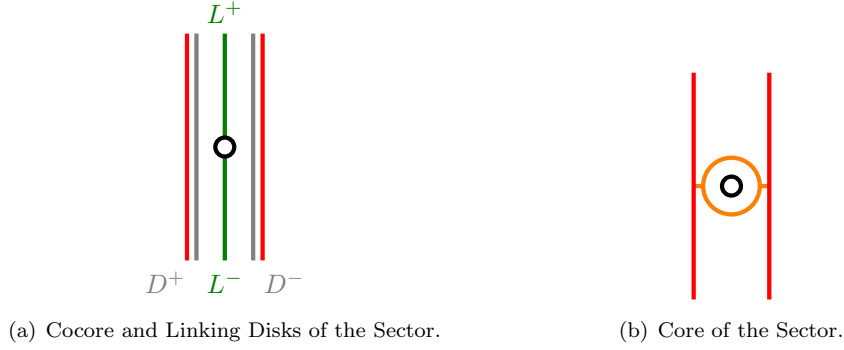
\noindent Here we regard $T^*[-\varepsilon, \varepsilon]$ as a Lioville sector with two disjoint sectorial hypersurfaces corresponding to the cotangent fibers at $\pm \varepsilon$, and $T^*[-\varepsilon, \varepsilon]^\times:=T^*[-\varepsilon, \varepsilon] \setminus \{(0,0)\}$. Therefore, there are two linking disks in $T^{\ast} [-\varepsilon ,\varepsilon]^{\times}$, one for the left end and the other for the right end. We write $D^s$ to be the left or right linking disks, depending on the chosen normal vectors of the hyperplane $H_i$. Since the normal vector defines the ``positive'' and ``negative'' part of the two half-spaces separated by $\rRe(\ell_i)^{-1} (0)$, we let $s$ be the corresponding sign. 

Note that $T^{\ast} [-\varepsilon ,\varepsilon]^{\times}$ has four natural Lagrangians: two generating Lagrangians $L^s$ and two linking disks $D^s$ as depicted in Figure \ref{func:cotangent-bundle-interval-removing-origin}, so we get four functors by tensoring any chamber Lagrangians in $\mathcal{W}\left(M(\mathbb{V}^i),\xi \right)$ with anyone of them (on the level of morphisms, it is given by tensoring with the unit), denoted by $\Psi_i^{s,\ast}\colon\mathcal{W}\left(M(\mathbb{V}^i),\xi\right)\to\mathcal{W}\left(\rRe(\ell_i)^{-1} ([-\varepsilon ,\varepsilon])\cap (M(\mathbb{V}), \xi)\right)$. Note that $\Psi_i^{s,L}$ and $\Psi_i^{s,D}$ would determine $\Psi_i^{-s,D}$ and $\Psi_i^{-s,L}$ via mapping cones, it suffices to compute two of them, and without loss of generality, we consider $\Psi_i^{+,L}$ and $\Psi_i^{s,D}$. Composing them with the stopped inclusion $\mathcal{W}\left(\rRe(\ell_i)^{-1} ([-\varepsilon ,\varepsilon])\cap (M(\mathbb{V}), \xi)\right)\to\mathcal{W} (M(\mathbb{V}),\xi)$, we obtain the restriction functors
\begin{equation}
\Res_i^{s,\ast}:\mathcal{W}\left(M(\mathbb{V}^i),\xi \right) \to \mathcal{W}(M(\mathbb{V}),\xi). 
\end{equation}

Given a standard Lagrangian $L_j$ in $\mathcal{W}\left(M(\mathbb{V}^i),\xi \right)$, there is a unique standard Lagrangian $\bar{L}_j$ in $\mathcal{W}(M(\mathbb{V}),\xi)$, whose closure in $\mathbb{C}^d$ contains $L_j$ as one of its face and $\rIm(\bar{L}_j) \geq 0$ corresponding to the choice  $\Psi_i^+$. Figure \ref{func:restriciton-functor} shows one such pairs in two-dimensional pairs-of-pants.
	
	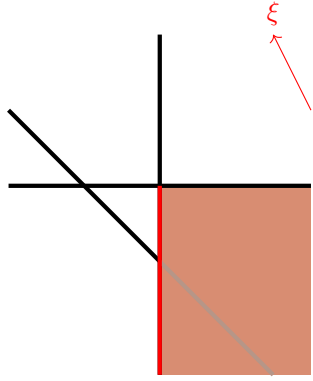
\begin{figure}[ht]
		\centering
		\begin{tikzpicture}[scale=0.5]
        \fill [BrickRed!50] (0,0) -- (0,-5) -- (4,-5) -- (4,0);
		\draw [ultra thick] (-4,0) -- (4,0);
        \draw [ultra thick] (0,-5) -- (0,4);
        \draw [ultra thick] (-4,2) -- (0,-2);
        \draw [ultra thick, black!50!BrickRed!50] (0,-2) -- (3,-5);
        \draw [->, red] (4,2) -- (3,4) node[above] {$\xi$};
		\draw [ultra thick, red] (0,0) -- (0,-5);
		\end{tikzpicture}
		\caption{The Restriction Functor.\label{func:restriciton-functor}}
	\end{figure}

	Proposition \ref{surgery:decomposing-standard-Lag} then implies that $\mathrm{Res}_i^{+,L} (L_j) =\bar{L}_j$ (on the level of morphisms, this functor is faithful but not full). Thus this functor sends standard Lagrangians to standard Lagrangians, although it does not necessarily preserve chamber Lagrangians. 
	
\subsubsection*{The Deletion Functor.}
 Let $H_i\in\mathcal{H}_\mathbb{C}$ be a hyperplane with the defining equation $\ell_i$, and let $\mathbb{V}_i$ be the hyperplane arrangement obtained by deleting $H_i$ from $\mathbb{V}$. Similar as before, we consider a product decomposition 

\[
\rRe(\ell_i)^{-1} ([-\varepsilon ,\varepsilon])\cong\ell_i^{-1} (0)\times T^{\ast} [-\varepsilon ,\varepsilon].
\]
We describe the operation of deleting the hyperplane $H_i$ from $\mathbb{V}$ as the following ``shopping bag construction''.

\begin{lemma}\label{lem:shopping bag}
There exists a Weinstein cobordism $\Phi_s$ from sector $T^{\ast} [-1,1]$ to $T^{\ast} [-1,1]^\times$, depending on the parity $s \in \{+,-\}$.
\end{lemma}

\begin{proof}
Attach a $1$-handle to $[-1,1]\times [-1,1]$ on the top edge $(s=-)$,  the resulting space is isomorphic to $T^{\ast} [-1,1]^\times \cong (\mathbb{C}^{\ast} ,\pm\infty)$. See Figure \ref{sector:ShoppingbagConstruction}, where the red line segments are sectorial boundaries, the black line segments are ideal boundaries for $T^{\ast} [-1,1]$, while the blue arcs denote the $1$-handles. Similarly, one can attach a 1-handle to $[-1,1] \times [-1,1]$ at the bottom edge $(s=+)$.  
\end{proof}
	
\begin{figure}[ht]
    \centering
    \begin{tikzpicture}
    \draw [ultra thick, red] (-1,-1) -- (-1,1);
    \draw [ultra thick, red] (1,-1) -- (1,1);
    \draw [ultra thick] (-1,1) -- (-0.75,1);
	\draw [ultra thick, dashed] (-0.75,1) -- (-0.25,1);
	\draw [ultra thick] (-0.25, 1) -- (0.25,1);
	\draw [ultra thick, dashed] (0.25, 1) -- (1,1);
	\draw [ultra thick] (0.75, 1) -- (1,1);
    \draw [ultra thick] (-1,-1) -- (1,-1);
    \draw [ultra thick, black!60!blue] (0.75,1) arc [start angle=0, end angle =180, radius=0.75];
    \draw [ultra thick, black!60!blue] (0.25,1) arc [start angle=0, end angle =180, radius=0.25];
    \end{tikzpicture}
    \caption{The Shoppingbag Construction.\label{sector:ShoppingbagConstruction}}
\end{figure}
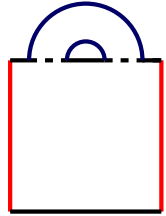

By Lemma \ref{lem:shopping bag}, for $s \in \{+,-\}$, we define the deletion functor
\begin{equation*}
\mathrm{Del}_i^s:\mathcal{W} (M(\mathbb{V}),\xi)\to\mathcal{W}\left(M(\mathbb{V}_i),\xi \right),
\end{equation*}
as the gluing of the Viterbo restriction functor 
\[
\Phi_s:\mathcal{W}\left(\ell^{-1} (0)\times (\mathbb{C}^{\ast} ,\pm\infty)\right)\to\mathcal{W}\left(\ell^{-1} (0)\times T^{\ast} [-\varepsilon ,\varepsilon]\right)\cong\mathcal{W}\left(\ell^{-1} (0)\right)
\]
with the identity functors of the remaining parts. More precisely, note that we have the sectorial decompositions 
\[
	(M(\mathbb{V}_i),\xi) =\left(\ell^{-1}_i (-\infty ,-\varepsilon],\xi\right)\cup\left(\ell^{-1}_i (0)\times T^{\ast} [-\varepsilon ,\varepsilon],\xi\right)\cup\left(\ell^{-1}_i [\varepsilon,+\infty),\xi\right),
	\]and\[
		(M(\mathbb{V}),\xi) =\left(\ell^{-1}_i (-\infty,-\varepsilon],\xi\right)\cup\left(\ell^{-1}_i (0)\times (\mathbb{C}^{\ast} ,\pm\infty),\xi\right)\cup\left(\ell^{-1}_i [\varepsilon,+\infty),\xi\right).
	\]We then consider the following diagram
	 \[
		\begin{tikzcd}
			&\mathcal{W}\left(\ell^{-1}_i (-\infty,-\varepsilon],\xi\right)\arrow[r,"\Id"]\arrow[ld]&\mathcal{W}\left(\ell^{-1}_i (-\varepsilon,-\varepsilon],\xi\right)\arrow[rd]&\\ 
			\mathcal{W} (M(\mathbb{V}),\xi)&\mathcal{W}\left(\ell^{-1}_i [-\varepsilon ,\varepsilon],\xi\right)\arrow[r,"\Phi_s"]\arrow[l]&\mathcal{W}\left(\ell^{-1}_i (0),\xi\right)\arrow[r]&\mathcal{W} (M(\mathbb{V}_i),\xi)\\ 
			&\mathcal{W}\left(\ell^{-1}_i (\varepsilon,+\infty],\xi\right)\arrow[r,"\Id"]\arrow[lu]&\mathcal{W}\left(\ell^{-1}_i [\varepsilon, +\infty),\xi\right)\arrow[ru]&
		\end{tikzcd}.
		\]Since the left and right-hand diagrams are homotopy colimit diagrams and the functors commute with the colimit diagrams, this diagram induces a well-defined $A_{\infty}$-functor $\mathrm{Del}_i^s\colon\mathcal{W} (M(\mathbb{V}),\xi)\to\mathcal{W} (M(\mathbb{V}_i),\xi)$. Let us compute this cobordism functor explicitly. Recall from the construction of Viterbo restriction functor in \cite{ganatra2018sectorial} that if we consider the product $\left(T^{\ast} [-1,1]^\times \right)\times\mathbb{C}_{\rRe\geq 0}$ of the larger sector, then the inclusion $T^{\ast} [-1,1]^\times \hookrightarrow T^{\ast} [-1,1]$ induces a backward stopped inclusion 
\[\left(\left(T^{\ast} [-1,1]^{\times}\right)\times\mathbb{C}_{\rRe\geq 0}, T^{\ast} [-1,1]^{\times}\times\mathbb{C}_{\rRe =+\infty}\right)\hookrightarrow\left(T^{\ast} [-1,1]^{\times}\times\mathbb{C}_{\rRe\geq 0} ,T^{\ast} [-1,1]\times\mathbb{C}_{\rRe=+\infty}\right).\]
\noindent Note that the left-hand side is the stabilization of $T^{\ast} [-1,1]^{\times}$, so the wrapped Fukaya category is quasi-equivalent to that of $T^{\ast} [-1,1]^{\times}$, which is generated by the two purely imaginary rays as shown in Figure \ref{func:cotangent-bundle-interval-removing-origin}.

It is easier to keep track of the Viterbo restriction functor using relative cores: depending on the parity, the relative core of $T^{\ast} [-1,1]$ includes into the upper or lower part of the relative core of $T^{\ast} [-1,1]^{\times}$, and the Viterbo restriction functor would send the imaginary ray that does not intersect with the relative core of $T^{\ast} [-1,1]$ to $0$, while the other imaginary ray would get sent to the canonical generator of $T^{\ast} [-1,1]$ (see Example \ref{Prelim:generation-cotangent-fiber}). Together with the K\"uneth functor, we are able to compute the image of deletion functor on objects. For each standard Lagrangian $L$ in $\ell_i^{-1} (0)$, we write $L^+$ for the product of $L$ with positive imaginary ray, $L^-$ for product with negative imaginary ray, and $L^0$ for product with the cotangent fiber at $0$ of $T^{\ast} [-1,1]$.

\begin{lemma}
    $\Phi_s$ sends $L^s$ to $L^0$ and $L^{-s}$ to $0$.
\end{lemma}

To determine $\mathrm{Del}_i^s$, we need to keep track of the Lagrangians $L^0$ and $L^s$ after gluing. Recall from the previous discussions that the inclusion $\left(\ell_i^{-1} (0)\times T^{\ast} [-1,1]^{\times},\xi\right)\hookrightarrow (M(\mathbb{V}),\xi)$ sends the Lagrangian $L^s$ to the standard Lagrangian $\tilde{L}$ whose closure contains $L$ as one of its faces, and the Lagrangian $L^0$ is away from all crossing points of $M(\mathbb{V}_i)$. Therefore we know that $L^0$ gets sent to $0$ after gluing back the remaining pieces, and hence the deletion functor sends all the standard Lagrangians $\tilde{L}$ to $0$. See Figure \ref{func:deletion-functor} for an example in the case of two-dimensional pairs-of-pants.

\begin{figure}[ht]
	\centering
	\begin{tikzpicture}[scale=0.5]
		\fill [BrickRed!50] (0,0) -- (0,-5) -- (4,-5) -- (4,0);
		\draw [ultra thick] (-4,0) -- (4,0);
        \draw [ultra thick,dashed] (0,-5) -- (0,4);
        \draw [ultra thick] (-4,2) -- (0,-2); 
        \draw [ultra thick, black!50!BrickRed!50] (0,-2) -- (3,-5);
        \draw [->, red] (4,2) -- (3,4) node[above] {$\xi$};
		\end{tikzpicture}
		\caption{The Deletion Functor.\label{func:deletion-functor}}
\end{figure}
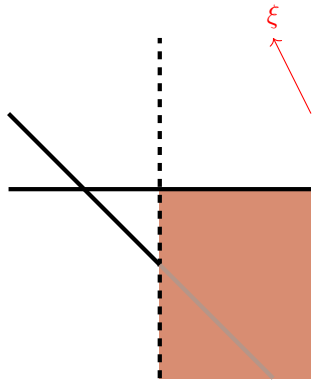

\begin{remark}
The restriction and deletion functors discussed here do not arise directly from the algebraic restriction and deletion operations outlined in Section \ref{sec:Delandres HypArr}. This is because they send standard Lagrangians to standard Lagrangians, whereas their algebraic counterparts do not do so in general. To describe the corresponding algebraic framework, one needs to generalize the restriction and deletion operators by imposing signs on each chamber separately, rather than uniformly for all chambers. This induces a well-defined morphism of graded algebra on the relevant convolution algebras $\widetilde{B}$. However, the natural pullback functor on the module categories still does not send standard objects to standard objects. Instead, one needs to construct a functor at the 0-categorical level by ensuring that the standard object is sent to the desired standard object. Then our generation and formality results will provide functors we want. We refrain from elaborating on this approach in this paper because it may appear somewhat unnatural from an algebraic viewpoint.
\end{remark}

%


\begin{thebibliography}{10}

\bibitem{abouzaid2011cotangent}
Mohammed Abouzaid.
\newblock A cotangent fibre generates the {F}ukaya category.
\newblock {\em Adv. Math.}, 228(2):894--939, 2011.

\bibitem{Ab-Au21}
Mohammed Abouzaid and Denis Auroux.
\newblock Homological mirror symmetry for hypersurfaces in
  $(\mathbb{C}^\ast)^n$.
\newblock {\em Geometry \& Topology}, 28(6):2825-2914, 2024

\bibitem{abouzaid2013homological}
Mohammed Abouzaid, Denis Auroux, Alexander Efimov, Ludmil Katzarkov, and Dmitri
  Orlov.
\newblock Homological mirror symmetry for punctured spheres.
\newblock {\em J. Amer. Math. Soc}, 26(4):1051--1083, 2013.

\bibitem{abouzaid2010open}
Mohammed Abouzaid and Paul Seidel.
\newblock An open string analogue of {V}iterbo functoriality.
\newblock {\em Geometry \& Topology}, 14(2):627--718, 2010.

\bibitem{mag}
Mina Aganagic.
\newblock Knot categorification from mirror symmetry, part {I}{I}: Lagrangians.
\newblock {\em arXiv:2105.06039.}, 2021.

\bibitem{Auroux2014Intro}
Denis Auroux. 
\newblock A Beginner's Introduction to Fukaya Categories.
\newblock {I}n: {B}ourgeois, {F}r{\'e}d{\'e}ric
and {C}olin, {V}incent
and {S}tipsicz, {A}ndr{\'a}s (eds) {\em{C}ontact and {S}ymplectic {T}opology. {B}olyai {S}ociety {M}athematical {S}tudies}, 26: 85--136, 2014.

\bibitem{Auroux-H}
Denis Auroux.
\newblock Fukaya categories and bordered {H}eegaard-{F}loer homology.
\newblock {\em Proceedings of the International Congress of Mathematicians},
  II:917--941, 2010.

\bibitem{auroux2017speculations}
Denis Auroux.
\newblock Speculations on homological mirror symmetry for hypersurfaces in
  $(\mathbb{C}^{\ast})^n$.
\newblock {\em Surveys in Differential Geometry}, 22:1--47, 2017.

\bibitem{BD2000}
Roger Bielawski and Andrew~S. Dancer.
\newblock The geometry and topology of toric hyperk\"{a}hler manifolds.
\newblock {\em Comm. Anal. Geom.}, 8(4):727--760, 2000.

\bibitem{BC1}
Paul Biran and Octav Cornea.
\newblock Lagrangian cobordism {I}.
\newblock {\em J. Amer. Math. Soc.}, 26(2):295--340, 2013.

\bibitem{BC2}
Paul Biran and Octav Cornea.
\newblock Lagrangian cobordism and {F}ukaya categories.
\newblock {\em Geom. Funct. Anal.}, 24(6):1731--1830, 2014.

\bibitem{bosshard2022lagrangian}
Valentin Bosshard.
\newblock Lagrangian cobordisms in {L}iouville manifolds.
\newblock {\em Journal of Topology and Analysis}, 1--55, 2022.

\bibitem{BLPPW}
Tom Braden, Anthony Licata, Christopher Phan, Nick Proudfoot, and Ben Webster.
\newblock Localization algebras and deformations of {K}oszul algebras.
\newblock {\em Selecta Math. (N.S.)}, 17(3):533--572, 2011.

\bibitem{BLPW2010}
Tom Braden, Anthony~M. Licata, Nick Proudfoot, and Ben Webster.
\newblock Gale duality and {K}oszul duality.
\newblock {\em Adv. Math.}, 225(4):2002--2049, 2010.

\bibitem{gc}
Guillem Cazassus.
\newblock Equivariant {L}agrangian {F}loer homology via cotangent bundles of
  {$EG_N$}.
\newblock {\em Journal of Topology}, 17(1), 2024.

\bibitem{chantraine2017geometric}Baptiste Chantraine, Georgios Dimitroglou Rizell, Paolo Ghiggini, and Roman Golovko. \newblock Geometric generation of the wrapped {F}ukaya category of {W}einstein manifolds and sectors. {\em Ann. Sci. Éc. Norm. Supér. (4)}. 57, 1-85, 2024.


\bibitem{cieliebak2012stein}
Kai Cieliebak and Yakov Eliashberg.
\newblock {\em From {S}tein to {W}einstein and back: symplectic geometry of
  affine complex manifolds}, {V}olume~59.
\newblock American Mathematical Soc., 2012.

\bibitem{df}
Aliakbar Daemi and Kenji Fukaya.
\newblock Atiyah-{F}loer conjecture: {A} formulation, a strategy to prove and
  generalizations.
\newblock {\em Modern geometry: a celebration of the work of Simon Donaldson,
  Proc. Sympos. Pure Math., 99, Amer. Math. Soc., Providence, RI}, pages
  23--57, 2018.

\bibitem{eliashberg2017weinstein}
Yakov Eliashberg.
\newblock Weinstein manifolds revisited.
\newblock {\em arXiv:1707.03442}, 2017.

\bibitem{Gammage-Shende}
Benjamin Gammage and Vivek Shende.
\newblock Mirror symmetry for very affine hypersurfaces.
\newblock {\em Acta Math.}, 229(2):287--346, 2022.

\bibitem{gammage2022functorial}
Benjamin Gammage and Maxim Jeffs.
\newblock Homological mirror symmetry for functors between {F}ukaya categories of very affine hypersurfaces.
\newblock {\em Journal of Topology}, 2024.


\bibitem{ganatra2020covariantly}
Sheel Ganatra, John Pardon, and Vivek Shende.
\newblock Covariantly functorial wrapped {F}loer theory on {L}iouville sectors.
\newblock {\em Publications math{\'e}matiques de l'IH{\'E}S}, 131(1):73--200, 2020.

\bibitem{ganatra2018sectorial}
Sheel Ganatra, John Pardon, and Vivek Shende.
\newblock Sectorial descent for wrapped {F}ukaya categories.
\newblock {\em J. Amer. Math. Soc.}, 37(2):499--635, 2024.

\bibitem{GPSmicrolocal}
Sheel Ganatra, John Pardon and Vivek Shende,
\newblock Microlocal Morse theory of wrapped Fukaya categories.
\newblock {\em Ann. Math.}, 199(3): 943--1042, 2024.

\bibitem{geiges}
Hansj\"{o}rg Geiges.
\newblock An introduction to contact topology. 
\newblock {\em Cambridge Studies in Advanced Mathematics}, vol. 109, Cambridge University Press, Cambridge, 2008.

\bibitem{hls}
Kristen Hendricks, Robert Lipshitz, and Sucharit Sarkar.
\newblock A simplicial construction of {$G$}-equivariant {F}loer homology.
\newblock {\em Proc. Lond. Math. Soc.}, 121(6)(3):1798--1866, 2020.

\bibitem{KLZ}
Yoosik Kim, Siu-Cheong Lau, and Xiao Zheng.
\newblock {$T$}-equivariant disc potential and {SYZ} mirror construction.
\newblock {\em Adv. Math.}, 430:Paper No. 109209, 70, 2023.

\bibitem{LLL}
Siu-Cheong Lau, Nai-Chung~Conan Leung, and Yan-Lung~Leon Li.
\newblock Equivariant {L}agrangian correspondence and a conjecture of
  {T}eleman.
\newblock {\em arXiv:2312.13926}, 2023.

\bibitem{LLM2020}
Aaron~D. Lauda, Anthony~M. Licata, and Manion Andrew.
\newblock From hypertoric geometry to bordered {F}loer homology via the $m=1$
  amplituhedron.
\newblock {\em Selecta Mathematica}, 30(43):1--59, 2024.

\bibitem{llm1}
Sukjoo Lee, Yin Li, and Cheuk~Yu Mak.
\newblock Hypertoric convolution algebras as {F}ukaya categories {I}.
\newblock {\em in preparation}, 2025.

\bibitem{lpa}
Yanki Lekili and James Pascaleff.
\newblock Floer cohomology of $\mathfrak{g}$-equivariant {L}agrangian branes.
\newblock {\em Compositio Mathematica}, 152:1071--1110, 2015.

\bibitem{lekili2020homological}
Yank{\i} Lekili and Alexander Polishchuk.
\newblock Homological mirror symmetry for higher-dimensional pairs of pants.
\newblock {\em Compositio Mathematica}, 156(7):1310--1347, 2020.

\bibitem{LS}
Yank{\i} Lekili and Ed~Segal.
\newblock Equivariant {F}ukaya categories at singular values.
\newblock {\em arXiv:2304.10969}, 2023.

\bibitem{inversematrix2000}
Tzon-Tzer Lu and Sheng-Hua Shiou.
\newblock Inverses of $2\times2$ block matrices.
\newblock {\em Computers and Mathematics with Applications}, 43:119--129, 2002.

\bibitem{mak2018dehn}
Cheuk~Yu Mak and Weiwei Wu.
\newblock Dehn twist exact sequences through {L}agrangian cobordism.
\newblock {\em Compositio Mathematica}, 154(12):2485--2533, 2018.

\bibitem{MSZ}
Michael McBreen, Vivek Shende, and Peng Zhou.
\newblock The {H}amiltonian reduction of hypertoric mirror symmetry.
\newblock {\em arXiv:2405.07955}, 2024.

\bibitem{mclean2012growth}
Mark McLean.
\newblock The growth rate of symplectic homology and affine varieties.
\newblock {\em Geom. Funct. Anal.}, 22:369--442, 2012.

\bibitem{ps3}
Paul Seidel.
\newblock {P}icard-{L}efschetz theory and dilating $\mathbb{C}^\ast$-actions.
\newblock {\em Journal of Topology}, 8(4):1167--1201, 2015.

\bibitem{ss}
Paul Seidel and Jake Solomon.
\newblock Symplectic cohomology and $q$-intersection numbers.
\newblock {\em Geom. Funct. Anal.}, 22:443--477, 2012.

\bibitem{Tanaka18}
Hiro Tanaka.
\newblock Surgery induces exact sequences in {L}agrangian cobordisms.
\newblock {\em arXiv:1805.07424}, 2018.

\bibitem{ct}
Constantin S. Teleman.
\newblock Gauge theory and mirror symmetry.
\newblock {\em arXiv:1404.2305}, 2014.

\bibitem{teschl2012ordinary}
Gerald Teschl.
\newblock {\em Ordinary differential equations and dynamical systems}, volume
  140.
\newblock American Mathematical Soc., 2012.

\bibitem{yx}
Yao Xiao.
\newblock Equivariant {L}agrangian {F}loer theory on compact toric manifolds.
\newblock {\em arXiv:2310.20202}, 2023.

\bibitem{ZS}
Zachery Zylvan.
\newblock On partially wrapped {F}ukaya categories.
\newblock {\em Journal of Topology}, 12(2):372--441, 2019.

\end{thebibliography}

{\small

\medskip
\noindent Sukjoo Lee\\
\noindent Center for Geometry and Physics, Institute for Basic Science (IBS), Pohang 37673, South Korea\\
\noindent {\it e-mail:} Sukjoo216@ibs.re.kr

\noindent 

\medskip
\noindent Yin Li\\
\noindent Department of Mathematics, Uppsala University, 753 10 Uppsala, Sweden \\
{\it e-mail:} yin.li@math.uu.se

\medskip
\noindent Si-Yang Liu\\
\noindent Department of Mathematics, Dornsife College of Letters, Arts and Sciences, University of Southern California, Los Angeles, United States\\ 
{\it e-mail:} liusiyan@usc.edu

\medskip
\noindent Cheuk Yu Mak\\
\noindent School of Mathematical and Physical Sciences, University of Sheffield, S10 2TN, United Kingdom\\
{\it e-mail:} c.mak@sheffield.ac.uk

}

\end{document}